\documentclass [leqno,11pt]{amsart}
\usepackage{amssymb,amsmath,latexsym,amsfonts,amsbsy,amsthm,mathtools,graphicx,color,subeqnarray,cases,bm}
\usepackage{float}
\usepackage{hyperref}
\usepackage{scalerel,stackengine}
\stackMath
\newcommand\reallywidehat[1]{%
\savestack{\tmpbox}{\stretchto{%
  \scaleto{%
    \scalerel*[\widthof{\ensuremath{#1}}]{\kern-.6pt\bigwedge\kern-.6pt}%
    {\rule[-\textheight/2]{1ex}{\textheight}}
  }{\textheight}%
}{0.5ex}}%
\stackon[1pt]{#1}{\tmpbox}%
}

\definecolor{myback}{RGB}{204,232,207}

\usepackage{cancel} 
\usepackage[dvipsnames]{xcolor}
\numberwithin{equation}{section}

\allowdisplaybreaks

\numberwithin{equation}{section}

\allowdisplaybreaks

\let\f=\frac

\let\na=\nabla

\let\pa=\partial

\def\eqdef{\buildrel\hbox{\footnotesize def}\over =}

\def\fc{{\mathfrak c}}

\newcommand{\beq}{\begin{equation}}
\newcommand{\eeq}{\end{equation}}
\newcommand{\ben}{\begin{eqnarray}}
\newcommand{\een}{\end{eqnarray}}
\newcommand{\beno}{\begin{eqnarray*}}
\newcommand{\eeno}{\end{eqnarray*}}

\newtheorem{theorem}{Theorem}[section]

\newtheorem{lemma}[theorem]{Lemma}
\newtheorem{proposition}[theorem]{Proposition}
\newtheorem{assumption}[theorem]{Assumption}
\newtheorem{corol}[theorem]{Corollary}
\newtheorem{remark}[theorem]{Remark}

\begin{document}
\title{Instability of shear flows with neutral embedded eigenvalues}

\author{Hui Li}
\address[H. Li]{School of Mathematics, Sichuan University, Chengdu 610064, P. R. China.}
\email{lihui92@scu.edu.cn}

\author{Siqi Ren}
\address[S. Ren]{School of Mathematics Sciences, Zhejiang University of Technology, Hangzhou 310032, P. R. China.}
\email{sirrenmath@zjut.edu.cn}

\author{Yuxi Wang}
\address[Y. Wang]{School of Mathematics, Sichuan University, Chengdu 610064, P. R. China.}
\email{wangyuxi@scu.edu.cn}

\author{Guoqing Zhang}
\address[G. Zhang]{School of Mathematics, Sichuan University, Chengdu 610064, P. R. China.}
\email{zhangguoqing2@stu.scu.edu.cn}

\begin{abstract}
We study the linear stability of a class of monotone shear flows. When the associated Rayleigh operator possesses a neutral embedded eigenvalue, we show that solutions of the linearized system may exhibit arbitrarily large growth in both the $L^\infty$ and $L^2$ norms. Moreover, when the embedded eigenvalue is multiple, we prove that the instability becomes stronger and explicitly construct solutions that grow linearly in time. This instability originates from the non-normality of the Rayleigh operator.
\end{abstract}

\maketitle

\section{Introduction}
We consider the two-dimensional incompressible Euler equation on $\mathbb{T}\times\mathbb{R}$
\begin{equation}\label{eq:Euler}
  \left\{
    \begin{array}{l}
    \pa_tV+V\cdot\na V+\na P=0,\\
     \na\cdot V=0,\\
    V(0,x,y)=V_{\text{in}}(x,y),
    \end{array}
  \right.
\end{equation}
where $V=(V_1,V_2)$ is the velocity field and $P$ is the pressure. Let $\Omega=\pa_xV_2-\pa_y V_1$ be the vorticity. The vorticity form of \eqref{eq:Euler} is
\begin{align}\label{eq:Euler1}
  \pa_t\Omega+V\cdot\na \Omega=0.
\end{align}

Each shear flow $V_s=\left(b(y),0\right)$ is a steady solution of the Euler equation. The linearized equation \eqref{eq:Euler1} around $\left(b(y),0\right)$ is
\begin{align}\label{eq-linEuler}
  \pa_t\omega+b\pa_x\omega-b''\pa_x\psi=0,\quad \psi=\Delta^{-1}\omega,
\end{align}
where $\omega$ and $\psi$ are the vorticity and stream function of the perturbation.

Taking the normal mode ansatz
\begin{align*}
  \psi(t,x,y)=\phi_k(y)e^{ik(x-\fc t)},
\end{align*}
where $k$ denotes the stream-wise wave number and $\fc=\fc_r+i\fc_i$ is the complex wave speed, we obtain from \eqref{eq-linEuler} the Rayleigh equation
\begin{align}\label{eq-Rayleigh}
  (b-\fc)\left(\pa_y^2-k^2\right)\phi_k-b''\phi_k=0,
 \end{align}
subject to the boundary condition $\phi_k(y)\to0$ as $y\to\pm\infty$. Consequently, for shear flows the linear stability problem for
\begin{align}\label{eq-linEuler-k}
  \pa_t\omega(t,y)+ik \mathcal R_{b,k}\omega(t,y) =0,\quad \omega(0,y)=\omega_{\text{in}}(y).
\end{align}
can be reduced to the spectral analysis of the Rayleigh operator
\begin{align*}
  \mathcal R_{b,k}=b \mathrm{Id}-b''(\pa_y^2-k^2)^{-1}.
\end{align*}

When the $\mathcal R_{b,k}$ admits an eigenvalue $\fc$ with $k \fc_i>0$, the linearized system \eqref{eq-linEuler-k} has exponentially growing solutions, and the shear flow is \emph{spectrally unstable}. If the spectrum consists only of continuous spectrum on the real axis, monotone shear flows are known to be linearly stable, and we refer to this case as \emph{spectrally stable}.

This paper focuses on an intermediate regime in which the Rayleigh operator admits embedded eigenvalues on the real axis. We refer to such shear flows as being in a \emph{neutral state}. We show that, despite the absence of exponential growth, neutral shear flows may still exhibit a form of instability induced by the non-normality of the Rayleigh operator.

\subsection{Backgrounds}
The study of shear flow stability dates back to the pioneering works of Rayleigh \cite{Rayleigh1880}, Kelvin \cite{Kelvin1887}, Reynolds \cite{reynolds1883}, Orr \cite{Orr1907}, and Sommerfeld \cite{Sommerfeld1908}. In his classical paper \cite{Rayleigh1880}, Rayleigh established a necessary condition for inviscid instability, now known as the Rayleigh inflection-point criterion, which states that an unstable shear flow must possess an inflection point. This criterion was later refined by Fj\o{}rtoft \cite{fjortoft1950application}. Moreover, Howard \cite{Howard1961} introduced the semicircle theorem, which provides a constraint on the possible location of eigenvalues in the complex plane.

For the Couette flow $b(y)=y$ on $\mathbb R$, the Rayleigh operator has no point spectrum, and its continuous spectrum occupies the entire real axis. Orr \cite{Orr1907} observed that, for modes with $k\neq 0$, the perturbation velocity decays at a polynomial rate in time (see \cite{Case1960,dikii1964} for the channel setting). This phenomenon is known as \emph{inviscid damping}, analogous to Landau damping in plasma physics.  A breakthrough result on nonlinear inviscid damping was obtained by Bedrossian and Masmoudi \cite{BM2015} for the Couette flow with Gevrey-$m$ perturbations ($1\le m<2$); see \cite{IonescuJia2020cmp} for the finite-channel setting, and \cite{LinZeng2011,deng2023long,DZ2021} for instability results.

For general monotone shear flows, Rosencrans and Sattinger \cite{RosSat1966} pointed out the stability of the continuous spectrum. Compared to the fundamental Couette flow, the associated Rayleigh operator contain a non-local term, which introduces significant mathematical difficulties in the analysis of general monotone flows. For spectrally stable monotone flows, Wei-Zhang-Zhao \cite{WeiZhangZhao2018} rigorously proved linear inviscid damping and showed that the decay rates of the velocity are the same as in the Couette case. More recently, Ionescu-Jia \cite{IJ2020} and Masmoudi-Zhao \cite{MasmoudiZhao2020} established nonlinear inviscid damping for spectrally stable monotone shear flows. Further  results on inviscid damping for monotone shear flows can be found in \cite{Zillinger2016,Zillinger2017,Jia2020siam,Jia2020arma,GNRS2020,WeiZhangZhu2020cmp,Ren2025}. For radially symmetric flows and non-monotone shear flows, there are also related stability results; we refer the reader to \cite{BCV2017,Zillinger2017jde,CZ2019,IonJia2022,RWWZ2023} and \cite{WeiZhangZhao2019,WeiZhangZhao2020,ionescu2024stability,BG2024,BCJ2026} respectively.

For shear flows in a neutral state, the presence of embedded eigenvalues makes the asymptotic behavior of the linearized system more complicated, and existing theoretical results remain limited. In \cite{stepin1996rayleigh}, monotone flows with embedded eigenvalues were studied, and a formal eigenfunction expansion associated with the continuous spectrum and the embedded eigenvalue of problem \eqref{eq-linEuler-k} was derived. The hyperbolic tangent flow $b(y)=\tanh y$ is a canonical example of a neutral-state flow. For the mode $k=1$, the corresponding Rayleigh operator possesses a unique embedded eigenvalue $\fc_*=0$. In \cite{RenZhang2025}, the second author and Zhang study the asymptotic stability  for $b(y)=\tanh y$, after removing the eigenspace associated with $\fc_*=0$ the remaining component of the solution (corresponding to the continuous spectrum) exhibits inviscid damping. However, due to the influence of the embedded eigenvalue, the spectral density function develops singularities, and consequently the decay rate of inviscid damping is slower than that for spectrally stable shear flows.

Neutral states and the spectrally unstable states are closely related. In \cite{lin2003instability}, Lin proves that for shear flows in the class $\mathcal K^+$, if the Rayleigh operator has an embedded eigenvalue with wave number $k=k_{max}$, then it necessarily has an unstable eigenvalue for all $k\in(0,k_{max})$ (See also related work in \cite{LinZeng2011,li2011resolution}). In \cite{sinambela2023transition}, Sinambela-Zhao constructed spectrally unstable shear flows in a neighborhood of shear flows in neutral states. More recently, the first author and Zhao \cite{LiZhao2025} show that, under viscous effect, a shear flow may evolve from spectrally stable to neutral, and subsequently to spectrally unstable state, thereby demonstrating the destabilizing effect of viscosity.

In all of the above works, neutral shear flows primarily serve as an intermediate step toward constructing or understanding spectrally unstable flows. In contrast, in the present paper we focus on the instability inherent to neutral shear flows themselves. Our main results are as follows.

\subsection{Main results} Let us now list some assumptions on the shear flow $b(y)$.

\begin{assumption}\label{assum}
Assume the shear flow $b(y)$ satisfies the following conditions:
  \begin{itemize}
\item (Monotone) There is $0<c_m<1$ such that for all $y\in\mathbb{R}$, it holds that $$0<c_m\leq b'(y)\leq c_m^{-1}.$$

\item (Regularity and integrability)  $b''(y)\in H^4(\mathbb R)$.
\item (Embedded eigenvalue) The Rayleigh operator $\mathcal R_{b,k}$ has a single embedded eigenvalue and has no other eigenvalues.
\end{itemize}
\end{assumption}

\begin{theorem}\label{thm}
  Let the background shear flow $b(y)$ satisfy Assumption \ref{assum}. For any $M>0$, there exist $\omega_{\text{in}}$ and $T$ such that
  \begin{align*}
    \left\|\omega_{\text{in}}(y)\right\|_{L^\infty(\mathbb R)}+\left\|\omega_{\text{in}}(y)\right\|_{L^2(\mathbb R)}=1,
  \end{align*}
  and the solution $\omega(t,y)$ to \eqref{eq-linEuler-k} with initial data $\omega_{\text{in}}$ satisfies
  \begin{align*}
    \left\|\omega(T,y)\right\|_{L^\infty(\mathbb R)}\ge M,\quad \left\|\omega(T,y)\right\|_{L^2(\mathbb R)}\ge M.
  \end{align*}
\end{theorem}
For a normal operator $L$ on a Hilbert space, the spectral theorem implies that if its spectrum is contained in the real axis, then the solution $\omega(t,y)=e^{iLt}\omega_{\text{in}}(y)$ does not exhibit growth in the norm of the Hilbert space. However, for general shear flows, the associated Rayleigh operator is non-normal, and it is this non-normality that gives rise to the instability proved in the theorem.

Under Assumption \ref{assum}, the embedded eigenvalue $\fc_*=\fc_{*,r}$ at most corresponding to one eigenfunction with velocity field in $H^1(\mathbb R)$, thus the geometric multiplicity of $\fc_*$ is $1$. We call $\fc_*$ a simple eigenvalue if its algebraic multiplicity is $1$  (see also \cite{stepin1996rayleigh}). Here the algebraic multiplicity means the order of the zero of the Wronskian defined in \eqref{eq-Wronskian-realline}. If the algebraic multiplicity of $\fc_*$ is greater than $1$, we call $\fc_*$ a multiple eigenvalue.

\begin{theorem}\label{thm2}
  Let the background shear flow $b(y)$ satisfy Assumption \ref{assum}.  If, in addition,  the embedded eigenvalue is a multiple eigenvalue, then there exists $\omega_{\text{in}}\in L^\infty(\mathbb R)\cap L^2(\mathbb R)$ and $C>0$ such that
  \begin{align*}
    \left\|\omega(t,y)\right\|_{L^\infty(\mathbb R)}\ge Ct,\quad   \left\|\omega(t,y)\right\|_{L^2(\mathbb R)}\ge Ct,\text{ for }\forall t\ge0.
  \end{align*}
\end{theorem}
\begin{remark}
 For monotone shear flows without embedded eigenvalues and eigenvalues, the instability mentioned above does not occur. In that case, it has been proved that the vorticity satisfies a uniform $L^2$ bound, see, for example, \cite{WeiZhangZhao2018,Jiahao2022,LiZhao2024}.  

 The instability in  Theorem \ref{thm} is exactly the classical instability, i.e., given $\epsilon_0>0$, for any $\delta\in(0,1)$, there exist $\tilde{\omega}_{\text{in}}$ and $T>0$ such that $\|\tilde{\omega}_{\text{in}}\|_{X}\leq \delta$,  $\|\tilde{\omega}(T)\|_{X}\ge \epsilon_0$. It suffices to consider $\delta=\epsilon_0/M$ and replace  $\tilde{\omega}_{\text{in}}=\delta\omega_{\text{in}}$, $\tilde{\omega}(t)=\delta\omega(t)$ in Theorem \ref{thm}. 

The instability in Theorem \ref{thm2} is stronger than the one in Theorem \ref{thm}, which means that the shear flow with multiple embedded eigenvalues is more unstable than the one with simple embedded eigenvalues.

In \cite{LiZhao2025}, the first author and Zhao show the existence of shear flows that satisfies Assumption \ref{assum}. By using similar techniques, we can show the existence of monotone shear flows that with a multiple embedded eigenvalue, see the appendix for details.
\end{remark}
\begin{remark}
  In this paper, we assume that the Rayleigh operator $\mathcal R_{b,k}$ has only a single embedded eigenvalue for simplicity. Using the same techniques, one can show that when $\mathcal R_{b,k}$ has finite embedded eigenvalues, the same type of instability still occurs.

  The techniques developed in this paper may also be applied to bounded monotone shear flows, such as the generalization of  $b(y)=\tanh y$, as well as to monotone flows in a finite channel.

  For the channel case $y\in[0,1]$, one may construct a smooth solution $\phi(y,\fc)$ of \eqref{eq-Rayleigh1} satisfying $\phi(0,\fc)=0$ and $\phi'(0,\fc)=1$. In this setting, the Wronskian can be regard as $W(\fc)=\phi(1,\fc)$ (see \cite{WeiZhangZhao2018,lin2003instability}). Therefore, if there exists $\fc_*\in \mathrm{Ran}(b)$ such that $\phi(1,\fc_*)=0$ and $\pa_{\fc}\phi(1,\fc_*)=0$, then $\fc_*$ is a multiple embedded eigenvalue. As noted in \cite{stepin1996rayleigh} there exists an associated function satisfying the boundary condition. Then, by the same technique used in the proof of Theorem \ref{thm2}, one can construct a solution that grows linearly in time. The same argument applies to non-monotone flows as well, provided that suitable definitions of $\phi(y,\fc)$ and the Wronskian are available.
\end{remark}

For the viscous fluid, the linear system system around the shear flow $b(y)$ is
\begin{align}\label{eq-linNS-k}
  \pa_t\omega(t,y)+\mathcal O_{b,k,\nu}\omega(t,y) =0,\quad \omega(0,y)=\omega_{\text{in}}(y),
\end{align}
where
\begin{align*}
  \mathcal O_{b,k,\nu}\omega=ik \mathcal R_{b,k}\omega-\nu \left(\pa_y^2-k^2\right)\omega
\end{align*}
is the Orr-Sommerfeld operator and $\nu$ is the viscous coefficient. If $\nu$ is small enough, the instability observed in Theorem \ref{thm} still holds.
\begin{corol}\label{cor}
  For background shear flow $b(y)$ satisfies the Assumption \ref{assum}. For any $M>0$, there exist $\omega_{\text{in}}$, $T$, and $\nu_0>0$ such that
  \begin{align*}
    \left\|\omega_{\text{in}}(y)\right\|_{L^2(\mathbb R)}=1,
  \end{align*}
  and the solution $\omega(t,y)$ of \eqref{eq-linNS-k} with $0\le\nu\le\nu_0$ from initial data $\omega_{\text{in}}$  satisfies
  \begin{align*}
    \left\|\omega(T,y)\right\|_{L^2(\mathbb R)}\ge M.
  \end{align*}
\end{corol}
\begin{remark}
  For monotone shear flow, it was proved in \cite{LiZhao2024} that if the Rayleigh operator has no eigenvalue and embedded eigenvalue, there will be some uniform enhanced dissipation
\begin{align*}
  \left\|\omega_{\neq}(t,y)\right\|_{L^2(\mathbb R)}\lesssim C e^{\frac{1}{C}\nu^\frac{1}{3}t},
\end{align*}
where $C$ is independent on $\nu$, and $\omega_{\neq}$ means the non-zero modes of the solution. The above Corollary shows that, for shear flow with embedded eigenvalue, this kind of uniform enhanced dissipation no longer holds.
\end{remark}

Without loss of generality, in the rest of this paper, we assume that $\fc_*=0$ is the embedded eigenvalue of $\mathcal R_{b,k}$ with $k=1$.

\subsection{Main ideas}
We next outline the main ideas of the proof and clarify the source of the instability. We begin with a toy model that provides an intuitive illustration of the growth induced by the non-normality of the operator.

\subsubsection{Toy model}
In \cite{trefethen1993}, Trefethen et al. introduced the following toy model (see also \cite{Chapman2002})
\begin{equation}\label{eq-toy}
    \frac{d}{dt}\left(\begin{array}{l}
      \tilde\psi(t)\\
      \tilde\phi(t)
    \end{array}\right)=A\cdot \left(\begin{array}{l}
      \tilde\psi(t)\\
      \tilde\phi(t)
    \end{array}\right),
\end{equation}
where
\begin{align}\label{eq-A1}
  A=\left(\begin{array}{cc}
      -2\nu&0\\
      1&-\nu
    \end{array}\right),
\end{align}
$\tilde\psi$ is stream-wise vortex, $\tilde\phi$ is stream-wise streak and $\nu$ is the viscosity, assumed to be small.

For the initial data $\left(\tilde\psi(0),\tilde\phi(0) \right)=(1,0)$, the solution is
\begin{align*}
  \tilde\psi(t)=e^{-2\nu t},\quad \tilde\phi(t)=\nu^{-1}\left(e^{-\nu t}-e^{-2\nu t}\right).
\end{align*}
At $t=\nu^{-1}$, we obtain $\tilde\phi(\nu^{-1})=\nu^{-1}\left(e^{-1}-e^{-2}\right)$, which becomes large as $\nu\to0$. Thus the system exhibits a transient growth.

The eigenvalues and eigenvectors of $A$ given in \eqref{eq-A1} are
\begin{align*}
  \lambda_1=-2\nu,\quad \mathbf{v}_1=(-\nu,1)^{\top},\\
  \lambda_2=-\nu,\quad \mathbf{v}_2=(0,1)^{\top},
\end{align*}
and the initial state decomposes as
\begin{align*}
  (1,0)^{\top}=-\nu^{-1}\mathbf{v}_1+\nu^{-1}\mathbf{v}_2.
\end{align*}
Hence a perturbation of unit size has $O(\nu^{-1})$ components along each eigendirection. The solution may therefore be written as
\begin{align}\label{eq-decomp}
  \left(\tilde\psi(t),\tilde\phi(t) \right)^\top=-\nu^{-1}e^{\lambda_1t}\mathbf{v}_1+\nu^{-1}e^{\lambda_2t}\mathbf{v}_2=-\nu^{-1}e^{-2\nu t}(-\nu,1)^{\top}+\nu^{-1}e^{-\nu t}(0,1)^{\top}.
\end{align}
Because $A$ is non-normal, $\mathbf{v}_1$ and $\mathbf{v}_2$ are not orthogonal; the initial near-cancellation of these components is destroyed by their different decay rates, leading to a transient growth.

If the entry $A_{11}$ in \eqref{eq-A1} is changed from $-2\nu$ to $-\nu$, the two eigenvalues coalesce. In this case, consider
\begin{align}\label{eq-A2}
  A=\left(\begin{array}{cc}
      -\nu&0\\
      1&-\nu
    \end{array}\right).
\end{align}
Since $A$ consists of a single Jordan block, it possesses only one eigenvalue and only one eigenvector,
\begin{align*}
  \lambda_1=-\nu,\quad \mathbf{v}_1=(0,1)^{\top}. \end{align*}
Because $A$ is defective, this eigenvalue $\lambda_1$ is multiple and there exists a generalized eigenvector
\begin{align*}
  \tilde {\mathbf{v}}_2=(1,0)^{\top},
\end{align*}
which satisfies
\begin{align}\label{eq-g-ev}
  \left(A-\lambda_1 \mathrm{Id}\right)\cdot \tilde {\mathbf{v}}_2=\mathbf{v}_1.
\end{align}
For this matrix $A$, the solution of \eqref{eq-toy} with initial data $\left(\tilde\psi(0),\tilde\phi(0) \right)=(1,0)$ is
\begin{align}\label{eq-decomp-mul}
  \left(\tilde\psi(t),\tilde\phi(t) \right)^\top=e^{\lambda_1t}\tilde {\mathbf{v}}_2+te^{\lambda_1t}\mathbf{v}_1=e^{-\nu t}(1,0)^{\top}+te^{-\nu t}(0,1)^{\top}.
\end{align}
Thus, the presence of the generalized eigenvector $\tilde {\mathbf{v}}_2$ continuously generates the eigenvector $\mathbf{v}_1$, producing a linear-in-time growth in the solution.

Trefethen et al. introduced the above toy model to illustrate the lift-up effect in three-dimensional viscous flows. In the present work, we study the instability of two-dimensional inviscid flows; although the physical settings differ, the underlying mechanism is similar: the growth is induced by the non-normality of the operator.

\subsubsection{Case with simple eigenvalue}
We now turn to the strategy of the proof. We first consider the case in which the embedded eigenvalue $\fc_*$ is simple.

The solution of \eqref{eq-linEuler-k} can be decomposed into two parts
\begin{align*}
  \omega(t,y)=\omega_{1}(t,y)+\omega_{2}(t,y).
\end{align*}

The first component
\begin{align*}
  \omega_{1}(t,y)=e^{-i\mathcal R_{b,k}t}\left(P \left(\omega_{\text{in}}\right)\omega_*\right)=P \left(\omega_{\text{in}}\right)\omega_*(y) e^{-i\fc_*k t}
\end{align*}
lies in the eigenspace. Here $\omega_*(y)$ denotes the eigenfunction associated with $\fc_*$, and $P \left(\omega_{\text{in}}\right)$ is the projection coefficient. This component is purely oscillatory in time.

The remaining component,
\begin{align*}
  \omega_{2}(t,y)=e^{-i\mathcal R_{b,k}t}\left(\omega_{\text{in}}-P \left(\omega_{\text{in}}\right)\omega_*\right),
\end{align*}
corresponds to the continuous spectrum. Owing to the mixing effect of the shear flow, the energy of $\omega_{2}$ is transferred from low to high frequencies in the $y$-direction.

Consequently, as in \eqref{eq-decomp}, the two components of the solution may exhibit near-cancellation at the initial time. However, since they evolve in different ways, this cancellation is gradually destroyed, which results in growth of the vorticity.

To establish the growth stated in Theorem \ref{thm}, we therefore need to verify two facts. First, for any $M>0$, one can find unit-size initial data $\omega_{\text{in}}$ such that $P \left(\omega_{\text{in}}\right)\ge M$. Second, the evolution of the continuous spectral component $\omega_{2}$ indeed leads to the loss of cancellation.

Without loss of generality, in the proof, we assume that $b(0)=0$, and for mode $k=1$, the embedded eigenvalue of $\mathcal R_{b,1}$ is $\fc_{*}=0$, then by the Rayleigh criteria, it holds that $b''(0)=0$. Then $\omega_{1}(t,y)$ is time-independent.

We first derive a suitable representation formula of the stream function. A key technical step here is to evaluate the limit involving the Wronskian
\begin{align*}
  \lim_{\pm\fc_i> 0,\left|\fc\right|\to 0}\frac{W(\fc)}{\fc}.
\end{align*}
The difficulty is that the Wronskian $W(\fc)$, defined in \eqref{eq-Wron}, is not analytic in $\fc$ (see \cite{RenZhang2025}). To overcome this, we introduce a modified Wronskian $\mathcal{W}(\fc)$ \eqref{eq-Wron-modi},  first used in \cite{LiMasmoudiZhao2022critical}, which is analytic in $\fc$. Using the analyticity, we derive the limit above and consequently obtain the projection coefficient $P(\omega_{\text{in}})$.

Formally, the projection coefficient can be viewed as
\begin{align*}
  P(\omega_{\text{in}})\sim \text{P.V.}\int_{\mathbb{R}}\frac{\omega_{\text{in}}(z)}{z}dz,
\end{align*}
which is a singular integration. We then take
\begin{align*}
 \omega_{\text{in}}(y)=\frac{1}{2} \chi_Z(y)\ge0,\quad \text{supp}\chi_Z(y)=[\frac{1}{Z},1],\quad \chi_Z(y)=1 \text{ on }[\frac{2}{Z},\frac{1}{2}].
\end{align*}
For $Z$ sufficiently large, one obtains
\begin{align*}
  |P(\omega_{\text{in}})|\gg\left\| \omega_{\text{in}}\right\|_{L^2},
\end{align*}
which shows that the projection coefficient can be made arbitrarily large while the initial data remains of unit size.

To quantify the mixing effect on $\omega_{2}(t,y)$, we use the duality argument developed in \cite{WeiZhangZhao2018} to get inviscid damping estimates for the stream function. We decompose the stream function into two parts
\begin{align*}
  \Psi(t,y)=\Psi_1(y)+\Psi_2(y),
\end{align*}
where $\Psi_1(y)=\Delta_k^{-1}\omega_{1}(y)$ is the eigenfunction component and $\Psi_2(t,y)=\Delta_k^{-1}\omega_{2}(t,y)$ corresponds to the continuous spectrum. By showing that
\begin{align*}
  \left\|\Psi_2(t,y)\right\|_{L^2}\lesssim \frac{1}{t},
\end{align*}
we obtain that, for sufficiently large $t$,
\begin{align*}
  \left\|\Psi(t,y)\right\|_{L^2}\approx \left\|\Psi_1(y)\right\|_{L^2}\approx P(\omega_{\text{in}}).
\end{align*}
Hence the growth of $\left\|\omega(t,y)\right\|_{L^2}$ follows directly.

\subsubsection{Case with multiple eigenvalue}
When the embedded eigenvalue $\fc_*$ has algebraic multiplicity greater than $1$, we establish the existence of the associated function $\eta$ that satisfies
\begin{align}\label{eq-assosi}
  \left(\mathcal R_{b,1}-\fc\right)\eta(y)=\omega_*(y).
\end{align}
The terminology associated function follows the usage in \cite{stepin1996rayleigh}.

Comparing \eqref{eq-assosi} with \eqref{eq-g-ev}, one can see that the associated function is the analogue of the generalized eigenvector corresponding to a defective matrix.

Therefore, in view of \eqref{eq-decomp-mul}, we find that, if the associated function $\eta$ is taken as the initial data, then
\begin{align*}
  \omega(t,y)=te^{-i\fc_* t}\omega_*(y)+ie^{-i\fc_* t}\eta(y)
\end{align*}
solves \eqref{eq-linEuler-k} and exhibits linear-in-time growth. Intuitively, the associated function $\eta$ keeps generating the eigenfunction $\omega_*$ in time, which results in the linear growth.

Next, we give the construction of the associated function under the assumption that $\fc*$ is a multiple embedded eigenvalue.

We introduce two family of solutions, denoted by $\varphi^+(y,\fc)$ and $\varphi^-(y,\fc)$ in \eqref{eq-phi-pm} (see also \cite{LiMasmoudiZhao2022critical}), which decay exponentially as $y\to +\infty$ and $y\to -\infty$, respectively. Since $\fc_*$ is an embedded eigenvalue, we have
\begin{align*}
  \varphi^+(y_{\fc},\fc_*)=\varphi^-(y_{\fc},\fc_*),\quad\pa_y\varphi^+(y,\fc_*)|_{y=y_{\fc}}=\pa_y\varphi^-(y,\fc_*)|_{y=y_{\fc}},
\end{align*}
where $y_{\fc}=b^{-1}\left(\fc_*\right)$. Then,
\begin{align*}
  \varphi_*(y)=
  \left\{
\begin{aligned}
&\varphi^+(y,\fc_*),\quad y\ge y_{\fc},\\
&\varphi^-(y,\fc_*),\quad y\le y_{\fc},
\end{aligned}
\right.
\end{align*}
is a solution of the homogeneous Rayleigh equation \eqref{eq-Rayleigh1} in $H^2(\mathbb R)$.

Moreover, we prove that if $\fc_*$ is a multiple eigenvalue, then
\begin{align*}
  \pa_\fc\varphi^+(y_{\fc},\fc)|_{\fc=\fc_*}=\pa_\fc\varphi^-(y_{\fc},\fc)|_{\fc=\fc_*},\quad\pa_y\pa_\fc\varphi^+(y,\fc)|_{\fc=\fc_*,y=y_{\fc}}=\pa_y\pa_\fc\varphi^-(y,\fc)|_{\fc=\fc_*,y=y_{\fc}},
\end{align*}
and consequently, for
\begin{align*}
  \pa_\fc\varphi_*(y)=
  \left\{
\begin{aligned}
&\pa_\fc\varphi^+(y,\fc_*),\quad y\ge y_{\fc},\\
&\pa_\fc\varphi^-(y,\fc_*),\quad y\le y_{\fc},
\end{aligned}
\right.
\end{align*}
we have that $\Delta_k\pa_\fc\varphi_*$ is a solution of \eqref{eq-assosi}.

Through the proof, we also show that the condition $b'''(y_{\fc})=0$ is necessary for an embedded eigenvalue to become a multiple eigenvalue. This is consistent with the analysis in \cite{stepin1996rayleigh} for the channel case.

Moreover, when the algebraic multiplicity of $\fc_*$ is $2$, the integrand in \eqref{eq-rep-Psi} exhibits a second-order singularity at $\fc_*$. This provides an alternative interpretation, via the Cauchy's differentiation formula, of the linear-in-time growth.

\subsection{Notations}
Through this paper, we will use the following notations.

We use $C$ to denote a positive big enough constant that may be different from line to line, and use $f\lesssim g$ ($f\approx g$) to denote
\begin{align*}
  f\le C g\quad (\frac{1}{C}g\le f\le C g).
\end{align*}
We use $f\ll g$ to indicate $f\le \frac{1}{C}g$ for a sufficient big $C$.

 Given a function $f(t,y)$, we denote its spatial derivation with respect to $y$ as
\begin{align*}
  f'(t,y)=\pa_yf(t,y),\qquad f''(t,y)=\pa_y^2f(t,y),\qquad f^{(n)}(t,y)=\pa_y^n f(t,y).
\end{align*}

\subsection{Organization}
The remainder of the paper is organized as follows. In Section 2, we give several useful properties of solutions to the homogeneous Rayleigh equation. In Section 3, we derive a representation formula for the stream function. In Section 4, we compute the projection coefficient and obtain inviscid-damping-type estimates for the solution corresponding to the continuous spectral component. In Section 5, we present the proofs of the main theorems.
\section{Homogeneous Rayleigh equation}

In this section, we study the homogeneous Rayleigh equation. Without loss of generality, we restrict our attention to the $k=1$ mode, which takes the form
\begin{align}\label{eq-Rayleigh1}
  \psi''(y)-\psi(y)-\frac{b''(y)}{b(y)-\fc}\psi(y)=0.
\end{align}
Here $b(y)$ is the shear flow that satisfies Assumption \ref{assum} and $\fc=\fc_r+i\fc_i\in\mathbb C$.

Throughout the rest of this paper, we use the shorthand notation
\begin{align*}
  \mathcal {R}=\mathcal {R}_{b,1}=b(y)-b''(y)\left(\partial_y^2-1\right)^{-1}.
\end{align*}

\subsection{Solving the homogeneous Rayleigh equation}\label{sec:Homogeneous Rayleigh equation}
In this subsection, we construct solutions to \eqref{eq-Rayleigh1}, and give the criterion for whether $\fc$ is an eigenvalue.  For the proof of the following results,  we refer the reader to \cite{LiMasmoudiZhao2022critical} and \cite{LiZhao2024}.
\begin{proposition}\label{prop-phi}
  There exist $0<\varepsilon \le 1$ and $0<C_4\le1$ which depend only on the upper and lower bounds of $b'(y)$, such that for $\fc\in\mathbb C$  with $0\leq |\fc_i|\le \varepsilon$, \eqref{eq-Rayleigh1} has a regular solution
  \begin{align*}
    \phi(y,\fc)=\big(b(y)-\fc\big)\phi_1(y,\fc_r)\phi_2(y,\fc),
  \end{align*}
  which satisfies $\phi(y_\fc,\fc)={-i\fc_i}$ and $\phi'(y_\fc,\fc)=b'(y_\fc)$. Here $y_\fc=b^{-1}(\fc_r)$, $\phi_1(y,\fc_r)$ is a real function that solves
  \begin{equation}\label{eq-phi1}
    \left\{
      \begin{array}{ll}
        \pa_y\left((b(y)-\fc_r)^2\phi_1'(y,\fc_r)\right)-\phi_1(y,\fc_r)(b(y)-\fc_r)^2=0,\\
        \phi_1(y_\fc,\fc_r)=1,\quad\phi_1'(y_\fc,\fc_r)=0,
      \end{array}
    \right.
  \end{equation}
and $\phi_2(y,\fc)$ solves
\begin{equation}\label{eq-phi2}
  \left\{
    \begin{array}{ll}
      \pa_y\Big((b(y)-\fc)^2\phi_1(y,\fc_r)^2\phi_2'(y,\fc)\Big)+\frac{2i{\fc_i} b'(y)(b(y)-\fc)}{b(y)-\fc_r}\phi_1(y,\fc_r)\phi_1'(y,\fc_r)\phi_2(y,\fc)=0,\\
      \phi_2(y_\fc,\fc)=1,\quad\phi_2'(y_\fc,\fc)=0.
    \end{array}
  \right.
\end{equation}
It holds that
  \begin{equation}\label{eq-phi1-est}
    \begin{aligned}
      \phi_1(y,\fc_r)\ge1,\quad Ce^{C_4(|y-y_\fc|)}\le\phi_1(y,\fc_r)\le e^{|y-y_\fc|},\quad\forall y\in\mathbb R,\\
      \frac{\phi_1'(y,\fc_r)}{y-y_\fc}\ge 0,\  |\phi_1'(y,\fc_r)|\le\phi_1(y,\fc_r),\quad \left|\frac{\phi_1'(y,\fc_r)}{\phi_1(y,\fc_r)}\right|\le \left|y-y_{\fc}\right|,\quad\forall y\in\mathbb R,\\
     |\phi_1(y,\fc_r)-1|\le C|y-y_\fc|^2\text{ for } |y-y_\fc|\le1,\  |\phi_1'(y,\fc_r)|\ge C_4\phi_1(y,\fc_r)\text{ for } |y-y_\fc|\ge1,\\
     \frac{\phi_1(y,\fc_r)}{\phi_1(y',\fc_r)}\le Ce^{-C_4(y-y')}\text{ for }y'\le y\le y_\fc,\quad \frac{\phi_1(y,\fc_r)}{\phi_1(y',\fc_r)}\le Ce^{-C_4(y'-y)}\text{ for }y'\ge y\ge y_\fc,
    \end{aligned}
  \end{equation}
    \begin{equation}\label{eq-phi2-est}
    \begin{aligned}
      \left|\phi_2(y,\fc)-1\right|\le C\min(|\fc_i|,|\fc_i||y-y_\fc|,|y-y_\fc|^2),\ \forall y\in\mathbb R,\\
       \left|\phi_2'(y,\fc)\right|\le C\min(|\fc_i|, |y-y_\fc|),\quad \|\phi_2''(y,\fc)\|_{L^\infty_y}\le C,\ \forall y\in\mathbb R,
    \end{aligned}
  \end{equation}
  and
    \begin{align}\label{eq-phi-est}
    C^{-1}(|y-y_\fc|+|\fc_i|)e^{C_4|y-y_\fc|}\le|\phi(y,\fc)|\le C(|y-y_\fc|+|\fc_i|)e^{|y-y_\fc|}\text{ for }y\in\mathbb R,
  \end{align}
  where $C>0$ is a constant independent of $\fc$.
\end{proposition}

It is easy to check that
\beno
\phi(y,\fc)\quad &\text{and}&\quad \varphi^{+}(y,\fc),\\
\phi(y,\fc)\quad &\text{and}&\quad \varphi^{-}(y,\fc),
\eeno
where
\begin{align}\label{eq-phi-pm}
  \varphi^\pm(y,\fc)=&\phi(y,\fc)\int^y_{\pm\infty}\frac{1}{\phi(y',\fc)^2}dy',
\end{align}
are two fundamental sets of solutions of \eqref{eq-Rayleigh1}.
Therefore, if $\psi(y,\fc)\in H^2_{loc}$ is a solution of \eqref{eq-Rayleigh1}, then $\psi(y,\fc)$ has the following form
\begin{align}\label{eq-psi-eigenfunction}
  \psi(y,\fc)=a_1^-\phi(y,\fc)+a_2^-\varphi^-(y,\fc)=a_1^+\phi(y,\fc)+a_2^+\varphi^+(y,\fc),
\end{align}
where $a_1^\pm$, $a_2^\pm$ are constants. Note that by Proposition \ref{prop-phi},  $\varphi^\pm$ are well defined for $0<|\fc_i|\le \varepsilon$.

We give a criterion for whether $\fc$ is an eigenvalue of $\mathcal {R}$.
\begin{lemma}\label{lem-eigen}
  Let
  \begin{align}\label{eq-Wron}
  W(\fc)=&\int^{+\infty}_{-\infty}\frac{1}{\phi(y',\fc)^2}dy'.
  \end{align}
Then, a number $\fc\in\mathbb C$  with $0<|\fc_i|\le \varepsilon$ is an eigenvalue of $\mathcal {R}$ if and only if $W(\fc)=0$.
\end{lemma}
The criterion function $W(\fc)$ can be extended to $\fc\in\mathbb R$.
\begin{lemma}\label{lem-lim-eigen}
  It holds that
  \begin{align}\label{eq-Wronskian-realline}
    \lim_{\fc_i\to0\pm}W(\fc)=\mathcal J_1(\fc_r)\mp i\mathcal J_2(\fc_r),
  \end{align}
  where
  \begin{align*}
    \mathcal J_1(\fc_r)&=\frac{1}{b'(y_\fc)}\Pi_1(\fc_r)+\Pi_2(\fc_r),\quad\mathcal J_2(\fc_r)=\pi\frac{b''(y_\fc)}{{b'}(y_\fc)^3}.
  \end{align*}
  and
  \begin{align*}
  \Pi_1(\fc_r)= \text{P.V.}\int  \frac{b'(y_\fc)-b'(y)}{(b(y)-\fc_r)^2} dy,\quad \Pi_2(\fc)= \int_{-\infty}^{+\infty} \frac{1}{(b(y)-\fc_r)^2}\left(\frac{1}{\phi_1(y,\fc_r)^2}-1\right) dy.
\end{align*}
\end{lemma}

Recall that Ran $b(y)=\mathbb R$ is the continuous spectrum of $\mathcal {R}$. If $\fc\in\mathbb R$ is an eigenvalue of $\mathcal {R}$, we call $\fc$ an embedded eigenvalue.
\begin{lemma}\label{lem-iff-emb}
A number $\fc_r\in \mathbb R$ is an embedded eigenvalue of $\mathcal {R}$ if and only if
\begin{align*}
  \mathcal J_1(\fc_r)=\mathcal J_2(\fc_r)=0.
\end{align*}
\end{lemma}

Without loss of generality, in the rest of this paper, we assume that $\fc_*=0$ is the embedded eigenvalue of $\mathcal {R}$, therefore, we have
\begin{align*}
  \mathcal J_1(0)=\mathcal J_2(0)=0.
\end{align*}

\subsection{Regularity estimates for the homogeneous solutions}
In this subsection, we give some useful estimates for the homogeneous solutions that constructed in Proposition \ref{prop-phi}.

We start with some auxiliary estimates. To this end, we introduce the good derivative
\begin{align*}
  \partial_G=\partial_{\fc_r}+\frac{\partial_y}{b'(y_{\fc})}.
\end{align*}
 Taking any order of $\partial_G$ of $b(y)-\fc_r$ does not change the fact that it vanishes to first order at $y=y_{\fc}$. More precisely, if \(b(y)\) satisfies Assumption \ref{assum} and \(\fc_r\in\mathbb{R}\), a direct calculation shows that for \(n=0,1,2,3\),
\begin{equation}\label{eq-pargkb1}
    \begin{aligned}
      \left|\partial_G^n\left(\frac{1}{\left(b(y)-\fc_r\right)^2}\right)\right|
      &\lesssim \frac{1}{\left|y-y_{\fc}\right|^2},
  \end{aligned}
  \end{equation}
  and
  \begin{equation}\label{eq-pargkb2}
    \begin{aligned}
      \left|\partial_{G}^{n}\partial_{y}\left(\frac{\partial_{G}\left(\left(b(y)-\fc_r\right)^2\right)}{\left(b(y)-\fc_r\right)^2}\right)\right|\le C.
    \end{aligned}
  \end{equation}

Next, we derive regularity estimates for $\phi_1(y,\fc_r)$ that constructed in Proposition \ref{prop-phi}.
\begin{lemma}\label{lem-paryphi}
  For $n=1,2,3$, it holds that
  \begin{equation}\label{eq-paryphi1}
    \begin{aligned}
      \left|\frac{\partial_{G}^n\phi_1(y,\fc_r)}{\phi_1(y,\fc_r)}\right|\lesssim \left|y-y_{\fc}\right|^3 ,\quad
      \left|\partial_y\left(\frac{\partial_{G}^n\phi_1(y,\fc_r)}{\phi_1(y,\fc_r)}\right)\right|\lesssim \left|y-y_{\fc}\right|^2,\text{ for }\left|y-y_{\fc}\right|\le 1,
  \end{aligned}
  \end{equation}
  and
\begin{equation}\label{eq-paryphi2}
  \begin{aligned}
      \left|\frac{\partial_{G}^n\phi_1(y,\fc_r)}{\phi_1(y,\fc_r)^2}\right|\lesssim \min\left\{e^{-C_4\left|y-y_{\fc}\right|},\left|y-y_{\fc}\right|^3\right\}, \text{ for }y\in\mathbb R.
  \end{aligned}
\end{equation}
\end{lemma}
\begin{proof}
  Recall that $\phi_1(y,\fc_r)$ solves \eqref{eq-phi1}. We have
  \begin{align*}
    \left( (b(y)-\fc_r)^2\phi_1'(y,\fc_r)\right)'=(b(y)-\fc_r)^2\phi_1(y,\fc_r).
  \end{align*}
Applying $\partial_G^n$ to the above equation, after tedious calculation, one can verify that the following holds  
\begin{align}
   \notag &-\left( \left(b(y)-\fc_r\right)^2 \left(\frac{\pa_G^n\phi_1(y,\fc_r)}{\phi_1(y,\fc_r)}\right)'\phi_1(y,\fc_r)^2\right)'\\
    &\qquad=\left(b(y)-\fc_r\right)^2\phi_1(y,\fc_r)\cdot\sum_{i=0}^{n-1} C_n^i\cdot\partial_G^{n-i-1}\partial_y\left(\frac{\partial_G\left(\left(b(y)-\fc_r\right)^2\right)}{\left(b(y)-\fc_r\right)^2}\right)
    \partial_G^i\partial_y\phi_1(y,\fc_r),
\label{eq-partialgnphi}\end{align}
  where $C_n^i=\frac{n!}{(n-i)!i!}$ is the binomial coefficient.

 Thanks to $\phi_1(y_{\fc},\fc_r)=1$, $\phi_1'(y_{\fc},\fc_r)=0$, we have $\partial_G^n\phi_1(y_{\fc},\fc_r)=\partial_G^n\partial_y\phi_1(y_{\fc},\fc_r)=0$ for $\fc_r\in\mathbb{R}$. For $n=1$,  we have by \eqref{eq-partialgnphi} that
\begin{equation}\label{eq-pargphin1}
  \begin{aligned}
    &\frac{\pa_G\phi_1(y,\fc_r)}{\phi_1(y,\fc_r)}
    =-\int_{y_{\fc}}^y \frac{\int_{y_{\fc}}^{y'} \left( \frac{\pa_{G''}\left(\left(b(y'')-\fc_r\right)^2\right)}{\left(b(y'')-\fc_r\right)^2}\right)'
    \left(b(y'')-\fc_r\right)^2\phi_1'(y'',\fc_r)\phi_1(y'',\fc_r) d y''}{\left(b(y')-\fc_r\right)^2\phi_1(y',\fc_r)^2} d y', \end{aligned}
\end{equation}
which together with \eqref{eq-pargkb2}, \eqref{eq-pargphin1} and \eqref{eq-phi1-est} gives
\begin{align*}
    &\left|\frac{\pa_G\phi_1(y,\fc_r)}{\phi_1(y,\fc_r)}\right|
    \lesssim \left|\int_{y_{\fc}}^y \frac{\int_{y_{\fc}}^{y'}
    \left(b(y'')-\fc_r\right)^2\phi_1'(y'',\fc_r)\phi_1(y'',\fc_r) d y''}{\left(b(y')-\fc_r\right)^2\phi_1(y',\fc_r)^2} d y'\right|\\
    &\lesssim\left|\int_{y_{\fc}}^y \frac{\int_{y_{\fc}}^{y'}
    \phi_1'(y'',\fc_r)\phi_1(y'',\fc_r) d y''}{\phi_1(y',\fc_r)^2} d y'\right|
    \lesssim\left|\int_{y_{\fc}}^y \frac{
    \phi_1(y',\fc_r)^2-1}{\phi_1(y',\fc_r)^2} d y'\right|\lesssim \min\left\{\left|y-y_{\fc}\right|,\left|y-y_{\fc}\right|^3\right\}.
\end{align*}
And it follows directly that \begin{align*}
    \left|\partial_y\left(\frac{\pa_G\phi_1(y,\fc_r)}{\phi_1(y,\fc_r)}\right)\right|=
    \left|\frac{\int_{y_{\fc}}^{y}
    \left(b(y')-\fc_r\right)^2\phi_1'(y',\fc_r)\phi_1(y',\fc_r) d y'}{\left(b(y)-\fc_r\right)^2\phi_1(y,\fc_r)^2} \right|
    \lesssim \min\left\{1,\left|y-y_{\fc}\right|^2\right\}.
\end{align*}
For $n>1$, using \eqref{eq-partialgnphi} and the fact that
\begin{align*}
  \partial_G^i\partial_y\phi_1(y,\fc_r)=\partial_y\phi_1(y,\fc_r)\cdot \frac{\partial_G^i\phi_1(y,\fc_r)}{\phi_1(y,\fc_r)}+\phi_1(y,\fc_r)\partial_y\left(\frac{\partial_G^i\phi_1(y,\fc_r)}{\phi_1(y,\fc_r)}\right),\ \text{for} \ 1\le i\le n,
\end{align*}
we write
\begin{align*}
    &\frac{\pa_G^n\phi_1(y,\fc_r)}{\phi_1(y,\fc_r)}=-\int_{y_{\fc}}^{y}\frac{1}{\left(b(y')-\fc_r\right)^2\phi_1(y',\fc_r)^2}
    \int_{y_{\fc}}^{y'}\left(b(y'')-\fc_r\right)^2\phi_1(y'',\fc_r)\cdot\\
    &\qquad\qquad \sum_{i=0}^{n-1} C_n^i\cdot\partial_{G''}^{n-i-1}\partial_{y''}\left(\frac{\partial_{G''}\left(\left(b(y'')-\fc_r\right)^2\right)}{\left(b(y'')-\fc_r\right)^2}\right)
    \partial_{G''}^i\partial_{y''}\phi_1(y'',\fc_r)dy'' d y'\\
    &\quad=-\int_{y_{\fc}}^{y}\frac{1}{\left(b(y')-\fc_r\right)^2\phi_1(y',\fc_r)^2}
    \int_{y_{\fc}}^{y'}\left(b(y'')-\fc_r\right)^2\phi_1(y'',\fc_r)\cdot\\
    &\quad\qquad\qquad  \partial_{G''}^{n-1}\partial_{y''}\left(\frac{\partial_{G''}\left(\left(b(y'')-\fc_r\right)^2\right)}{\left(b(y'')-\fc_r\right)^2}\right)
    \phi_1'(y'',\fc_r)dy'' d y'\\
    &\quad\quad -\int_{y_{\fc}}^{y}\frac{1}{\left(b(y')-\fc_r\right)^2\phi_1(y',\fc_r)^2}
    \int_{y_{\fc}}^{y'}\left(b(y'')-\fc_r\right)^2\phi_1(y'',\fc_r)\cdot\\
    &\quad\qquad\qquad \sum_{i=1}^{n-1} C_n^i\cdot\partial_{G''}^{n-i-1}\partial_{y''}\left(\frac{\partial_{G''}\left(\left(b(y'')-\fc_r\right)^2\right)}{\left(b(y'')-\fc_r\right)^2}\right)
    \phi_1(y'',\fc_r)\partial_y\left( \frac{\partial_{G''}^i\phi_1(y'',\fc_r)}{\phi_1(y'',\fc_r)}\right)dy'' d y'\\
    &\quad\quad -\int_{y_{\fc}}^{y}\frac{1}{\left(b(y')-\fc_r\right)^2\phi_1(y',\fc_r)^2}
    \int_{y_{\fc}}^{y'}\left(b(y'')-\fc_r\right)^2\phi_1(y'',\fc_r)\cdot\\
    &\quad\qquad\qquad \sum_{i=1}^{n-1} C_n^i\cdot\partial_{G''}^{n-i-1}\partial_{y''}\left(\frac{\partial_{G''}\left(\left(b(y'')-\fc_r\right)^2\right)}{\left(b(y'')-\fc_r\right)^2}\right)
    \phi_1'(y'',\fc_r)\frac{\partial_{G''}^i\phi_1(y'',\fc_r)}{\phi_1(y'',\fc_r)}dy'' d y'\\
    &\quad=I_1+I_2+I_3.
\end{align*}

For $I_1$, similar to the case $n=1$, applying \eqref{eq-pargkb2}, we have
\begin{equation}\label{eq-i1}
  \begin{aligned}
      \left|I_1\right|&\lesssim \left|\int_{y_{\fc}}^{y}\frac{\phi_1(y',\fc_r)^2-1}{\phi_1(y',\fc_r)^2} d y'\right|\lesssim \min\left\{\left|y-y_{\fc}\right|,\left|y-y_{\fc}\right|^3\right\}.
  \end{aligned}
\end{equation}

For $I_2$ and $I_3$,  it follows for $\left|y-y_{\fc}\right|\le 1$ and $1\le i\le n-1$ that \begin{align*}
    \left|\frac{\partial_{G}^i\phi_1(y,\fc_r)}{\phi_1(y,\fc_r)}\right|\lesssim \left|y-y_{\fc}\right|^3 ,\quad
    \left|\partial_y\left(\frac{\partial_{G}^i\phi_1(y,\fc_r)}{\phi_1(y,\fc_r)}\right)\right|\lesssim \left|y-y_{\fc}\right|^2.
\end{align*}
Then we have for $\left|y-y_{\fc}\right|\le 1$ and $n=1,2,3$ that
\begin{equation}\label{eq-i2}
  \begin{aligned}
      \left|I_2\right|\lesssim \left|\int_{y_{\fc}}^{y}
      \int_{y_{\fc}}^{y'} \left|y''-y_{\fc}\right|^2 d y'' d y'\right|
      \lesssim \left|y-y_{\fc}\right|^3.
  \end{aligned}
\end{equation}
Integrating by parts also gives
\begin{equation}\label{eq-i3}
  \begin{aligned}
      &\left|I_3\right|\lesssim \left|\int_{y_{\fc}}^{y} \frac{1}{\phi_1(y',\fc_r)^2}
      \int_{y_{\fc}}^{y'}\partial_{y''}\left(\phi_1(y'',\fc_r)^2-1\right) \cdot\left|y''-y_{\fc}\right|^3 d y'' d y'\right|\\
      &\lesssim \left|\int_{y_{\fc}}^{y} \frac{\phi_1(y',\fc_r)^2-1}{\phi_1(y',\fc_r)^2}
      \cdot \left|y'-y_{\fc}\right|^3 d y'+\int_{y_{\fc}}^{y} \frac{1}{\phi_1(y',\fc_r)^2}
      \int_{y_{\fc}}^{y'}\left(\phi_1(y'',\fc_r)^2-1\right) \cdot\left|y''-y_{\fc}\right|^2 d y'' d y'\right|\\
      &\lesssim \left|y-y_{\fc}\right|^3.
  \end{aligned}
\end{equation}
Then combining  \eqref{eq-i1}-\eqref{eq-i3},  we have for $\left|y-y_{\fc}\right|\le 1$ and $n=1,2,3$,
\begin{align*}
    \left|\frac{\partial_{G}^n\phi_1(y,\fc_r)}{\phi_1(y,\fc_r)}\right|\lesssim \left|y-y_{\fc}\right|^3 ,\quad
    \left|\partial_y\left(\frac{\partial_{G}^n\phi_1(y,\fc_r)}{\phi_1(y,\fc_r)}\right)\right|\lesssim \left|y-y_{\fc}\right|^2.
\end{align*}
Thus, the desired estimate \eqref{eq-paryphi1} follows.

Regarding \eqref{eq-paryphi2}, the previous arguments already imply that $\frac{\partial_{G}^n\phi_1(y,\fc_r)}{\phi_1(y,\fc_r)}$ grows at most polynomially.
Moreover, by \eqref{eq-phi1-est}, we have $\phi_1(y,\fc_r)\ge Ce^{C_4|y-y_{\fc}|}$, which together yields the desired estimate $ \left|\frac{\partial_{G}^n\phi_1(y,\fc_r)}{\phi_1(y,\fc_r)^2}\right|\lesssim \min\left\{e^{-C_4|y-y_{\fc}|}, \left|y-y_{\fc}\right|^3\right\}$. \end{proof}

\subsection{Properties of $\mathcal{J}_1(c_r)$ and $\mathcal{J}_2(c_r)$}
 In this subsection, we provide some properties of $\mathcal{J}_1(c_r)$ and $\mathcal{J}_2(c_r)$, which are defined as 
\begin{align*}
  \mathcal{J}_2(\fc_r) =\pi \frac{b''(y_\fc)}{{b'}(y_\fc)^3},\quad \mathcal J_1(\fc_r)=\frac{1}{b'(y_\fc)}\Pi_1(\fc_r)+\Pi_2(\fc_r),
\end{align*}
with
\begin{align*}
  \Pi_1(\fc_r)= \text{P.V.}\int  \frac{b'(y_\fc)-b'(y)}{(b(y)-\fc_r)^2} dy,\quad \Pi_2(\fc)= \int_{-\infty}^{+\infty} \frac{1}{(b(y)-\fc_r)^2}\left(\frac{1}{\phi_1(y,\fc_r)^2}-1\right) dy.
\end{align*}
These properties play a key role in the derivation of the representation formula for the stream function.

We first establish regularity estimates.
\begin{proposition}\label{prop-j1}
  Let $\fc_r\in \mathbb{R}$. It holds that
  \begin{align*}
    \left|\partial_{\fc_r}^n \mathcal{J}_1(\fc_r) \right| + \left|\partial_{\fc_r}^n \mathcal{J}_2(\fc_r) \right|\le C,\quad n=0,1,2,3.
  \end{align*}
\end{proposition}
\begin{proof}
  For $\mathcal J_2(\fc_r)$, a direct calculation yields
  \begin{align*}
    \left|\partial_{\fc_r}^n \mathcal{J}_2(\fc_r) \right|&=\pi\left|\partial_{\fc_r}^n\left(\frac{b''(y_\fc)}{{b'}(y_\fc)^3}\right)\right|
    \lesssim \left\|b''(y)\right\|_{H^{n+1}}.
  \end{align*}
 For $\mathcal J_1(\fc_r)$,
  by the Hilbert transform, we write
  \begin{align}
   \notag &\frac{1}{b'(y_{\fc})}\Pi_1(\fc_r)=\frac{1}{b'(y_{\fc})}\text{P.V.}\int_{\mathbb{R}} \frac{b'(y_{\fc})-b'(y)}{(b(y)-\fc_r)^2}dy\\
    &\quad=\frac{1}{b'(y_{\fc})}\text{P.V.}\int_{\mathbb{R}}\frac{1}{v-\fc_r}\pa_v\left(\frac{b'(y_\fc)-b'(b^{-1}(v))}{b'(b^{-1}(v))} \right)dv=-\mathcal H\left(\pa_v^2(b^{-1})(v)\right)(\fc_r).  \label{Hilbert}
\end{align}
By the property of Hilbert transform, for $g\in H^1(\mathbb{R})$, we have
\begin{align*}
    \left|\mathcal H(g)(y)\right|\le \left\|\widehat{\mathcal H(g)}\left(\xi\right)\right\|_{L^1}\le \left\|\widehat{g}\left(\xi\right)\right\|_{L^1}\lesssim \|g\|_{H^1},
\quad \text{and}\quad  \frac{d}{dy}\mathcal H(g)(y)=\mathcal H(g')(y).
\end{align*}
Then we have for $n=0,1,2,3$,
\begin{align*}
    \left|\partial_{\fc_r}^n\left(\frac{1}{b'(y_{\fc})}\Pi_1(\fc_r)\right)\right|=\left|\partial_{\fc_r}^{n}\mathcal H\left(\pa_v^2(b^{-1})(v)\right)(\fc_r)\right|
    \lesssim \left\|\pa_v^{n+2}(b^{-1})(v)\right\|_{H^{1}}\lesssim \left\|b''(y)\right\|_{H^{n+1}}.
\end{align*}
For $\Pi_2$, directly applying \eqref{eq-phi1-est}, we have
\begin{align*}
    \left|\Pi_2(\fc_r)\right|&\le \int_{-\infty}^{+\infty}\frac{1}{\left|b(y)-\fc_r\right|^2}
    \left|\frac{1}{\phi_1(y,\fc_r)^2}-1\right|dy\lesssim \int_{-\infty}^{+\infty}\frac{\left|y-y_{\fc}\right|^2}{\left|b(y)-\fc_r\right|^2}\cdot\frac{\phi_1(y,\fc_r)+1}{\phi_1(y,\fc_r)^2}dy\le C.
\end{align*}
Recall that $\phi_1(y,\fc_r)$ satisfies \eqref{eq-phi1} and $b''(y)\in H^4$ from Assumption \ref{assum}, we can deduce  for $n=1,2,3$ that  $\phi_1^{(n)}(y,\fc_r)$ is continuous and  the following limit exists
\begin{align*}
    \lim_{y\to y_{\fc}}\partial_y^n\left(\frac{1}{\left(b(y)-\fc_r\right)^2}\left(\frac{1}{\phi_1(y,\fc_r)^2}-1\right) \right)<\infty.
  \end{align*} Then the integral below is well-defined, and the fact $b(\infty)=\infty$ implies that
\begin{equation}\label{eq-inte-welldef}
  \begin{aligned}
      \int_{-\infty}^{+\infty}\partial_y^n
      \left(\frac{1}{\left(b(y)-\fc_r\right)^2}
      \left(\frac{1}{\phi_1(y,\fc_r)^2}-1\right)\right)dy=0.
  \end{aligned}
\end{equation}
Therefore, noticing $\partial_G=\partial_{\fc_r}+\frac{\partial_y}{b'(y_{\fc})}$, we have for $n=1,2,3$ \begin{align*}
    \partial_{\fc_r}^n\Pi_2(\fc_r) &=
    \int_{-\infty}^{+\infty}
   \partial_G^n\left(\frac{1}{\left(b(y)-
    \fc_r\right)^2}\left(\frac{1}{\phi_1(y,\fc_r)^2}-1\right)\right)dy.
\end{align*}
 Applying \eqref{eq-pargkb1} and Lemma \ref{lem-paryphi}, we have
\begin{equation}
  \begin{aligned}
    \notag  &\left|\partial_G\left(\frac{1}{\left(b(y)-\fc_r\right)^2}\left(\frac{1}{\phi_1(y,\fc_r)^2}-1\right)\right)\right|\\
      &\quad =\left|\partial_G\left(\frac{1}{\left(b(y)-\fc_r\right)^2}\right)\left(\frac{1}{\phi_1(y,\fc_r)^2}-1\right)
       -\frac{4}{\left(b(y)-\fc_r\right)^2\phi_1(y,\fc_r)^2}\frac{\partial_G\phi_1(y,\fc_r)}{\phi_1(y,\fc_r)}\right|\\
       &\quad\lesssim \min\left\{1,\frac{1}{\left(y-y_\fc\right)^2}\right\} +e^{-C_4\left|y-y_{\fc}\right|}.
  \end{aligned}
\end{equation}
And the same process  gives the same bounds for $\left|\partial_G\left(\frac{1}{\left(b(y)-\fc_r\right)^n}\left(\frac{1}{\phi_1(y,\fc_r)^2}-1\right)\right)\right|$, $n=2,3$.
Combining the above bounds with \eqref{eq-inte-welldef} gives $\left|\partial_{\fc_r}^n\Pi_2(\fc_r)\right|\leq C $, $n=1,2,3$.
\end{proof}

Next, we give a lower bound estimate for $\mathcal{J}_1(\fc_r)^2 + \mathcal{J}_2(\fc_r)^2$.
\begin{proposition}\label{prop-lowerbound}
  For $b(y)$ satisfies Assumption \ref{assum} and  $\partial_{\fc_r}\mathcal{J}_1(0)+ i\partial_{\fc_r}\mathcal{J}_2(0)\neq 0$, there exists a constant $C_l>0$ such that
  \begin{align*}
    \mathcal{J}_1(c_r)^2 + \mathcal{J}_2(c_r)^2\ge C_l,\quad \fc_r\in \mathbb{R}\setminus(-1,1),
  \end{align*}
  and
  \begin{align*}
    \frac{\mathcal{J}_1(c_r)^2 + \mathcal{J}_2(c_r)^2}{\fc_r^2}\ge C_l,\quad \fc_r\in (-2,2).
  \end{align*}
\end{proposition}

\begin{proof}
  First, we claim that $|\Pi_2(c_r)|$ admits an uniform lower bound $A>0$, which is independent of $c_r$. Following \eqref{eq-phi1-est}, we have
  \begin{align*}
      \left|\Pi_2(c_r)\right|\gtrsim \int_{-\infty}^{+\infty}\frac{1}{\left(y-y_{\fc}\right)^2}\frac{\left(y-y_{\fc}\right)^2\left(\phi_1(y,c_r)+1\right)}{\phi_1(y,c_r)^2}dy\ge A>0.
  \end{align*}
  Next, we show that $\Pi_1(c_r)$ becomes arbitrarily small for large $|c_r|$.
  To this end, we use \eqref{Hilbert} and the boundedness of the Hilbert transform in $L^2$ to get  $\frac{1}{b'(y_{\fc})}\Pi_1(\fc_r)\in H^{4}(\mathbb{R})$. Therefore, we can choose a sufficiently large $R > 0$ such that $\left| \frac{1}{b'(y_{\fc})} \Pi_1(\fc_r) \right| \le \frac{A}{2}$ for all $|\fc_r| > R$, which together with the lower bound of $\left|\Pi_2(\fc_r) \right|$ gives $\mathcal{J}_1(\fc_r)=\frac{1}{b'(y_{\fc})} \Pi_1(\fc_r)+ \Pi_2(\fc_r)\ge \frac{A}{2}$ for $|\fc_r|>R$.
On the other hand, by Assumption \ref{assum}, $\mathcal{R} $ has a single embedded eigenvalue at $0$. Lemma~\ref{lem-iff-emb} then implies that   $\mathcal{J}_1(\fc_r) + i\mathcal{J}_2(\fc_r) \neq 0$
on the compact set $\{ \fc_r \in \mathbb{R} \mid 1 \le |\fc_r| \le R\}$.
This shows that $\mathcal{J}_1(\fc_r)^2 + \mathcal{J}_2(\fc_r)^2$ attains a uniform positive lower bound for $1\leq|\fc_r|\leq R$.
To conclude, $\mathcal{J}_1(\fc_r)^2 + \mathcal{J}_2(\fc_r)^2$ is uniformly bounded below on $\mathbb{R} \setminus (-1, 1)$.

Next, we consider the case when $\fc_r \in (-2, 2)$. By L'Hospital's rule, we have
\begin{align*}
  \lim_{\fc_r\to 0}\frac{\mathcal{J}_1(\fc_r)^2 + \mathcal{J}_2(\fc_r)^2}{\fc_r^2}=\left(\partial_{\fc_r}\mathcal{J}_1(0)\right)^2+ \left(\partial_{\fc_r}\mathcal{J}_2(0)\right)^2.
\end{align*}
By our assumption, we have $\partial_{\fc_r}\mathcal{J}_1(0)+ i\partial_{\fc_r}\mathcal{J}_2(0)\neq 0$, then
there exists a constant $\delta>0$ such that
\begin{align*}
  \frac{\mathcal{J}_1(\fc_r)^2 + \mathcal{J}_2(\fc_r)^2}{\fc_r^2}\ge \frac{1}{2}\left(\left(\partial_{\fc_r}\mathcal{J}_1(0)\right)^2+ \left(\partial_{\fc_r}\mathcal{J}_2(0)\right)^2\right)\quad\text{for}\quad \fc_r\in (-\delta,\delta).
\end{align*}
We still use $\mathcal{J}_1(c_r) + i\mathcal{J}_2(c_r) \neq 0$ on the compact set $\{ \fc_r \in \mathbb{R} \mid \delta \le |\fc_r| \le 2 \}$ to obtain that $\frac{\mathcal{J}_1(\fc_r)^2 + \mathcal{J}_2(\fc_r)^2}{\fc_r^2}$ attains an uniform positive lower bound for $\delta \le |\fc_r| \le 2$.
Finally, $\frac{\mathcal{J}_1(\fc_r)^2 + \mathcal{J}_2(\fc_r)^2}{\fc_r^2}$ is uniformly bounded below on $(-2, 2)$.
\end{proof}

\section{Representation formula of the stream function}

In this section, we derive the representation formula of the stream function $\Psi(t,y)$ that solves the following linear system
\begin{align}\label{eq: LinearEuler-Psi}
  \left\{
\begin{aligned}
&\pa_t \omega(t,y)+i\mathcal {R}\omega(t,y)=0,\\
&(\pa_y^2-1)\Psi(t,y)=\omega(t,y),\\
&\omega|_{t=0}(y)=\omega_{\text{in}}(y)=(\pa_y^2-1)\Psi_{\text{in}}(y).
\end{aligned}
\right.
\end{align}

By the Laplace transform, for some $c^\diamond>0$, we have the following representation formula:
\begin{equation}\label{eq-rep-Psi}
  \begin{aligned}
    \Psi(t,y)=&\lim_{T\to \infty}\frac{1}{2\pi i}\int_{-T}^T\big[e^{-i(\fc_r-ic^\diamond) t}(\fc_r-ic^\diamond-\mathcal{L})^{-1}\Psi_{\text{in}}\\
    &\qquad\qquad\qquad\qquad\qquad\qquad-e^{-i(\fc_r+ic^\diamond) t}(\fc_r+ic^\diamond-\mathcal{L})^{-1}\Psi_{\text{in}}\big]d\fc_r,
  \end{aligned}
\end{equation}
where $\mathcal L=(\pa_y^2-1)^{-1}\mathcal{R}(\pa_y^2-1)$. A rigorous deduction of \eqref{eq-rep-Psi} can be found in \cite{LiMasmoudiZhao2022critical}.

Let $(\fc-\mathcal{L})^{-1}\Psi_{\text{in}}=i\Phi(y,\fc)$, then $\Phi(y,\fc)$ solves the inhomogeneous Rayleigh equation:
\begin{align}\label{eq-Phi-ori}
  \pa_y^2\Phi(y,\fc)-\Phi(y,\fc)-\frac{b''(y)}{b(y)-\fc}\Phi(y,\fc)=i\frac{\omega_{\text{in}}(y)}{b(y)-\fc}.
\end{align}
Recall that $\phi(y,\fc)$ given in Proposition \ref{prop-phi} solves the homogeneous Rayleigh equation \eqref{eq-Rayleigh1}, it is easy to check that
\begin{align}\label{eq-Phi}
  \pa_y\Big(\phi^2(y,\fc)\pa_y\Big(\frac{\Phi(y,\fc)}{\phi(y,\fc)}\Big)\Big)=i\omega_{\text{in}}(y)\phi_1(y,\fc).
\end{align}
Here we briefly write $\phi_1(y,\fc)=\phi_1(y,\fc_r)\phi_2(y,\fc)$. It is clear that $\phi_1(y,\fc)=\phi_1(y,\fc_r)$ for $\fc_i=0$. For $\fc_i\neq 0$, there is a unique solution $\Phi(y,\fc)$ of \eqref{eq-Phi} which decays at infinity. Then $\Phi(y,\fc)$ can be written as follows:{{}
\begin{align}\label{eq-Phi-exp}
\Phi(y,\fc)=&\Phi_{i,l}(y,\fc)-\mu(\omega_{\text{in}},\fc)\Phi_{h,l}(y,\fc)=\Phi_{i,r}(y,\fc)-\mu(\omega_{\text{in}},\fc)\Phi_{h,r}(y,\fc),
\end{align}
where
\begin{align*}
  \Phi_{i,l}(y,\fc)=&i\phi(y,\fc)\int_{-\infty}^y\frac{\int_{y_{\fc}}^{y'}\omega_{\text{in}}(y'')\phi_1(y'',\fc)dy''}{\phi(y',\fc)^2}dy',\ \Phi_{h,l}(y,\fc)=i\phi(y,\fc)\int_{-\infty}^y\frac{1}{\phi(y',\fc)^2}dy',\\
  \Phi_{i,r}(y,\fc)=&i\phi(y,\fc)\int_{+\infty}^y\frac{\int_{y_{\fc}}^{y'}\omega_{\text{in}}(y'')\phi_1(y'',\fc)dy''}{\phi(y',\fc)^2}dy',\  \Phi_{h,r}(y,\fc)=i\phi(y,\fc)\int_{+\infty}^y\frac{1}{\phi(y',\fc)^2}dy',
\end{align*}
and
\begin{align*}
  \mu(\omega_{\text{in}},\fc)=\frac{1}{\int_{-\infty}^{+\infty}\frac{1}{\phi(y',\fc)^2}dy'}\int_{-\infty}^{+\infty}\frac{\int_{y_{\fc}}^{y'}\omega_{\text{in}}(y'')\phi_1(y'',\fc)dy''}{\phi(y',\fc)^2}dy'.
\end{align*}}

Let
\begin{align}\label{def-jstar}
  J^*(\omega_{\text{in}},\fc)=\int_{-\infty}^{+\infty}\frac{\int_{y_{\fc}}^{y'}\omega_{\text{in}}(y'')\phi_1(y'',\fc)dy''}{\phi(y',\fc)^2}dy'.
\end{align}
Recalling the definition of the Wronskian, we write
\begin{align*}
  \mu(\omega_{\text{in}},\fc)=\frac{J^*(\omega_{\text{in}},\fc)}{W(\fc)}.
\end{align*}
Since $\fc=0$ is an embedded eigenvalue, it follows from Lemma \ref{lem-lim-eigen} and Lemma \ref{lem-iff-emb} that $\lim_{\fc\to0}W(\fc)=0$. Consequently, as $\fc\to0$, the coefficient $\mu(\omega_{\emph{in}},\fc)$ develops a singularity. This behavior is quite different to the case without embedded eigenvalue studied in \cite{LiMasmoudiZhao2022critical}.

Therefore, in order to derive the representation formula, it is necessary to analyze the limit $\lim_{\pm\fc_i> 0,\left|\fc\right|\to 0} \fc\mu(\omega_{\text{in}},\fc)$. To this end, we need to investigate both
\begin{align*}
  \lim_{\pm\fc_i> 0,\left|\fc\right|\to 0} \frac{W(\fc)}{\fc}\text{ and }\lim_{\fc_i\to 0\pm} J^*(\omega_{\text{in}},\fc).
\end{align*}

\subsection{The limiting absorption principle}
In this subsection, we derive the limit of $\fc\Phi(y,\fc)$ as $\fc\to0$. To do so, we first study the limit about $W(\fc)$ and $J^*(\omega_{\text{in}},\fc)$.
\subsubsection{An important limit of the Wronskian}
Now, we analyze the limit
\begin{align}\label{eq-lim-Wron}
  \lim_{\pm\fc_i> 0,\left|\fc\right|\to 0}\frac{W(\fc)}{\fc},
\end{align}
which is a key step in deriving the projection operator. Recall that
\begin{align*}
  W(\fc)=\int_{\mathbb{R}}\f{1}{\phi^2(y,\fc)}dy=\det\left(\begin{matrix}\varphi^{-}(y,\fc)&\varphi^{+}(y,\fc)\\
  \pa_y\varphi^{-}(y,\fc)&\pa_y\varphi^{+}(y,\fc)\end{matrix}\right).
\end{align*}
The main difficulty in evaluating \eqref{eq-lim-Wron} arises from the fact that $W(\fc)$ is not analytic with respect to $\fc$. To overcome this difficulty, we introduce a new Wronskian.

\begin{lemma}\label{lem-wronskian}
  There exists $Y>0$, such that for all $\fc$ with $0 < |\fc_i| \le \varepsilon$, we have $\varphi^{+}(Y,\fc) \neq 0$ and $\varphi^{-}(-Y,\fc) \neq 0$. Define
  \begin{align}\label{eq-Wron-modi}
    f(\fc) \eqdef \frac{1}{\varphi^{+}(Y,\fc)\varphi^{-}(-Y,\fc)},\quad \mathcal{W}(\fc)=\mathcal{W}_r(\fc)+i\mathcal{W}_i(\fc) \eqdef f(\fc)W(\fc).
  \end{align}
Then $\mathcal{W}(\fc)$ is analytic in the domain $0 < |\fc_i| \le \varepsilon$, where $\varepsilon$ is given in Lemma \ref{lem-lim-eigen}.
 \end{lemma}
For the proof of this lemma, we refer the readers to Remark B.1 of \cite{LiMasmoudiZhao2022critical}. We also recall that
\[
\varphi^{-}(y,\fc)=\phi(y,\fc)\int_{-\infty}^y\frac{1}{\phi(y',\fc)^2}\,dy'
\]
can be extended continuously in $\fc$ to the real axis for \(y<y_\fc=b^{-1}(\fc_r)\). For \(|\fc_r|<2\varepsilon<b(Y)\)(with \(\varepsilon>0\)  small) the values \(\varphi^{+}(Y,\fc_r)\) and \(\varphi^{-}(-Y,\fc_r)\) are well-defined.
Consequently, \(f(\fc)\) itself admits a continuous extension to the real axis whenever \(|\fc_r|<2\varepsilon<b(Y)\). Moreover, $\pa_{\fc_r}f(\fc_r)$ is also continuous for $|\fc_r|<2\varepsilon$. Now we are in a position to calculate the limit.
\begin{lemma}
  For the Wronskian $W(\fc)$ of the homogeneous Rayleigh equation \eqref{eq-Rayleigh1} satisfying $\lim_{\fc_i\to0\pm}W(\fc)=\mathcal J_1(\fc_r)\mp i\mathcal J_2(\fc_r)$, we have
  \begin{align*}
    \lim_{\pm\fc_i> 0,\left|\fc\right|\to 0}\frac{W(\fc)}{\fc}
    &= \partial_{\fc_r}\mathcal{J}_1(0)\mp i\partial_{\fc_r}\mathcal{J}_2(0).
\end{align*}
\end{lemma}
\begin{proof}We only present the proof for the case $\fc_i>0$, since the case $\fc_i<0$ can be treated in the same way. 

We continuously extend the $\mathcal{W}$ in \eqref{eq-Wron-modi} onto $\fc\in\mathbb{R}$ that:
 \begin{align} \label{extend:W}
  \mathcal{W}(\fc)=
  \left\{
\begin{aligned}
&\;\;\;\; \mathcal{W}(\fc),\quad \quad \quad \quad \quad \quad \quad \quad \quad \quad \quad \quad \quad \quad \fc_i>0,\\
&\lim_{\fc_i\to 0+}\mathcal{W}(\fc)=f(\fc_r)\left(\mathcal{J}_1(\fc_r)- i\mathcal{J}_2(\fc_r)\right),\quad
\fc_i=0. \end{aligned}
\right.
\end{align}
We begin by showing the partial derivative with respect to $\fc_r$ for this extended function is also continuous:
\begin{equation}\label{eq-limtoaxis}
  \begin{aligned}
    \partial_{\fc_r}\mathcal{W}(\fc_r)=\partial_{\fc_r} \Big(f(\fc_r)\left(\mathcal{J}_1(\fc_r)- i\mathcal{J}_2(\fc_r)\right)\Big)=\lim_{\fc_i\to 0+}\partial_{\fc_r}\mathcal{W}(\fc_r+i\fc_i),
    \quad \text{for}\quad \left|\fc_r\right|<\varepsilon.
  \end{aligned}
\end{equation}
Noticing  $\mathcal{W}(\fc)$ is analytic in $\left\{\fc\left|\fc_r\in\mathbb{R}, 0<\fc_i<\varepsilon\right.\right\}$ and the analyticity is a localized property. Since $f(\fc)$ is sufficiently regular in the  region $\left\{\fc\left|\fc_r\in(-2\varepsilon, 2\varepsilon),\; 0\le\fc_i<\varepsilon\right.\right\}$, we decompose the extended function \eqref{extend:W}  in this region as
$\mathcal{W}=u_0(\fc_r, \fc_i)+u_1(\fc_r, \fc_i)$. 

Here $u_0$ is defined as the harmonic extension of the boundary data via Poisson's formula:
 \begin{equation}\label{eq-lap1}
  \begin{aligned}
      \left\{\begin{array}{ll}
          \partial_{\fc_r}^2 u_0+\partial_{\fc_i}^2 u_0=0,\quad \fc_r \in\mathbb{R}, \fc_i\in (0,\varepsilon),\\
          u_0(\fc_r,0)=\chi_\varepsilon(\fc_r)
          f(\fc_r)\left(\mathcal{J}_1(\fc_r)- i\mathcal{J}_2(\fc_r)\right),\quad \fc_r\in\mathbb{R}, \end{array}\right.
  \end{aligned}
\end{equation}
where $\chi_\varepsilon(\fc_r)$ is a smooth cut-off function such that $\chi_\varepsilon(\fc_r)=1$ for $\left|\fc_r\right|\le \varepsilon$ and $\chi_\varepsilon(\fc_r)=0$ for $\left|\fc_r\right|\ge 2\varepsilon$. We then define $u_1=\mathcal{W}-u_0$, which satisfies
\begin{align}\label{eq-lap2}
  \left\{\begin{array}{ll}
          \partial_{\fc_r}^2 u_1+\partial_{\fc_i}^2 u_1=0,\quad \fc_r\in\mathbb{R}, \fc_i\in (0,\varepsilon),\\
          u_1(\fc_r,0)=\left(1-\chi_\varepsilon(\fc_r)\right)f(\fc_r)\left(\mathcal{J}_1(\fc_r)- i\mathcal{J}_2(\fc_r)\right),\quad \fc_r \in\mathbb{R}.
      \end{array}\right.
\end{align}

 For \eqref{eq-lap1}, by Proposition \ref{prop-j1} and Lemma \ref{lem-wronskian}, we have
\begin{align*}
   \chi_\varepsilon(\fc_r)
          f(\fc_r)\left(\mathcal{J}_1(\fc_r)- i\mathcal{J}_2(\fc_r)\right)\in L^\infty(\mathbb{R})\cap C^1(\mathbb{R}).
 \end{align*} 
By standard properties of Poisson's formula, this implies the continuity of the tangential derivative, namely,
\begin{equation*} \begin{aligned}
    \lim_{\fc_i\to 0+}\partial_{\fc_r}u_0(\fc_r,\fc_i)&=\partial_{\fc_r} \Big(\chi_\varepsilon(\fc_r)f(\fc_r)\left(\mathcal{J}_1(\fc_r)- i\mathcal{J}_2(\fc_r)\right)\Big)\\
    &=  \partial_{\fc_r} \Big(f(\fc_r)\left(\mathcal{J}_1(\fc_r)- i\mathcal{J}_2(\fc_r)\right)\Big)
    \quad \text{for}\quad \left|\fc_r\right|<\varepsilon.
  \end{aligned}
\end{equation*} 

 For \eqref{eq-lap2}, since the boundary value of $u_1$ is zero for $\fc_r\in(-\varepsilon,\varepsilon)$, we have by the Schwarz reflection principle that
$u_1(\fc_r,\fc_i)$ can be extended to be harmonic in the  disk $\left\{\fc\left|\left|\fc\right|<\varepsilon\right.\right\}$(See p. 95 of \cite{krantz1999handbook} for the details).
  Therefore, the following continuity of the derivative holds  \begin{equation*} \begin{aligned}
    \lim_{\fc_i\to 0+}\partial_{\fc_r}u_1(\fc_r,\fc_i)&=\partial_{\fc_r} \Big(\big(1-\chi_\varepsilon(\fc_r)\big)f(\fc_r)\left(\mathcal{J}_1(\fc_r)- i\mathcal{J}_2(\fc_r)\right)\Big)=0 \quad\text{for}\quad \left|\fc_r\right|<\varepsilon.
  \end{aligned}
\end{equation*}
  The above limits gives \eqref{eq-limtoaxis}.  To proceed, we notice the analyticity of $\mathcal{W}(\fc)$ gives the Cauchy-Riemann equations on $\left\{\fc\left|\fc_r\in(-\varepsilon, \varepsilon),\; 0<\fc_i<\varepsilon\right.\right\}$
  \begin{align*}
      \partial_{\fc_i}\mathcal{W}(\fc)=i\partial_{\fc_r}\mathcal{W}(\fc). \end{align*}
      This together with \eqref{eq-limtoaxis} gives the continuity for the $\pa_{\fc_i}$ derivative of the extended continuous function \eqref{extend:W}:
      \begin{align}  \label{eq-limtoaxis-ci}
   \notag \partial_{\fc_i}\mathcal{W}(\fc_r)&=\lim_{\fc_i\to 0+}\partial_{\fc_i}\mathcal{W}(\fc_r+i\fc_i)\\
   &=i \lim_{\fc_i\to 0+}\partial_{\fc_r}\mathcal{W}(\fc_r+i\fc_i)=i\partial_{\fc_r} \Big(f(\fc_r)\left(\mathcal{J}_1(\fc_r)- i\mathcal{J}_2(\fc_r)\right)\Big),
    \;\; \text{for}\;\; \left|\fc_r\right|<\varepsilon.
  \end{align}
For $\fc=0$, set
\begin{align*}
  \alpha \eqdef \partial_{\fc_r}\mathcal{W}_r(0)=\partial_{\fc_i}\mathcal{W}_i(0),\quad
  \beta \eqdef -\partial_{\fc_i}\mathcal{W}_r(0)=\partial_{\fc_r}\mathcal{W}_i(0).
\end{align*}
Recall that $\mathcal{W}(0)=0$, by the continuity of the derivatives \eqref{eq-limtoaxis} and \eqref{eq-limtoaxis-ci}, we can write
\begin{align*}
  \mathcal{W}_r(\fc)=\alpha \fc_r-\beta\fc_i +o\left(\sqrt{\fc_r^2+\fc_i^2}\right),\
  \mathcal{W}_i(\fc)=\beta \fc_r+\alpha\fc_i +o\left(\sqrt{\fc_r^2+\fc_i^2}\right),\ \text{as}\ \fc_r\to 0,\fc_i\to 0+.
\end{align*}
Then  it follows
\begin{align}\label{eq-existence}
  \lim_{\fc_i> 0,\left|\fc\right|\to 0}\frac{\mathcal{W}(\fc)}{\fc}
  &=\lim_{\fc_i> 0,\left|\fc\right|\to 0}\frac{\left(\alpha+i\beta\right)\left(\fc_r+i\fc_i\right)+o\left(\sqrt{\fc_r^2+\fc_i^2}\right)}{\fc_r+i\fc_i}
  =\alpha+i\beta=\partial_{\fc_r}\mathcal{W}(0).
\end{align}
Therefore, we obtain the desired limit as
\begin{align*}
    &\lim_{\fc_i> 0,\left|\fc\right|\to 0}\frac{W(\fc)}{\fc}
    =\lim_{\fc_i> 0,\left|\fc\right|\to 0}\frac{1}{f(\fc)}\frac{\mathcal{W}(\fc)}{\fc}
    =\frac{1}{f(0)}\lim_{\fc_i> 0,\left|\fc\right|\to 0}\frac{\mathcal{W}(\fc)}{\fc} =\frac{1}{f(0)}\partial_{\fc_r}\mathcal{W}(0),
    \end{align*}
 which together with the first identity in \eqref{eq-limtoaxis} and $\mathcal{J}_1(0)-i\mathcal{J}_2(0)=0$, $\left|\pa_{\fc_r}f(0)\right|<\infty$ gives
\begin{align}\label{eq-parb}
 \lim_{\fc_i> 0,\left|\fc\right|\to 0}\frac{W(\fc)}{\fc}=\partial_{\fc_r}\mathcal{J}_1(0)- i\partial_{\fc_r}\mathcal{J}_2(0). \end{align}
We finish the proof for the limit when $\fc_i> 0$.
\end{proof}

\subsubsection{Limit of $J^*(\omega_{\text{in}},\fc)$}
We next study the limit of $\mathcal{J}^*(\omega_{\text{in}},\fc)
=\int_{-\infty}^{+\infty}\frac{\int_{y_{\fc}}^{y}
\omega_{\text{in}}(y')\phi_1(y,\fc)dy'}{\phi(y,\fc)^2}dy$ as $\fc$ approaches the real axis.
\begin{lemma}\label{lem-j_3}
  Let $\omega_{\text{in}}(y)\in H^2(\mathbb{R})$. It holds that
  \begin{align*}
    \lim_{\fc_i\to 0\pm}\mathcal{J}^*(\omega_{\text{in}},\fc)=\mathcal{J}_{3}(\omega_{\text{in}},\fc_r)\pm i\mathcal J_4\left(\omega_{\text{in}},\fc_r\right),
\end{align*}
where
\begin{equation}\label{def-j_3-j_4}
  \begin{aligned}
    \mathcal J_3\left(\omega_{\text{in}},\fc_r\right)=\text{P.V.}\int_{\mathbb{R}}\frac{\int_{y_{\fc}}^{y}\omega_{\text{in}}(y')\phi_1(y',\fc_r)dy'}{\left(b(y)-\fc_r\right)^2\phi_1(y,\fc_r)^2}dy,\quad \mathcal J_4\left(\omega_{\text{in}},\fc_r\right)
    =\frac{\pi\omega_{\text{in}}(y_{\fc})}{b'(y_{\fc})^2}.
  \end{aligned}
\end{equation}
\end{lemma}

\begin{proof}
  We decompose $\mathcal{J}^*(\omega_{\text{in}},\fc)$ as
  \begin{align*}
      \mathcal{J}^*(\omega_{\text{in}},\fc)
      &=\int_{y_{\fc}-1}^{y_{\fc}+1}\int_{y_{c}}^{y}\frac{\omega_{\text{in}}(y')}{\left(b(y)-\fc\right)^2}\left(\frac{\phi_1(y',\fc)}{\phi_1(y,\fc)^2}-\frac{b'(y)}{b'(y_{\fc})}\right)dy'dy\\
      &\quad +\int_{y_{\fc}-1}^{y_{\fc}+1}\frac{b'(y)}{b'(y_{\fc})\left(b(y)-\fc\right)^2}\left(\int_{y_{\fc}}^{y}\omega_{\text{in}}(y')dy'-\omega_{\text{in}}(y_{\fc})\frac{b(y)-\fc_r}{b'(y_{\fc})}\right)dy\\
      &\quad +\int_{y_{\fc}-1}^{y_{\fc}+1}\frac{\omega_{\text{in}}(y_{\fc})b'(y)(b(y)-\fc_r)}{b'(y_{\fc})^2\left(b(y)-\fc\right)^2}dy\\
      &\quad  +\int_{\mathbb{R}\setminus[y_{\fc}-1,y_{\fc}+1]}\frac{\int_{y_{c}}^{y}\omega_{\text{in}}(y')\phi_1(y',\fc)dy'}{\phi(y,\fc)^2}dy\\
      &=\mathcal{J}^*_{3,1}(\omega_{\text{in}},\fc)
      +\mathcal{J}^*_{3,2}(\omega_{\text{in}},\fc)
      +\mathcal{J}^*_{3,3}(\omega_{\text{in}},\fc)+\mathcal{J}^*_{3,4}(\omega_{\text{in}},\fc).
  \end{align*}
Note that $\left|y'-y_{\fc}\right|\le \left|y-y_{\fc}\right|$. By \eqref{eq-phi1-est} and \eqref{eq-phi2-est}, we obtain the following bounds
  \begin{align*}
      &\left|\frac{\phi_1(y',\fc)}{\phi_1(y,\fc)^2}-\frac{b'(y)}{b'(y_{\fc})}\right|\le \frac{b'(y)}{b'(y_{\fc})}\cdot\left( \left|\frac{\phi_1(y',\fc)-1}{\phi_1(y,\fc)^2}\right|+\left|\frac{\phi_1(y,\fc)^2-1}{\phi_1(y,\fc)^2}\right|\right)
      +\left|\frac{\phi_1(y',\fc)}{\phi_1(y,\fc)^2}\right|\cdot\left|\frac{b'(y)}{b'(y_{\fc})}-1\right|\\
      &\quad\qquad\qquad\qquad\qquad\lesssim \min\left\{\left|y-y_{\fc}\right|^2,1\right\}+e^{-C\left|y-y_{\fc}\right|}\left|y-y_{\fc}\right|,\\
      &\left|\int_{y_{c}}^{y}\omega_{\text{in}}(y')dy'\right|\le \left\|\omega_{\text{in}}(y)\right\|_{L^\infty}\left|y-y_{\fc}\right|\lesssim \left\|\omega_{\text{in}}(y)\right\|_{H^1}\left|y-y_{\fc}\right|,\\
      &\int_{y_{\fc}}^{y}\omega_{\text{in}}(y')dy'-\omega_{\text{in}}(y_{\fc})\frac{b(y)-\fc_r}{b'(y_{\fc})}
        \lesssim\left|y-y_{\fc}\right|^2.
  \end{align*}
  Then by the dominated convergence theorem, we have
  \begin{equation}\label{eq-limj31}
    \begin{aligned}
        \lim_{\fc_i\to 0}\mathcal{J}^*_{3,1}(\omega_{\text{in}},\fc)+\mathcal{J}^*_{3,2}(\omega_{\text{in}},\fc)+\mathcal{J}^*_{3,4}(\omega_{\text{in}},\fc)=\mathcal{J}^*_{3,1}(\omega_{\text{in}},\fc_r)+\mathcal{J}^*_{3,2}(\omega_{\text{in}},\fc_r)+\mathcal{J}^*_{3,4}(\omega_{\text{in}},\fc_r).
    \end{aligned}
  \end{equation}
For $\mathcal{J}^*_{3,3}(\omega_{\text{in}},\fc)$, we have
$\mathcal{J}^*_{3,3}(\omega_{\text{in}},\fc)=\frac{\omega_{\text{in}}(y_{\fc})}{b'(y_{\fc})^2}\left(\ln \frac{b(y_{\fc}+1)-\fc}{b(y_{\fc}-1)-\fc}
      -\frac{2i\fc_i}{b(y_{\fc}+1)^2-\fc^2}\right)$, with
  \begin{align*}
      \lim_{\fc_i\to 0\pm}\ln \frac{b(y_{\fc}+1)-\fc}{b(y_{\fc}-1)-\fc}=\ln \left|\frac{b(y_{\fc}+1)-\fc}{b(y_{\fc}-1)-\fc}\right|\pm i\pi,
  \end{align*}
  which gives
  \begin{align}\label{eq-limj33}
      \lim_{\fc_i\to 0\pm}\mathcal{J}^*_{3,3}(\omega_{\text{in}},\fc)
      &= \frac{\omega_{\text{in}}(y_{\fc})}{b'(y_{\fc})^2}\left(\ln \left|\frac{b(y_{\fc}+1)-\fc_r}{b(y_{\fc}-1)-\fc_r}\right|\pm  i \pi\right)\\
      &=\text{P.V.}\int_{y_{\fc}-1}^{y_{\fc}+1}\frac{\omega_{\text{in}}(y_{\fc})b'(y)}{b'(y_{\fc})^2\left(b(y)-\fc_r\right)}dy\pm i\mathcal J_4\left(\omega_{\text{in}},\fc_r\right).
  \end{align}
  Combining \eqref{eq-limj31} with \eqref{eq-limj33}, we conclude the proof.
\end{proof}
Next, we give the regularity estimates for $\mathcal{J}_3$ and $\mathcal{J}_4$, as defined in \eqref{def-j_3-j_4}. Similar estimates were previously obtained in \cite{LiZhao2024}.

\begin{proposition}\label{prop-j3}
  Let $\fc_r\in \mathbb{R}$, $g(y)\in H^2\left(\mathbb{R}\right)$. It holds that
  \begin{align*}
    \left\|\partial_{\fc_r}^n\mathcal{J}_3(g,\fc_r)\right\|_{L^2}+\left\|\partial_{\fc_r}^n\mathcal{J}_4(g,\fc_r)\right\|_{L^2}\lesssim \left\|g\right\|_{H^n},\quad n=0,1,2.
  \end{align*}
\end{proposition}
\begin{proof}
  For $\mathcal J_4\left(\omega_{\text{in}},\fc_r\right)=\frac{\pi\omega_{\text{in}}(y_{\fc})}{b'(y_{\fc})^2}$,  the assumption $b''(y)\in H^4$ immediately yields
  \begin{align*}
    \left\|\partial_{\fc_r}^n\mathcal{J}_4(g,\fc_r)\right\|_{L^2}\lesssim \left\|g\right\|_{H^n},\quad n=0,1,2.
  \end{align*}
To handle $\mathcal J_3\left(\omega_{\text{in}},\fc_r\right)=\text{P.V.}\int_{\mathbb{R}}\frac{\int_{y_{\fc}}^{y}\omega_{\text{in}}(y')\phi_1(y',\fc_r)dy'}{\left(b(y)-\fc_r\right)^2\phi_1(y,\fc_r)^2}dy$, we notice that the L'Hopital's rule gives
  \begin{align*}
    \left|\lim_{y\to y_\fc}\frac{1}{\left(y-y_\fc\right)}\int_{y_{\fc}}^{y}g(y')\phi_1(y',\fc_r)dy'\right|=\left|g(y_\fc)\right|\le\left\|g\right\|_{H^1}.
  \end{align*}
  Using the limit obtained above and integrating by parts,  we decompose $\mathcal{J}_3(g,\fc_r)$ as
  \begin{align*}
    \mathcal{J}_3(g,\fc_r)&=\text{P.V.}\int_\mathbb{R} \int_{y_{\fc}}^{y}g(y')\phi_1(y',\fc_r)dy'\frac{1}{\phi_1^2(y,\fc_r)}
    \frac{\left(y-y_\fc\right)^2}{\left(b(y)-\fc_r\right)^2}\frac{1}{\left(y-y_\fc\right)^2}dy\\
    &=\text{P.V.}\int_\mathbb{R} g(y)\phi_1(y,\fc_r)\frac{1}{\phi_1^2(y,\fc_r)}
    \frac{\left(y-y_\fc\right)^2}{\left(b(y)-\fc_r\right)^2}\frac{1}{y-y_\fc}dy\\
    &\quad +\text{P.V.}\int_\mathbb{R} \int_{y_{\fc}}^{y}g(y')\phi_1(y',\fc_r)dy'\partial_y\left(\frac{1}{\phi_1^2(y,\fc_r)}
    \frac{\left(y-y_\fc\right)^2}{\left(b(y)-\fc_r\right)^2}\right)\frac{1}{y-y_\fc}dy\\
    &=\text{P.V.}\int_{y_\fc-1}^{y_\fc+1}\frac{1}{b'(y_\fc)^2} \frac{g(y)}{y-y_\fc}dy\\
    &\quad+\text{P.V.}\int_{y_\fc-1}^{y_\fc+1} \frac{g(y)}{y-y_\fc}\left(\frac{1}{\phi_1(y,\fc_r)}
    \frac{\left(y-y_\fc\right)^2}{\left(b(y)-\fc_r\right)^2}-\frac{1}{b'(y_\fc)^2}\right)dy\\
    &\quad +\text{P.V.}\int_{\mathbb{R}\setminus [y_\fc-1,y_\fc+1]}\frac{g(y)}{\phi_1(y,\fc_r)}
    \frac{\left(y-y_\fc\right)}{\left(b(y)-\fc_r\right)^2}dy\\
    &\quad -\text{P.V.}\int_\mathbb{R} \frac{\int_{y_{\fc}}^{y}g(y')\phi_1(y',\fc_r)dy'}{y-y_\fc}\frac{\partial_y\phi_1(y,\fc_r)}{\phi_1^3(y,\fc_r)}
    \frac{\left(y-y_\fc\right)^2}{\left(b(y)-\fc_r\right)^2}dy\\
    &\quad +\text{P.V.}\int_\mathbb{R} \frac{\int_{y_{\fc}}^{y}g(y')\phi_1(y',\fc_r)dy'}{y-y_\fc}\frac{1}{\phi_1^2(y,\fc_r)}
    \partial_y\left(\frac{\left(y-y_\fc\right)^2}{\left(b(y)-\fc_r\right)^2}\right)dy\\
    &=II_1+II_2+II_3+II_4+II_5.
  \end{align*}
  Applying \eqref{eq-phi1-est}, we have the bounds
  \begin{align*}
    &\left|\frac{1}{\phi_1(y,\fc_r)}
    \frac{\left(y-y_\fc\right)^2}{\left(b(y)-\fc_r\right)^2}-\frac{1}{b'(y_\fc)^2}\right|\lesssim \left|y-y_\fc\right|,\quad \left|\frac{1}{\phi_1(y,\fc_r)}
    \frac{\left(y-y_\fc\right)}{\left(b(y)-\fc_r\right)^2}\right|\lesssim \frac{1}{\left|y-y_\fc\right|e^{C_4\left|y-y_\fc\right|}},\\
    &\left|\frac{\partial_y\phi_1(y,\fc_r)}{\phi_1^3(y,\fc_r)}
    \frac{\left(y-y_\fc\right)^2}{\left(b(y)-\fc_r\right)^2}\right|
    +\left|\frac{1}{\phi_1^2(y,\fc_r)}
    \partial_y\frac{\left(y-y_\fc\right)^2}{\left(b(y)-\fc_r\right)^2}\right|\lesssim \frac{1}{\phi_1(y,\fc_r)e^{C_4\left|y-y_\fc\right|}}.
  \end{align*}

Therefore, applying the boundedness of the Hilbert transform to \(II_1\), Young's convolution inequality to \(II_2\) and \(II_3\), and the Hardy–Littlewood maximal inequality to \(II_4\) and \(II_5\), we obtain
    \begin{align*}
      \left\|\mathcal{J}_3(g,\fc_r)\right\|_{L^2_{\fc_r}}&\lesssim \left\|\mathcal{H}(g)(\fc_r)-\int_{\left|y-y_\fc\right|>1}\frac{g(y)}{y-y_\fc}dy\right\|_{L^2_{\fc_r}}
      + \left\|g(y)\right\|_{L^2_{y}}+\left\|\sup_{y\in \mathbb{R}}\frac{\int_{b^{-1}\left(\fc_r\right)}^{y}\left|g(y')\right|dy'}{y-b^{-1}\left(\fc_r\right)}\right\|_{L^2_{\fc_r}}\\
      &\lesssim \left\|g(y)\right\|_{L^2_{y}}.
    \end{align*}
We obtain for $n=1,2$
\begin{align}\label{eq-newt} &\partial_{\fc_r}^n\mathcal{J}_3(g,\fc_r)=\text{P.V.}
    \int_{\mathbb{R}}
    \partial_{G}^n\left(\frac{\int_{y_{\fc}}^{y}
    \phi_1(y',\fc_r)g(y')dy'}{\left(b(y)-\fc_r\right)^2
    \phi_1(y,\fc_r)^2}\right)dy,\end{align}
and decompose  
\begin{equation}\label{eq-decom-parn}
  \begin{aligned}
      &\text{P.V.}\int_{\mathbb{R}}\partial_{G}^n\left(\frac{\int_{y_{\fc}}^{y}\phi_1(y',\fc_r)g(y')dy'}{\left(b(y)-\fc_r\right)^2\phi_1(y,\fc_r)^2}\right)dy\\
      &\ =\text{P.V.}\int_{\mathbb{R}}\sum_{i=0}^{n}C_n^i  \int_{y_{\fc}}^{y}\phi_1(y',\fc_r)\partial_{y'}^ig(y')dy'\cdot\partial_{G}^{n-i}\left(\frac{1}{\left(b(y)-\fc_r\right)^2}\right)\cdot\frac{1}{\phi_1(y,\fc_r)^2}dy\\
      &\quad +\int_{\mathbb{R}}\sum_{i=0}^{n}C_n^i  \int_{y_{\fc}}^{y}\phi_1(y',\fc_r)\partial_{y'}^ig(y')dy'\cdot\sum_{j=0}^{n-i-1}C_{n-i}^j\partial_{G}^{j}\left(\frac{1}{\left(b(y)-\fc_r\right)^2}\right)\cdot\partial_{G}^{n-i-j}\left(\frac{1}{\phi_1(y,\fc_r)^2}\right)dy\\
      &\quad +\int_{\mathbb{R}}\sum_{i=0}^{n}C_n^i  \int_{y_{\fc}}^{y}\sum_{l=1}^{i}C_i^l\partial_{G'}^l\phi_1(y',\fc_r)\left(\frac{\partial_{y'}}{b'(y_\fc)}\right)^{i-l}g(y')dy'\cdot\partial_{G}^{n-i}\left(\frac{1}{\left(b(y)-\fc_r\right)^2\phi_1(y,\fc_r)^2}\right)dy\\
      &\ =III_1+III_2+III_3.
  \end{aligned}
\end{equation}
 For $III_1$, \eqref{eq-pargkb1} gives $\left|\partial_{G}^{n}\Big(\frac{1}{\left(b(y)-\fc_r\right)^2}\Big)\right|\lesssim \frac{1}{\left(y-y_{\fc}\right)^2}$. Using the same  arguments for $\left\|\mathcal{J}_3(g,\fc_r)\right\|_{L^2}$, we obtain
\begin{equation}\label{eq-iii1}
  \begin{aligned}
    \left\|III_1\right\|_{L^2_{\fc_r}}\lesssim \sum_{i=0}^{n}\left\|\partial_{y}^ig\right\|_{L^2}.
  \end{aligned}
\end{equation}
By  Lemma \ref{lem-paryphi}, we have
\begin{align*}
  \left|\partial_{G}^{n}\left(\frac{1}{\phi_1(y,\fc_r)^2}\right)\right|\lesssim \left|\frac{1}{\phi_1(y,\fc_r)^2}\frac{\partial_{G}^{n}\phi_1(y,\fc_r)}{\phi_1(y,\fc_r)}\right|\lesssim \frac{\left|y-y_{\fc}\right|^3 e^{-C_4\left|y-y_{\fc}\right|}}{\phi_1(y,\fc_r)},
\end{align*}
which along with \eqref{eq-pargkb1} implies
\begin{equation}\label{eq-iii2}
  \begin{aligned}
    \left|III_2\right|\lesssim \int_{\mathbb{R}}\sum_{i=0}^{n}\frac{\int_{y_{\fc}}^y\left|\partial_{y'}^ig(y')\right|dy'}{y-y_{\fc}}\cdot \left|y-y_{\fc}\right|^2 e^{-C_4\left|y-y_{\fc}\right|}dy
    \lesssim \sum_{i=0}^{n}\sup_{y\in\mathbb{R}}\frac{\int_{y_{\fc}}^y\left|\partial_{y'}^ig(y')\right|dy'}{y-y_{\fc}}.
  \end{aligned}
\end{equation}
For $III_3$,  similarly, applying $\left|\partial^n_{G}\phi_1(y,\fc_r)\right|\le \left|y-y_\fc\right|^3\phi_1(y,\fc_r)$ from Lemma \ref{lem-paryphi}, we obtain
\begin{equation}\label{eq-iii3}
  \begin{aligned}
    \left|III_3\right|&\lesssim \sum_{i=0}^{n}\sum_{l=1}^{i}\int_{\mathbb{R}} \left|\int_{y_{\fc}}^{y}\left|y'-y_{\fc}\right|^3\phi_1(y',\fc_r)\left|\partial_{y'}^{i-l}g(y')\right|dy'\right|\cdot\left|\partial_{G}^{n-i}\left(\frac{1}{\left(b(y)-\fc_r\right)^2\phi_1(y,\fc_r)^2}\right)\right|dy \\
    &\lesssim \sum_{i=0}^{n}\int_{\mathbb{R}}\frac{\int_{y_{\fc}}^y\left|\partial_{y'}^ig(y')\right|dy'}{y-y_{\fc}}\frac{\left|y-y_{\fc}\right|}{e^{C_4\left|y-y_{\fc}\right|}}dy
    \lesssim \sum_{i=0}^{n}\sup_{y\in\mathbb{R}}\frac{\int_{y_{\fc}}^y\left|\partial_{y'}^ig(y')\right|dy'}{y-y_{\fc}}.
  \end{aligned}
\end{equation}
Therefore, combining estimates \eqref{eq-newt}-\eqref{eq-iii3} with the Hardy-Littlewood maximal inequality, we obtain the desired bound, completing the proof of the proposition.
\end{proof}

\subsubsection{Limit of $\fc\Phi(y,\fc)$}
Now we are in a position to give the limit of $\fc\Phi(y,\fc)$ at the origin.
\begin{proposition}[Limiting absorption principle]\label{prop:LAP}
  Let $b(y)$ satisfy Assumption~\ref{assum}. If we further assume that $\partial_{\fc_r}\mathcal{J}_1(0)+ i\partial_{\fc_r}\mathcal{J}_2(0)\neq 0$, then
   \begin{align*}
    \lim_{\pm \fc_i> 0,\left|\fc\right|\to 0}\fc\Phi(y,\fc)
    =-\frac{\mathcal{J}_3(\omega_{\text{in}},0)\pm i\mathcal{J}_4(\omega_{\text{in}},0)}{\partial_{\fc_r}\mathcal{J}_1(0)\mp i\partial_{\fc_r}\mathcal{J}_2(0)}\cdot i\Gamma(y,0),
   \end{align*}
   where $\Gamma(y,0)\in H^2_y$ is the eigenfunction (in the sense of stream function) of $\mathcal{R}$ associated to the embedded eigenvalue $\fc_*=0$.
\end{proposition}

\begin{proof}
  To begin with, we consider the case $y<0$. For the term $\fc\Phi_{i,l}(y,\fc)$, we first estimate
  \begin{align*}
      \left|\Phi_{i,l}(y,\fc)\right|&=\left|\phi(y,\fc)\int_{-\infty}^{y}\frac{\int_{y_{c}}^{y'}\omega_{\text{in}}(y'')\phi_1(y'',\fc_r)\phi_2(y'',\fc)dy''}{\phi(y',\fc)^2}dy\right|\\
      &\le \left\|\omega_{\text{in}}(y)\right\|_{L^\infty}\left|\phi_1(y,\fc)\left(b(y)-\fc_r\right)\right|\int_{-\infty}^{y}\frac{1}{\left(\fc_r-b(y')\right)\phi_1(y',\fc_r)}dy'.
  \end{align*}
  Then we have
  \begin{equation}\label{eq-phiil}
    \begin{aligned}
        \lim_{\pm\fc_i> 0,\;\fc\to 0}\fc\Phi_{i,l}(y,\fc) =0.
    \end{aligned}
  \end{equation}

  Recall that if $\psi(y,\fc)\in H^2_{loc}$ is a solution of \eqref{eq-Rayleigh1}, then $\psi(y,\fc)$ has the  form of \eqref{eq-psi-eigenfunction}
\begin{align*}
  \psi(y,\fc)=a_1^-\phi(y,\fc)+a_2^-\varphi^-(y,\fc)=a_1^+\phi(y,\fc)+a_2^+\varphi^+(y,\fc),
\end{align*}
with
\begin{align*}
  \varphi^-(y,\fc)=\phi(y,\fc)\int_{-\infty}^{y}\frac{1}{\phi(y',\fc)^2}dy',\quad
  \varphi^+(y,\fc)=\phi(y,\fc)\int_{y}^{+\infty}\frac{1}{\phi(y',\fc)^2}dy',
\end{align*}
where $\fc_i\neq 0$ and $a_1^\pm$, $a_2^\pm$ are constants.
As \(y\to -\infty\) (or \(y\to +\infty\)), the function \(\varphi^-(y,\fc)\) (or \(\varphi^+(y,\fc)\)) and \(\partial_y\varphi^-(y,\fc)\) (or \(\partial_y\varphi^+(y,\fc)\)) decay exponentially, whereas \(\phi(y,\fc)\) grows exponentially.
Under the hypothesis that \(\mathcal{R}=b(y)-b''(y)\left(\partial_y^2-1\right)^{-1}\) possesses an embedded eigenvalue at \(0\),
the functions \(\varphi^\pm(y,0)\) are non-singular at \(y=0\) and
provide continuous extensions of \(\varphi^\pm(y,\fc)\) as $\fc\to 0$.
Hence the corresponding eigenfunction must satisfy
\begin{align*}
  \Gamma(y,0)=\left\{\begin{aligned}
  &\phi(y,0)\int_{-\infty}^y\frac{1}{\phi^2(y',0)}dy',\quad y<0,\\
  &\phi(y,0)\int_{+\infty}^y\frac{1}{\phi^2(y',0)}dy',\quad y>0,
  \end{aligned}\right.
  \end{align*}
  and $\Gamma(y,0)\in H^2_y$.
For a detailed discussion of this point, see \eqref{parvarphi} and the subsequent analysis, as well as the proofs of Lemma 5.4 and Lemma 5.6 in \cite{LiMasmoudiZhao2022critical}.

  An application of the dominated convergence theorem gives
  \begin{equation}\label{eq-phihl}
    \begin{aligned}
        \lim_{\pm\fc_i> 0,\left|\fc\right|\to 0}\Phi_{h,l}(y,\fc)
        =\lim_{\pm\fc_i> 0,\left|\fc\right|\to 0}i\phi(y,\fc)\int_{-\infty}^{y}\frac{1}{\phi(y',\fc)^2}dy'
        =i\Gamma(y,0).
    \end{aligned}
  \end{equation}
For $\mu\left(\omega_{\text{in}},\fc\right)$, applying Lemma \ref{lem-j_3} and Lemma \ref{lem-wronskian}, we have
  \begin{equation}\label{eq-fcmu}
    \begin{aligned}
        \lim_{\pm\fc_i> 0,\left|\fc\right|\to 0}\fc\mu\left(\omega_{\text{in}},\fc\right)
        =\lim_{\pm\fc_i> 0,\left|\fc\right|\to 0}\frac{\mathcal{J}^*(\omega_{\text{in}},\fc)}{\frac{W(\fc)}{\fc}}
        =\frac{\mathcal{J}_3(\omega_{\text{in}},0)\pm i\mathcal{J}_4(\omega_{\text{in}},0)}{\partial_{\fc_r}\mathcal{J}_1(0)\mp i\partial_{\fc_r}\mathcal{J}_2(0)}.
    \end{aligned}
  \end{equation}
  Thus, combining \eqref{eq-phiil}, \eqref{eq-phihl} and \eqref{eq-fcmu}, we conclude that for $y<y_{\fc}$,
  \begin{align*}
      \lim_{\pm \fc_i> 0,\left|\fc\right|\to 0}\fc\Phi(y,\fc)=-\frac{\mathcal{J}_3(\omega_{\text{in}},0)\pm i\mathcal{J}_4(\omega_{\text{in}},0)}{\partial_{\fc_r}\mathcal{J}_1(0)\mp i\partial_{\fc_r}\mathcal{J}_2(0)}\cdot i\Gamma(y,0).
  \end{align*}
  The case for $y>0$ can be treated similarly. This ends the proof.
\end{proof}

\subsection{Representation formula of the stream function}
The aim of this subsection is to derive the representation formula for the stream function $\Psi(t,y)$ in \eqref{eq-rep-Psi}.
We first establish some necessary preliminaries.
We define
\begin{align*}
  \Phi_{\pm}(y,\fc_r)=
  \left\{
\begin{aligned}
&i\phi(y,\fc_r)\int_{-\infty}^y\frac{\int_{y_{\fc}}^{y'}\omega_{\text{in}}(y'')\phi_1(y'',\fc_r)dy''}{\phi^2(y',\fc_r)}dy'
 -i\mu_{\pm}(\omega_{\text{in}},\fc_r)\Gamma(y,\fc_r),\quad y<y_\fc,\\
&i\phi(y,\fc_r)\int_{+\infty}^y\frac{\int_{y_{\fc}}^{y'}\omega_{\text{in}}(y'')\phi_1(y'',\fc_r)dy''}{\phi^2(y',\fc_r)}dy'
-i\mu_{\pm}(\omega_{\text{in}},\fc_r)\Gamma(y,\fc_r),\quad y<y_\fc,
\end{aligned}
\right.
\end{align*}
where
\begin{align*}
  \Gamma(y,\fc_r)=\left\{\begin{aligned}
  &\phi(y,\fc_r)\int_{-\infty}^y\frac{1}{\phi^2(y',\fc_r)}dy',\quad y<y_\fc,\\
  &\phi(y,\fc_r)\int_{+\infty}^y\frac{1}{\phi^2(y',\fc_r)}dy',\quad y>y_\fc,
  \end{aligned}\right.
\end{align*}
and
\begin{align*}
  \mu_{\pm}(\omega_{\text{in}},\fc_r)=\frac{\mathcal J_3(\omega_{\text{in}},\fc_r)\pm i\mathcal J_4(\omega_{\text{in}},\fc_r)}{\mathcal J_1(\fc_r)\mp i\mathcal J_2(\fc_r)},
\end{align*}
with $\mathcal J_1(\fc_r)$, $\mathcal J_2(\fc_r)$ given in Lemma \ref{lem-lim-eigen}, and
 $\mathcal J_3(\omega_{\text{in}},\fc_r)$, $\mathcal J_4(\omega_{\text{in}},\fc_r)$ given in Lemma \ref{lem-j_3}.
Refer to \cite{LiMasmoudiZhao2022critical} Lemma 5.12, we have for $\fc_r\neq 0$
\begin{equation}\label{eq-limphipm}
  \begin{aligned}
    \lim_{\fc_i\to 0\pm}\Phi(y,\fc_r+i\fc_i)&=\Phi_{\pm}(y,\fc_r).
  \end{aligned}
\end{equation}

\begin{proposition}[Representation formula]\label{prop:RF}
  Under the same assumption to Proposition \ref{prop:LAP}, for $\omega_{\text{in}}(y)\in H^2$, we have
  \begin{align*}
      \Psi(t,y)&=-\frac{1}{\pi}\text{P.V.}\int_{-\infty}^{+\infty} e^{-i\fc_r t} \frac{\mathcal J_1\mathcal J_4+\mathcal J_2\mathcal J_3}{\mathcal J_1^2+\mathcal J_2^2} \Gamma(y,\fc_r) d\fc_r\\
      &\quad +\frac{\Gamma(y,0)}{\left(\partial_{\fc_r}\mathcal{J}_1(0)\right)^2+\left(\partial_{\fc_r}\mathcal{J}_2(0)\right)^2}
      \cdot\left(\partial_{\fc_r}\mathcal{J}_1(0)\mathcal{J}_3(\omega_{\text{in}},0)-\partial_{\fc_r}\mathcal{J}_2(0)\mathcal{J}_4(\omega_{\text{in}},0)\right).
  \end{align*}
\end{proposition}

\begin{proof}
  By Assumption \ref{assum} that $\mathcal{R}$ has no eigenvalues except $0$, \mbox{$\{c_r+ic_i|c_i \neq 0\ \text{and}\ c_r^2+c_i^2>\delta\}$} is in the resolvent set of $\mathcal{L}$ for any $\delta>0$.
  Thus, $\left(\fc-\mathcal{L}\right)^{-1}$ is an analytic operator-valued function in this domain.

  It then follows from the Cauchy integral theorem and \eqref{eq-limphipm} that
  \begin{align*}
      \Psi(t,y)&=\frac{1}{2\pi i}\lim_{T\to +\infty}\int_{-T}^{T}e^{-i(\fc_r-ic^\diamond)t}\cdot i\Phi(y,\fc_r-ic^\diamond)-e^{-i(\fc_r+ic^\diamond)t} \cdot i\Phi(y,\fc_r+ic^\diamond)d\fc_r\\
      &=\frac{1}{2\pi }\lim_{T\to +\infty}\left(\int_{\left(-T,-\delta\right)\cup \left(\delta,T\right)}e^{-i\fc_r t}\left(\Phi_-(y,\fc_r)-\Phi_+(y,\fc_r)\right)d\fc_r
      +\int_{\left|\fc\right|=\delta}e^{-i\fc t}\Phi(y,\fc)d\fc\right)\\
      &\quad +\frac{i}{2\pi }\lim_{T\to +\infty}\int_{0}^{c^\diamond}e^{-i\left(-T-i\fc_i\right)t}\Phi(y,-T-i\fc_i)-e^{-i\left(T-i\fc_i\right)t}\Phi(y,T-i\fc_i)d\fc_i\\
      &\quad +\frac{i}{2\pi }\lim_{T\to +\infty}\int_{0}^{c^\diamond}e^{-i\left(-T+i\fc_i\right)t}\Phi(y,-T+i\fc_i)-e^{-i\left(T+i\fc_i\right)t}\Phi(y,T+i\fc_i)d\fc_i.
  \end{align*}
  By Proposition 2.11 in \cite{LiZhao2024}, we have
  \begin{align*}
      \left\|\Phi(y,\fc)\right\|_{L_y^\infty}\lesssim \frac{\left\|\omega_{\text{in}}(y)\right\|_{L_y^2}}{1+\left|\fc_r\right|}\quad \text{for}\quad \left|\fc_r\right|\ge C.
  \end{align*}
  Then   for any $\delta\in (0,\frac{1}{2}c^\diamond)$, it follows
  \begin{align*}
      \Psi(t,y)&=\frac{i}{2\pi}\int_{\left(-\infty,-\delta\right)\cup \left(\delta,\infty\right)}e^{-i\fc_r t}\left(\Phi_-(y,\fc_r)-\Phi_+(y,\fc_r)\right)d\fc_r
      +\frac{1}{2\pi}\int_{\left|\fc\right|=\delta}e^{-i\fc t}\Phi(y,\fc)d\fc\\
      &=-\frac{1}{\pi}\int_{\left(-\infty,-\delta\right)\cup \left(\delta,\infty\right)}e^{-i\fc_r t} \frac{\mathcal J_1\mathcal J_4+\mathcal J_2\mathcal J_3}{\mathcal J_1^2+\mathcal J_2^2} \Gamma(y,\fc_r) d\fc_r
      +\frac{1}{2\pi}\int_{\left|\fc\right|=\delta}e^{-i\fc t}\Phi(y,\fc)d\fc.
  \end{align*}
  Now letting $\delta\to 0+$, by the limiting absorption principle and the dominated convergence theorem, we obtain
  \begin{align*}
      &\lim_{\delta\to 0+}\frac{1}{2\pi }\int_{\left|\fc\right|=\delta}e^{-i\fc t}\Phi(y,\fc)d\fc \overset{\fc=\delta e^{i\theta}}{=\joinrel=\joinrel=}
      \lim_{\delta\to 0+}\frac{i}{2\pi }\int_{-\pi}^{\pi} e^{-i\fc t}\fc\Phi(y,\fc) d\theta=\frac{i}{2\pi}\int_{-\pi}^{\pi}\lim_{\delta\to 0+} e^{-i\fc t}\fc\Phi(y,\fc) d\theta\\
      &\quad =\frac{1}{2\pi }\int_{0}^{\pi}\frac{\mathcal{J}_3(\omega_{\text{in}},0)+\frac{i\pi}{b'(0)^2}\omega_{\text{in}}(0)}{ \partial_{\fc_r}\mathcal{J}_1(0)- i\partial_{\fc_r}\mathcal{J}_2(0)}\Gamma(y,0) d\theta
      +\frac{1}{2\pi }\int_{-\pi}^{0}\frac{\mathcal{J}_3(\omega_{\text{in}},0)-\frac{i\pi}{b'(0)^2}\omega_{\text{in}}(0)}{\partial_{\fc_r}\mathcal{J}_1(0)+ i\partial_{\fc_r}\mathcal{J}_2(0)}\Gamma(y,0) d\theta\\
      &\quad=\frac{\Gamma(y,0)}{\left(\partial_{\fc_r}\mathcal{J}_1(0)\right)^2+\left(\partial_{\fc_r}\mathcal{J}_2(0)\right)^2}\left(\partial_{\fc_r}\mathcal{J}_1(0)\mathcal{J}_3(\omega_{\text{in}},0)-\partial_{\fc_r}\mathcal{J}_2(0)\mathcal{J}_4(\omega_{\text{in}},0)\right).
  \end{align*}

  Thus, we obtain the explicit expression of $\Psi(t,y)$ as
  \begin{align*}
      \Psi(t,y)&=-\frac{1}{\pi}\text{P.V.}\int_{-\infty}^{+\infty} e^{-i\fc_r t} \frac{\mathcal J_1\mathcal J_4+\mathcal J_2\mathcal J_3}{\mathcal J_1^2+\mathcal J_2^2} \Gamma(y,\fc_r) d\fc_r\\
      &\quad +\frac{\Gamma(y,0)}{\left(\partial_{\fc_r}\mathcal{J}_1(0)\right)^2+\left(\partial_{\fc_r}\mathcal{J}_2(0)\right)^2}
      \cdot\left(\partial_{\fc_r}\mathcal{J}_1(0)\mathcal{J}_3(\omega_{\text{in}},0)-\partial_{\fc_r}\mathcal{J}_2(0)\mathcal{J}_4(\omega_{\text{in}},0)\right).
  \end{align*}
This completes the proof of the proposition.
\end{proof}

\section{Inviscid damping type estimate}
Based on the representation formula of the stream function derived above, $\Psi(t,y)$ can be decomposed as
\begin{align}\label{eq-Psi-decomp}
  \Psi(t,y)=\Psi_1(t,y)+\Psi_2(t,y),
\end{align}
where $\Psi_1(t,y)$ corresponds to the contribution from the embedded eigenvalue, and $\Psi_2(t,y)$ corresponds to the contribution from the continuous spectrum.

In this section, we derive that
\begin{align*}
  \Psi_1(t,y)=P(\omega_{\text{in}})\Gamma(y,0)=\frac{\mathcal{J}_{3}(\omega_{\text{in}},0)+ i\mathcal{J}_{4}(\omega_{\text{in}},0)}{\partial_{\fc_r}\mathcal{J}_1(0)- i\partial_{\fc_r}\mathcal{J}_2(0)}\Gamma(y,0),
\end{align*}
where $P(\omega_{\mathrm{in}})$ is referred to as the projection coefficient. Since $\fc_*=0$, this part does not evolve in time. For the remaining part
\begin{align*}
  \Psi_2(t,y)=\Psi(t,y)-P(\omega_{\text{in}})\Gamma(y,0),
\end{align*}
we prove that it admits decay due to inviscid damping.

Moreover, we show that $\Psi_2(t,y)$ admits a further decomposition of the form
\begin{align*}
  \Psi_2(t,y)=\mathring{\Psi}(t,y)-\frac{1}{\pi}\cdot\frac{\left(\partial_{\fc_r}\mathcal{J}_1(0)\mathcal{J}_4(\omega_{\text{in}},0)+\partial_{\fc_r}\mathcal{J}_2(0)\mathcal{J}_3(\omega_{\text{in}},0)\right)\Gamma(y,0)}
    {\left(\partial_{\fc_r}\mathcal{J}_1(0)\right)^2+\left(\partial_{\fc_r}\mathcal{J}_2(0)\right)^2}\Psi_{\chi}(t),
\end{align*}
where
\begin{align}\label{eq-chi}
 \Psi_{\chi}(t)=\int_{(-\infty,-\frac{1}{2})\cup(\frac{1}{2},+\infty)} \chi(\fc)\frac{e^{-i\fc t}}{\fc} d\fc
      +\int_{l_{\frac{1}{2}}} \frac{e^{-i\fc t}}{\fc} d\fc
\end{align}
depends only on $t$ and exhibits polynomial decay in time.
Here \(l_a\) denotes the lower half-circle of radius \(a\) centred at the origin, traversed counter-clockwise.
On the other hand, we prove that $\mathring{\Psi}(t,y)$ satisfies an inviscid damping estimate, which is established via a duality argument.

\begin{proposition}\label{prop-inviscid-damping}
Let $b(y)$ satisfy Assumption~\ref{assum}. If we further assume that $\partial_{\fc_r}\mathcal{J}_1(0)+ i\partial_{\fc_r}\mathcal{J}_2(0)\neq 0$, then it holds that
\begin{align*}
  \left\|\Psi_2(t,y)\right\|_{L_y^2}\lesssim  \frac{1}{t}\left\|\omega_{\text{in}}(y)\right\|_{H_y^2},\quad \left\|\Psi_2(t,y)\right\|_{H_y^1}\to 0,\quad \text{as}\quad t\to +\infty.
\end{align*}
More precisely, we prove that
\begin{align*}
    &\left\|\mathring{\Psi}(t,y)\right\|_{L_y^2}\lesssim \frac{1}{t}\left\|\omega_{\text{in}}(y)\right\|_{H_y^2},\quad \left|\Psi_{\chi}(t)\right|
    \lesssim \frac{1}{t},\quad \left\|\mathring{\Psi}(t,y)\right\|_{ H_y^1}\to 0,\quad \text{as}\quad t\to +\infty,
\end{align*}
where
\begin{align*}
    \mathring \Psi(t,\fc_r)&=\Psi_2(t,y)+\frac{1}{\pi}\cdot\frac{\left(\partial_{\fc_r}\mathcal{J}_1(0)\mathcal{J}_4(\omega_{\text{in}},0)+\partial_{\fc_r}\mathcal{J}_2(0)\mathcal{J}_3(\omega_{\text{in}},0)\right)\Gamma(y,0)}
    {\left(\partial_{\fc_r}\mathcal{J}_1(0)\right)^2+\left(\partial_{\fc_r}\mathcal{J}_2(0)\right)^2}\Psi_{\chi}(t).
    \end{align*}
\end{proposition}

\begin{proof}
We begin by analyzing the principal value integral appearing in the representation formula given in Proposition \ref{prop:RF}.
By our assumption, $\mathcal{J}_1(0)=\mathcal{J}_2(0)=0$ and $\partial_{\fc_r}\mathcal{J}_1(0)+i\partial_{\fc_r}\mathcal{J}_2(0)\neq0$.  The function $\frac{\mathcal J_1\mathcal J_4+\mathcal J_2\mathcal J_3}{\mathcal J_1^2+\mathcal J_2^2}$
may exhibit a first-order singularity at the origin.
To handle this singularity, we examine the following two limits. By L'Hospital's rule, we have
  \begin{align*}
      \lim_{\fc_r\to 0}\frac{\mathcal J_1\mathcal J_4\fc_r}{\mathcal J_1^2+\mathcal J_2^2}&=
      \lim_{\fc_r\to 0}\frac{\left(\mathcal{J}_1\mathcal{J}_4\right)'\fc_r+\mathcal{J}_1\mathcal{J}_4}{2\left(\mathcal{J}_1\mathcal{J}'_1+\mathcal{J}_2\mathcal{J}'_2\right)}
      =\lim_{\fc_r\to 0}\frac{\left(\mathcal{J}_1\mathcal{J}_4\right)''\fc_r+2\left(\mathcal{J}_1\mathcal{J}_4\right)'}{2\left(\mathcal{J}_1\mathcal{J}''_1+\left(\mathcal{J}'_1\right)^2+\mathcal{J}_2\mathcal{J}''_2+\left(\mathcal{J}'_2\right)^2\right)}\\
      &=\frac{\partial_{\fc_r}\mathcal{J}_1(0)\mathcal{J}_4(\omega_{\text{in}},0)}{\left(\partial_{\fc_r}\mathcal{J}_1(0)\right)^2+\left(\partial_{\fc_r}\mathcal{J}_2(0)\right)^2},
  \end{align*}
and similarly,
  \begin{align*}
      \lim_{\fc_r\to 0}\frac{\mathcal J_2\mathcal J_3\fc_r}{\mathcal J_1^2+\mathcal J_2^2}=\frac{\partial_{\fc_r}\mathcal{J}_2(0)\mathcal{J}_3(\omega_{\text{in}},0)}{\left(\partial_{\fc_r}\mathcal{J}_1(0)\right)^2+\left(\partial_{\fc_r}\mathcal{J}_2(0)\right)^2}.
  \end{align*}

We then write
\begin{equation}\label{eq-decom-decay}
  \begin{aligned}
      &\text{P.V.}\int_{-\infty}^{+\infty} e^{-i\fc_r t} \frac{\mathcal J_1\mathcal J_4+\mathcal J_2\mathcal J_3}{\mathcal J_1^2+\mathcal J_2^2} \Gamma(y,\fc_r) d\fc_r\\
      &\quad=\text{P.V.}\int_{-\infty}^{+\infty} \frac{e^{-i\fc_r t}}{\fc_r} \left(\frac{\mathcal J_1\mathcal J_4\fc_r}{\mathcal J_1^2+\mathcal J_2^2}\Gamma(y,\fc_r)-\left.\chi(\fc_r)\cdot\frac{\mathcal{J}_1'\mathcal{J}_4}{\left(\mathcal{J}'_1\right)^2+\left(\mathcal{J}'_2\right)^2} \right|_{\fc_r=0}\Gamma(y,0) \right) d\fc_r\\
      &\qquad +\text{P.V.}\int_{-\infty}^{+\infty} \frac{e^{-i\fc_r t}}{\fc_r} \left(\frac{\mathcal J_2\mathcal J_3\fc_r}{\mathcal J_1^2+\mathcal J_2^2}\Gamma(y,\fc_r)-\left.\chi(\fc_r)\cdot\frac{\mathcal{J}_2'\mathcal{J}_3}{\left(\mathcal{J}'_1\right)^2+\left(\mathcal{J}'_2\right)^2} \right|_{\fc_r=0}\Gamma(y,0) \right) d\fc_r\\
      &\qquad +\text{P.V.}\int_{-\infty}^{+\infty} \chi(\fc_r)\frac{e^{-i\fc_r t}}{\fc_r}  d\fc_r\cdot \frac{\partial_{\fc_r}\mathcal{J}_1(0)\mathcal{J}_4(\omega_{\text{in}},0)+\partial_{\fc_r}\mathcal{J}_2(0)\mathcal{J}_3(\omega_{\text{in}},0)}{\left(\partial_{\fc_r}\mathcal{J}_1(0)\right)^2+\left(\partial_{\fc_r}\mathcal{J}_2(0)\right)^2}\Gamma(y,0) ,
  \end{aligned}
\end{equation}
where $\chi(\fc)$ is a smooth cut-off function defined on $\mathbb C$ such that $\operatorname{supp}\chi\subset \left\{\fc| \left|\fc\right|<2\right\}$ and $\chi=1$ for ${\left|\fc\right|<1}$.

\begin{figure}[t]
    \centering
    \includegraphics[width=0.8\textwidth]{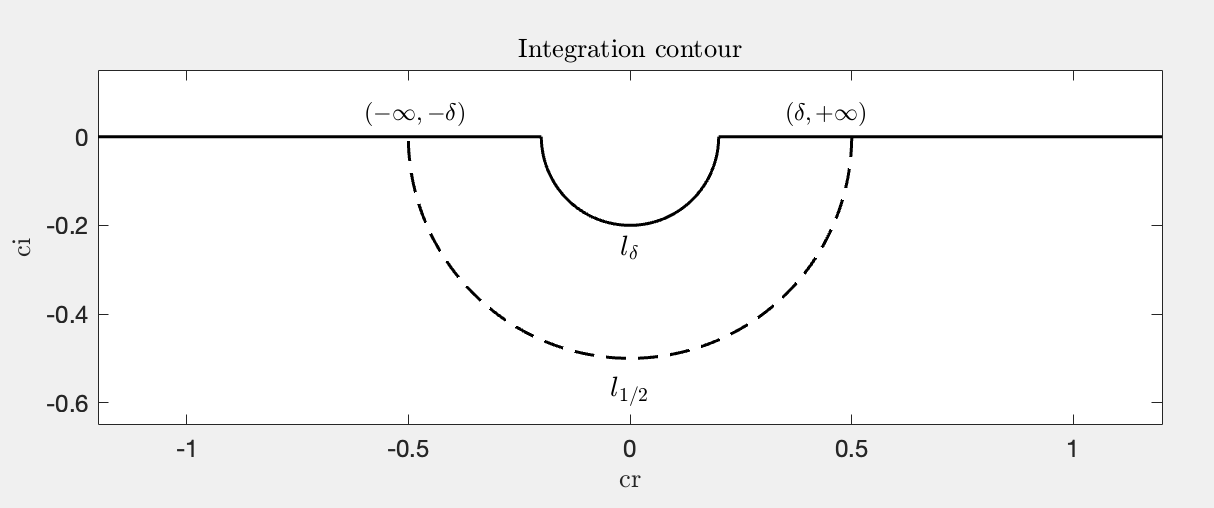}
    \caption{integration contour}
    \label{fig:intcon}
\end{figure}
And since $\chi(\fc)\frac{e^{-i\fc t}}{\fc}$ is analytic in $0<\left|\fc\right|<1$,  for any $0<\delta<\frac{1}{2}$, we have
\begin{align*}
   \left(\int_{-\frac{1}{2}}^{-\delta}+\int_{\delta}^{\frac{1}{2}}+\int_{l_\delta}-\int_{l_\frac{1}{2}} \right)\chi(\fc)\frac{e^{-i\fc t}}{\fc} d\fc=0,
\end{align*}
where \(l_a\) denotes the lower half-circle of radius \(a\) centred at the origin, traversed counter-clockwise.

Then we obtain (see Figure \ref{fig:intcon} for the integration contour)
\begin{equation}\label{eq-chi2}
  \begin{aligned}
      &\text{P.V.}\int_{-\infty}^{+\infty} \chi(\fc)\frac{e^{-i\fc t}}{\fc} d\fc\\
      &\quad=\int_{(-\infty,-\frac{1}{2})\cup(\frac{1}{2},+\infty)} \chi(\fc)\frac{e^{-i\fc t}}{\fc} d\fc
      +\int_{l_{\frac{1}{2}}} \chi(\fc)\frac{e^{-i\fc t}}{\fc} d\fc
      -\lim_{\delta\to 0}\int_{l_{\delta}} \chi(\fc)\frac{e^{-i\fc t}}{\fc} d\fc\\
      &\quad =\int_{(-\infty,-\frac{1}{2})\cup(\frac{1}{2},+\infty)} \chi(\fc)\frac{e^{-i\fc t}}{\fc} d\fc
      +\int_{l_{\frac{1}{2}}} \frac{e^{-i\fc t}}{\fc} d\fc-\pi i
      =\Psi_{\chi}(t)-\pi i.
  \end{aligned}
\end{equation}

Thus plugging \eqref{eq-chi2}  and \eqref{eq-decom-decay} into the explicit expression of $\Psi(t,y)$ in Proposition \ref{prop:RF}, we arrive at
\begin{align*}
    \Psi(t,y)&=\mathring \Psi(t,\fc_r)
    -\frac{1}{\pi}\Psi_{\chi}(t)\cdot\frac{\partial_{\fc_r}\mathcal{J}_1(0)\mathcal{J}_4(\omega_{\text{in}},0)+\partial_{\fc_r}\mathcal{J}_2(0)\mathcal{J}_3(\omega_{\text{in}},0)}{\left(\partial_{\fc_r}\mathcal{J}_1(0)\right)^2+\left(\partial_{\fc_r}\mathcal{J}_2(0)\right)^2}\Gamma(y,0)\\
    &\quad +\frac{\mathcal{J}_{3}(\omega_{\text{in}},0)+ i\mathcal{J}_{4}(\omega_{\text{in}},0)}{\partial_{\fc_r}\mathcal{J}_1(0)- i\partial_{\fc_r}\mathcal{J}_2(0)}\cdot \Gamma(y,0),
\end{align*}
where
\begin{align*}
    \mathring \Psi(t,\fc_r)&=-\frac{1}{\pi}\text{P.V.}\int_{-\infty}^{+\infty} \chi(\fc_r)\frac{e^{-i\fc_r t}}{\fc_r} \left(\frac{\mathcal J_1\mathcal J_4\fc_r}{\mathcal J_1^2+\mathcal J_2^2}\Gamma(y,\fc_r)-\left.\frac{\mathcal{J}_1'\mathcal{J}_4}{\left(\mathcal{J}'_1\right)^2+\left(\mathcal{J}'_2\right)^2} \right|_{\fc_r=0}\Gamma(y,0) \right) d\fc_r\\
    &\qquad -\frac{1}{\pi}\text{P.V.}\int_{-\infty}^{+\infty} \chi(\fc_r)\frac{e^{-i\fc_r t}}{\fc_r} \left(\frac{\mathcal J_2\mathcal J_3\fc_r}{\mathcal J_1^2+\mathcal J_2^2}\Gamma(y,\fc_r)-\left.\frac{\mathcal{J}_2'\mathcal{J}_3}{\left(\mathcal{J}'_1\right)^2+\left(\mathcal{J}'_2\right)^2} \right|_{\fc_r=0}\Gamma(y,0) \right) d\fc_r\\
    &\qquad -\frac{1}{\pi}\text{P.V.}\int_{-\infty}^{+\infty}\left(1- \chi(\fc_r)\right)e^{-i\fc_r t} \frac{\mathcal J_1\mathcal J_4+\mathcal J_2\mathcal J_3}{\mathcal J_1^2+\mathcal J_2^2}\Gamma(y,\fc_r) d\fc_r\\
    &=-\frac{1}{\pi}\left(\mathring \Psi_1(t,\fc_r)+\mathring \Psi_2(t,\fc_r)+\mathring \Psi_3(t,\fc_r)\right).
\end{align*}

Thus, it remains to establish decay estimates for $\mathring \Psi(t,c_r)$ and $\Psi_{\chi}(t)$. We begin by considering $\Psi_{\chi}(t)$ and writing
\begin{align*}
  \Psi_{\chi}(t)=\int_{(-\infty,-\frac{1}{2})\cup(\frac{1}{2},+\infty)}\chi(\fc) \frac{e^{-i\fc t}}{\fc} d\fc
      +\int_{l_{\frac{1}{2}}} \frac{e^{-i\fc t}}{\fc} d\fc=\Psi_{\chi,1}(t)+\Psi_{\chi,2}(t).
\end{align*}
Treating $\Psi_{\chi,1}(t)$ as a real integral and integrating by parts yields
\begin{align*}
  \Psi_{\chi,1}(t)&=\frac{1}{-it}\left(\left.\frac{\chi(\fc)}{\fc}e^{-i\fc t} \right|_{\fc=\frac{1}{2}}^{-\frac{1}{2}}-\int_{(-\infty,-\frac{1}{2})\cup(\frac{1}{2},+\infty)}e^{-i\fc t} \frac{d}{d\fc}\frac{\chi(\fc)}{\fc} d\fc\right).
\end{align*}
Then noting that $\frac{d}{d\fc}\chi(\fc)$ is supported in $(-2,2)$, we have
\begin{equation}\label{eq-psichi1}
  \begin{aligned}
    \left|\Psi_{\chi,1}(t)\right|&\le \frac{1}{t}\left(4+\int_{(-2,-\frac{1}{2})\cup(\frac{1}{2},2)}\left|\frac{\chi'(\fc)}{\fc}\right|d\fc+\int_{(-\infty,-\frac{1}{2})\cup(\frac{1}{2},+\infty)}\frac{\chi(\fc)}{\fc^2} d\fc\right)\lesssim \frac{1}{t}.
  \end{aligned}
\end{equation}
For $\Psi_{\chi,2}(t)$, since $\frac{e^{-i\fc t}}{\fc}$ is analytic, we apply the Newton-Leibniz formula to obtain
\begin{align*}
  \Psi_{\chi,2}(t)&=\frac{1}{-it}\int_{l_{\frac{1}{2}}}\left( \frac{d}{d\fc}\frac{e^{-i\fc t}}{\fc}+\frac{e^{-i\fc t}}{\fc^2} \right)d\fc
  =\frac{1}{-it}\left(\left.\frac{e^{-i\fc t}}{\fc}\right|_{\fc=-\frac{1}{2}}^{\frac{1}{2}}+ \int_{l_{\frac{1}{2}}}\frac{e^{-i\fc t}}{\fc^2} d\fc\right).
\end{align*}
Since $l_{1/2}$ lies in the lower half-plane, we have $|e^{-i c t}| \le 1$ for $\fc\in l_{1/2}, t\in (0,+\infty)$. Therefore
\begin{equation}\label{eq-psichi2}
  \begin{aligned}
    \left|\Psi_{\chi,2}(t)\right|&\le \frac{1}{t}\left(4+ \int_{l_{\frac{1}{2}}}\frac{1}{\left|\fc\right|^2} d\fc\right)
    \le \frac{1}{t}\left(4+ 2\pi\right).
  \end{aligned}
\end{equation}
Combining \eqref{eq-psichi1} and \eqref{eq-psichi2}, we obtain
\begin{equation}\label{eq-psichi}
  \begin{aligned}
    \left|\Psi_{\chi}(t)\right|
      \lesssim \frac{1}{t}.
  \end{aligned}
\end{equation}
Indeed, the same argument yields $\left|\Psi_{\chi}(t)\right|\lesssim t^{-n}$ for any $n\geq 1$.

Next, we give decay estimate for $\mathring \Psi(t,c_r)$ via a duality argument. Recall that
\begin{align*}
  \Gamma(y,\fc_r)=\left\{\begin{aligned}
  &\phi(y,\fc_r)\int_{-\infty}^y\frac{1}{\phi^2(y',\fc_r)}dy',\quad y<y_\fc,\\
  &\phi(y,\fc_r)\int_{+\infty}^y\frac{1}{\phi^2(y',\fc_r)}dy',\quad y>y_\fc,
  \end{aligned}\right.
\end{align*}
Following \cite{WeiZhangZhao2018} (see also \cite{RenZhang2025} and \cite{LiMasmoudiZhao2022critical}), for a test function $f(y)=(\partial_y^2-1)g(y)\in L^2$, an integration by parts in $y$ yields
\begin{align*}
    \int_{\mathbb{R}}\Gamma(y,\fc_r)\cdot (\partial_y^2-1)g(y) d y&=
    -\int_{\mathbb{R}}\frac{\int_{y_{\fc}}^{y'}\phi(y,\fc_r)\cdot (\partial_y^2-1)g(y) d y}{\phi^2(y',\fc_r)} d y'\\
    &=-\mathcal{J}_3\left(b''g,\fc_r\right)-b'(y_{\fc})g(y_{\fc})\mathcal{J}_1(\fc_r).
\end{align*}
Consequently, an integration by parts in $\fc_r$ gives
\begin{equation}\label{eq-psi1}
  \begin{aligned}
      \left\|\mathring \Psi_1(t,y)\right\|_{L^2_y}&=\sup_{\left\|(\partial_y^2-1)g\right\|_{L^2}=1}
      \int_{\mathbb{R}}\mathring \Psi_1(t,y)\cdot (\partial_y^2-1)g(y)\ d y\\
      &=\sup_{\left\|(\partial_y^2-1)g\right\|_{L^2}=1}
      \int_{\mathbb{R}} e^{-i\fc_r t}\chi(\fc_r)\frac{K_1(g,\fc_r)-K_1(g,0)}{\fc_r}d\fc_r\\
      &=\sup_{\left\|(\partial_y^2-1)g\right\|_{L^2}=1}
      \left(\frac{1}{it}\right)\int_{\mathbb{R}} e^{-i\fc_r t}\partial_{\fc_r}\left(\chi(\fc_r)\frac{K_1(g,\fc_r)-K_1(g,0)}{\fc_r}\right)d\fc_r\\
      &\le \sup_{\left\|(\partial_y^2-1)g\right\|_{L^2}=1}
      \frac{1}{t}\left\|\partial_{\fc_r}\left(\chi(\fc_r)\frac{K_1(g,\fc_r)-K_1(g,0)}{\fc_r}\right)\right\|_{L^1_{\fc_r}},
  \end{aligned}
\end{equation}
where
\begin{align*}
    K_1(g,\fc_r)=\frac{\mathcal J_2(\fc_r)\mathcal J_3(\omega_{\text{in}},\fc_r)\fc_r}{\mathcal J_1(\fc_r)^2+\mathcal J_2(\fc_r)^2}\left(\mathcal J_3(b''g,\fc_r)+b'(y_{\fc})g(y_{\fc})\mathcal J_1(\fc_r) \right).
\end{align*}
Since $\partial_{c_r}\chi(c_r)$ is supported on the set $1 < |c_r| < 2$, we have
\begin{equation}\label{eq-par2k2}
  \begin{aligned}
      \left\|\partial_{\fc_r}\chi(\fc_r)\cdot\frac{K_1(g,\fc_r)-K_1(g,0)}{\fc_r}\right\|_{L^1}
      &\lesssim \int_{1<\left|\fc_r\right|<2}\frac{\left|K_1(g,\fc_r)-K_1(g,0)\right|}{\left|\fc_r\right|}d\fc_r\\
      \lesssim& \sup_{1<\left|\fc_r\right|<2} \frac{\left|K_1(g,\fc_r)-K_1(g,0)\right|}{\left|\fc_r\right|}
      \lesssim  \left\|\partial_{\fc_r}K_1(g,\fc_r)\right\|_{L^1_{\fc_r}}.
  \end{aligned}
\end{equation}
Applying Minkowski's inequality yields
\begin{equation}\label{eq-par2k1}
  \begin{aligned}
      &\left\|\chi(\fc_r)\cdot \partial_{\fc_r}\frac{K_1(g,\fc_r)-K_1(g,0)}{\fc_r}\right\|_{L^1_{\fc_r}}
      \le  \left\|\chi(\fc_r)\right\|_{L^\infty_{\fc_r}}\cdot\left\|\int_{0}^{1}s \left(\partial_{\fc_r}^2K_1\right)(g,s\fc_r)ds\right\|_{L^1_{\fc_r}}\\
      &\qquad \le \int_{0}^{1} \left\|s \left(\partial_{\fc_r}^2K_1\right)(g,s\fc_r)\right\|_{L^1_{\fc_r}}ds
      = \int_{0}^{1} \left\| \left(\partial_{\fc_r}^2K_1\right)(g,\fc_r)\right\|_{L^1_{\fc_r}}ds
      = \left\|\partial_{\fc_r}^2K_1(g,\fc_r)\right\|_{L^1_{\fc_r}}.
  \end{aligned}
\end{equation}
Combining Proposition \ref{prop-j1}, Proposition \ref{prop-lowerbound} and Proposition \ref{prop-j3}, we have for $\fc_r\in(-2,2)$ and $\omega_{\text{in}}\in H^2$
\begin{align*}
    &\left|\frac{\fc_r^2}{\mathcal{J}_1(\fc_r)^2+\mathcal{J}_2(\fc_r)^2}\right| \le C,
    \quad\left|\partial_{\fc_r}^n\frac{\mathcal{J}_1(\fc_r)}{\fc_r}\right| +\left|\partial_{\fc_r}^n\frac{\mathcal{J}_2(\fc_r)}{\fc_r}\right|\le C,\\
    &\left\|\partial_{\fc_r}^n\mathcal{J}_3(\omega_{\text{in}},\fc_r)\right\|_{L^2}\le C\left\|\omega_{\text{in}}\right\|_{H^n}, \quad n=0,1,2.
\end{align*}
Then by H\"older's inequality,we have
\begin{equation}\label{eq-k1}
  \begin{aligned}
    \left\|K_1(g,\fc_r)\right\|_{L^1_{\fc_r}}+\left\|\partial_{\fc_r}K_1(g,\fc_r)\right\|_{L^1_{\fc_r}}+\left\|\partial_{\fc_r}^2K_1(g,\fc_r)\right\|_{L^1_{\fc_r}}
    \lesssim \left\|\omega_{\text{in}}\right\|_{H^2}
  \end{aligned}
\end{equation}
Thus combining \eqref{eq-psi1}-\eqref{eq-k1}  we have
\begin{equation}\label{eq-est-psi1}
  \begin{aligned}
      &\left\|\mathring \Psi_1(t,y)\right\|_{L^2_y}\lesssim \frac{1}{t}\left\|\omega_{\text{in}}\right\|_{H^2}.
  \end{aligned}
\end{equation}

In addition, we have
\begin{align*}
      \left\|\mathring \Psi_1(t,y)\right\|_{H^1_y}^2&=
      \int_{\mathbb{R}} e^{-i\fc_r t}\chi(\fc_r)\frac{K_1(\mathring \Psi_1(t,y),\fc_r)-K_1(\mathring \Psi_1(t,y),0)}{\fc_r}d\fc_r,
\end{align*}
with
\begin{align*}
  &\left\|\mathring \Psi_1(t,y)\right\|_{H^1_y}^2\lesssim\left\|\chi(\fc_r)\frac{K_1(\mathring \Psi_1(t,y),\fc_r)-K_1(\mathring \Psi_1(t,y),0)}{\fc_r}\right\|_{L^1_{\fc_r}}\\
    &\qquad \lesssim \left\|K_1(\mathring \Psi_1(t,y),\fc_r)\right\|_{L^2_{\fc_r}}+\left\|\partial_{\fc_r}K_1(\mathring \Psi_1(t,y),\fc_r)\right\|_{L^2_{\fc_r}}
    \lesssim \left\|\mathring \Psi_1(t,y)\right\|_{H^1_y}\left\|\omega_{\text{in}}\right\|_{H^2}.
\end{align*}
Then we have
  \begin{align*}
  &\left\|\chi(\fc_r)\frac{K_1(\mathring \Psi_1(t,y),\fc_r)-K_1(\mathring \Psi_1(t,y),0)}{\fc_r}\right\|_{L^1_{\fc_r}}
     \lesssim \left\|\omega_{\text{in}}\right\|_{H^2}^2,
\end{align*}
which further implies by Riemann-Lebesgue lemma that
  \begin{equation}\label{eq-est-psi1h1}
    \begin{aligned}
        \left\|\mathring \Psi_1(t,y)\right\|_{H^1_y}\to 0, \quad \text{as}\ t\to +\infty.
    \end{aligned}
  \end{equation}

The same argument applies to $\mathring \Psi_2(t,y)$, yielding
  \begin{align*}
      \left\|\mathring \Psi_1(t,y)\right\|_{L^2_y}&\le \sup_{\left\|(\partial_y^2-1)g\right\|_{L^2}=1}
      \frac{1}{t}\left\|\partial_{\fc_r}\left(\chi(\fc_r)\frac{K_2(g,\fc_r)-K_2(g,0)}{\fc_r}\right)\right\|_{L^1_{\fc_r}},
  \end{align*}
where
\begin{align*}
    K_2(g,\fc_r)=\frac{\mathcal J_1(\fc_r)\mathcal J_4(\omega_{\text{in}},\fc_r)\fc_r}{\mathcal J_1(\fc_r)^2+\mathcal J_2(\fc_r)^2}\left(\mathcal J_3(b''g,\fc_r)+b'(y_{\fc})g(y_{\fc})\mathcal J_1(\fc_r) \right).
\end{align*}
And we also have
\begin{equation}\label{eq-est-psi2}
  \begin{aligned}
      &\left\|\mathring \Psi_2(t,y)\right\|_{L^2_y}\lesssim \frac{1}{t}\left\|\omega_{\text{in}}\right\|_{H^2},\quad \text{and}
        \left\|\mathring \Psi_2(t,y)\right\|_{H^1_y}\to 0, \quad \text{as}\ t\to +\infty.
    \end{aligned}
  \end{equation}

As for $\mathring \Psi_3(t,y)$, we have
\begin{align*}
  \left\|\mathring \Psi_3(t,y)\right\|_{L^2_y}&=\sup_{\left\|(\partial_y^2-1)g\right\|_{L^2}=1}
      \int_{\mathbb{R}}\mathring \Psi_3(t,y)\cdot (\partial_y^2-1)g(y)\ d y\\
      &\le \sup_{\left\|(\partial_y^2-1)g\right\|_{L^2}=1}
      \frac{1}{t}\left\|\partial_{\fc_r}\left(\left(1-\chi(\fc_r)\right)K_3(g,\fc_r)\right)\right\|_{L^1_{\fc_r}},
\end{align*}
where
\begin{align*}
    K_3(g,\fc_r)=\frac{\mathcal J_1(\fc_r)\mathcal J_4(\omega_{\text{in}},\fc_r)+\mathcal J_2(\fc_r)\mathcal J_3(\omega_{\text{in}},\fc_r)}{\mathcal J_1(\fc_r)^2+\mathcal J_2(\fc_r)^2}\left(\mathcal J_3(b''g,\fc_r)+b'(y_{\fc})g(y_{\fc})\mathcal J_1(\fc_r) \right).
\end{align*}
Similarly, combining Proposition \ref{prop-j1}, Proposition \ref{prop-lowerbound} and Proposition \ref{prop-j3} we have for $\fc_r\in\mathbb{R}\setminus (-1,1)$   and $\omega_{\text{in}}\in H^2$
\begin{align*}
    &\left|\frac{1}{\mathcal{J}_1(\fc_r)^2+\mathcal{J}_2(\fc_r)^2}\right| \le C,
    \quad\left|\partial_{\fc_r}^n\mathcal{J}_1(\fc_r)\right| +\left|\partial_{\fc_r}^n\mathcal{J}_2(\fc_r)\right|\le C,\\
    &\left\|\partial_{\fc_r}^n\mathcal{J}_3(\omega_{\text{in}},\fc_r)\right\|_{L^2}+\left\|\partial_{\fc_r}^n\mathcal{J}_4(\omega_{\text{in}},\fc_r)\right\|_{L^2}\le C\left\|\omega_{\text{in}}\right\|_{H^n}, \quad n=0,1.
\end{align*}
Then by the H\"older's inequality, we have
\begin{align*}
  \left\|K_3(g,\fc_r)\right\|_{L^1_{\fc_r}}+\left\|\partial_{\fc_r}K_3(g,\fc_r)\right\|_{L^1_{\fc_r}}
    \lesssim \left\|\omega_{\text{in}}\right\|_{H^2},
\end{align*}
which leads  to
\begin{equation}\label{eq-est-psi3}
  \begin{aligned}
      &\left\|\mathring \Psi_3(t,y)\right\|_{L^2_y}\lesssim \frac{1}{t}\left\|\omega_{\text{in}}\right\|_{H^2},\;\;\text{and}\;\;\left\|\mathring \Psi_3(t,y)\right\|_{H^1_y}\to 0 \quad \text{as}\ t\to +\infty.
  \end{aligned}
\end{equation}
Thus by \eqref{eq-est-psi1}, \eqref{eq-est-psi2} and \eqref{eq-est-psi3}, we have $ \left\|\mathring \Psi(t,y)\right\|_{L^2_y}\lesssim \frac{1}{t}\left\|\omega_{\text{in}}\right\|_{H^2}$. And  by \eqref{eq-est-psi1h1}, \eqref{eq-est-psi2} and \eqref{eq-est-psi3}, we have $\left\|\mathring \Psi(t,y)\right\|_{H^1_y}\to 0$, as $t\to +\infty$.
Then the proof is complete.

\end{proof}

\section{Proof of the main results}
In this section, we show the instability from the non-normality of the Rayleigh operator, and give the proof of the main results.

\subsection{Instability associated with a simple embedded eigenvalue}
In this subsection, we prove Theorem \ref{thm}. Since a stronger linear growth will be shown in the next subsection when $\fc_*=0$ is a multiple embedded eigenvalue, it suffices to consider the case where $\fc_*=0$ is a simple embedded eigenvalue. The proof in this case relies on the inviscid damping result established in Proposition \ref{prop-inviscid-damping}.
\begin{proof}[Proof of Theorem \ref{thm}]
  For $b(y)$ satisfying  $\partial_{\fc_r}\mathcal{J}_1(0)=\partial_{\fc_r}\mathcal{J}_2(0)=0$,
   the $L^2$ norm of $\omega(t,y)$ has linear growth, which is later established in the proof of Theorem \ref{thm2}.

  Therefore, in what follows, we restrict our consideration to $b(y)$ for which  $\partial_{\fc_r}\mathcal{J}_1(0)+i\partial_{\fc_r}\mathcal{J}_2(0)\neq 0$.
  For this class, the inviscid damping estimate given in Proposition \ref{prop-inviscid-damping} holds.

  First, for a given large $M > 0$,
  we construct $\omega_{\text{in}}$  with
  $\left\|\omega_{\text{in}}\right\|_{L^2}+\left\|\omega_{\text{in}}\right\|_{L^{\infty}}\le 1$,
  such that
  \begin{equation}\label{eq-con-j3j4}
    \begin{aligned}
      \mathcal{J}_{3}(\omega_{\text{in}},0)\gtrsim M\quad \text{and}\quad \mathcal{J}_{4}(\omega_{\text{in}},0)=0.
    \end{aligned}
  \end{equation}
Let $\chi_Z(y)$ be a smooth cut-off function which is supported in $\left[\frac{1}{Z},1\right]$ with
  value $1$ on $\left[\frac{2}{Z},\frac{1}{2}\right]$.
  Define
\begin{align*}
    \omega_{\text{in}}(y)\eqdef\frac{1}{2} \chi_Z(y).
\end{align*}
Then by a direct calculation, we have
\begin{align*}
  \left\|\omega_{\text{in}}\right\|_{L^2}+\left\|\omega_{\text{in}}\right\|_{L^{\infty}}\le 1.
\end{align*}
Since $\phi_1(y,0)\ge 1$, we have
\begin{align*}
    \int_{0}^{y}\omega_{\text{in}}(y')\phi_1(y',0)dy'\ge
    \left\{
    \begin{array}{ll}
        & \frac{y}{2}-\frac{1}{Z}, \quad y\in \left[\frac{2}{Z},\frac{1}{2}\right],\\
        & 0,\quad \text{otherwise}.
    \end{array}\right.
\end{align*}
Then applying \eqref{eq-phi1-est} yields
\begin{align*}
    \mathcal J_3\left(\omega_{\text{in}},0\right)\gtrsim \int_{\frac{2}{Z}}^{\frac{1}{2}}\frac{\frac{y}{2}-\frac{1}{Z}}{y^2}dy
    =\frac{1}{2}\ln\frac{Z}{4}+\frac{2}{Z}-\frac{1}{2}.
\end{align*}
Taking $Z=e^M$ for $M>0$ large enough,  we obtain the required $\omega_{\text{in}}$.

  Next, we show this $\omega_{\text{in}}$
  is exactly the example required by Theorem \ref{thm}. Define
  \begin{align*}
    \omega_*(y)\eqdef (\partial_y^2-1)\Gamma(y,0).
  \end{align*}
  Then for the stream function $\Psi(t, y)$
  solving \eqref{eq: LinearEuler-Psi}, we recall the decomposition \eqref{eq-Psi-decomp} and write:
  \begin{align*}
    \Psi(t, y)&=\left(\partial_y^2-1\right)^{-1}e^{-i\mathcal{R}t}\omega_{\text{in}}(y)\\
    &=\left(\partial_y^2-1\right)^{-1}\left(\frac{\mathcal{J}_{3}(\omega_{\text{in}},0)+ i\mathcal{J}_{4}(\omega_{\text{in}},0)}{\partial_{\fc_r}\mathcal{J}_1(0)- i\partial_{\fc_r}\mathcal{J}_2(0)}\omega_*(y)\right)\\
    &\quad +\left(\partial_y^2-1\right)^{-1}\left(e^{-i\mathcal{R}t}\omega_{\text{in}}(y)-\frac{\mathcal{J}_{3}(\omega_{\text{in}},0)+ i\mathcal{J}_{4}(\omega_{\text{in}},0)}{\partial_{\fc_r}\mathcal{J}_1(0)- i\partial_{\fc_r}\mathcal{J}_2(0)}\omega_*(y)\right)\\
    &=\Psi_1(t, y)+\Psi_2(t, y).
  \end{align*}
  Then, by Proposition \ref{prop-inviscid-damping}, we have $\left\|\Psi_2(t, y)\right\|_{L^2}\lesssim \frac{1}{t}$. Consequently,
  by \eqref{eq-con-j3j4}, there exists a constant $T>0$ large enough such that for all $t\ge T$,
  \begin{align*}
    \left\|\Psi(t, y)\right\|_{L^2}&\ge  \frac{1}{2}\left\|\Psi_1(t, y)\right\|_{L^2}
    \gtrsim \left|\mathcal{J}_{3}(\omega_{\text{in}},0)\right|\left\|\left(\partial_y^2-1\right)^{-1}\omega_{*}(y)\right\|_{L^2}
    \gtrsim M\left\|\Gamma(y,0)\right\|_{L^2}.
  \end{align*}
  Applying Plancherel's theorem, we obtain for $t\ge T$,
  \begin{equation}\label{est-l2}
    \begin{aligned}
      \left\|e^{-i\mathcal{R}t}\omega_{\text{in}}(y)\right\|_{L^2}&=\left\|\left(1+\xi^2\right)\widehat{\Psi}(t, \xi)\right\|_{L^2}\ge\left\|\Psi(t, y)\right\|_{L^2}
      \gtrsim M\left\|\Gamma(y,0)\right\|_{L^2}.
    \end{aligned}
  \end{equation}
  As for the $L^{\infty}$ norm, using the Green's function, we have
  \begin{align*}
    \left\|\Psi(t, y)\right\|_{L^\infty}=\left\|\int_{\mathbb{R}}\frac{1}{2}e^{-\left|y-x\right|}e^{-i\mathcal{R}t}\omega_{\text{in}}(x)dx\right\|_{L^\infty}
    \lesssim \left\|e^{-i\mathcal{R}t}\omega_{\text{in}}(y)\right\|_{L^\infty}.
  \end{align*}
  Again by Proposition \ref{prop-inviscid-damping}, we have
  \begin{align*}
    \left\|\Psi_2(t, y)\right\|_{L^{\infty}}^2\le \left\|\Psi_2(t, y)\right\|_{L^{2}}\left\|\partial_y\Psi_2(t, y)\right\|_{L^{2}}
    \lesssim \frac{1}{t}\left\|\Psi_2(t, y)\right\|_{H^{1}}^2\lesssim \frac{1}{t}.
  \end{align*}
  Similar to the $L^2$ estimate, by Proposition \ref{prop-inviscid-damping}, for  $t\ge T$ with $T>0$ large enough, we have
  \begin{align*}
    \left\|\Psi(t, y)\right\|_{L^\infty}\ge  \frac{1}{2}\left\|\Psi_1(t, y)\right\|_{L^\infty}\gtrsim M\left\|\Gamma(y,0)\right\|_{L^\infty}.
  \end{align*}
  Finally, we obtain for $t\ge T$,
  \begin{equation}\label{est-linfty}
    \begin{aligned}
      \left\|e^{-i\mathcal{R}t}\omega_{\text{in}}(y)\right\|_{L^\infty}\gtrsim \left\|\Psi(t, y)\right\|_{L^\infty}\gtrsim M\left\|\Gamma(y,0)\right\|_{L^\infty}.
    \end{aligned}
  \end{equation}
Thus, combining \eqref{est-l2} and \eqref{est-linfty}, we conclude the proof of Theorem \ref{thm}.
\end{proof}
\begin{remark}
  We may alternatively choose \(\omega_{\text{in}}(y)=Z e^{-Z^4y^{2}}\) with a large parameter \(Z>0\). Then \(\|\omega_{\text{in}}\|_{L^{2}}\lesssim 1\) while \(\|\omega_{\text{in}}\|_{L^{\infty}}\gtrsim Z\).
We also have \(\mathcal{J}_{4}(\omega_{\text{in}},0)\gtrsim\omega_{\text{in}}(0)=Z\).
Repeating the arguments of the previous proof shows that this choice of \(\omega_{\text{in}}\) also leads to an instability. 
\end{remark}

\subsection{Instability associated with a multiple embedded eigenvalue}
In this subsection, we give the proof of Theorem \ref{thm2}.

According to Lemma 2 in \cite{stepin1996rayleigh},
the condition $\partial_{\fc_r}\mathcal{J}_1(\fc_r)=\partial_{\fc_r}\mathcal{J}_2(\fc_r)=0$
is equivalent to the eigenvalue $\fc_r$ having geometric multiplicity greater than 1. In this situation, we show that the solutions $\varphi^\pm(y,\fc_r)$ constructed in \eqref{eq-phi-pm} have certain specific properties, which ensure the existence of an associated solving \eqref{eq-assosi}. This observation is a key step in the construction of unstable solutions.
\begin{proposition}\label{prop-linear-growth}Under the Assumption \ref{assum},
  for $\fc_r\in \mathbb{R}$,
  \begin{align}\label{equiv-1}
    \lim_{y\to y_{\fc}+}\partial_{\fc_r}\varphi^+(y,\fc_r)=\lim_{y\to y_{\fc}-}\partial_{\fc_r}\varphi^-(y,\fc_r)
   \;\;\text{if\;and\;only\;if}\;\;\mathcal{J}_1(\fc_r)=\mathcal{J}_2(\fc_r)=0.
  \end{align}
  Moreover, if $\mathcal{J}_1(\fc_r)=\mathcal{J}_2(\fc_r)=0$,
  then
  \begin{align}\label{equiv-2}
    \lim_{y\to y_{\fc}+}\partial_y\partial_{\fc_r}\varphi^+(y,\fc_r)=\lim_{y\to y_{\fc}-}\partial_y\partial_{\fc_r}\varphi^-(y,\fc_r)
     \;\;\;\text{if\;and\;only\;if}\;\; \partial_{\fc_r}\mathcal{J}_1(\fc_r)=\partial_{\fc_r}\mathcal{J}_2(\fc_r)=0.
  \end{align}
\end{proposition}
\begin{proof}
Recall that $\varphi^-(y,\fc_r)=\phi(y,\fc_r)\int^y_{-\infty}
  \frac{1}{\phi(y',\fc_r)^2}dy'$ for $ y<y_c$, 
   $\varphi^+(y,\fc_r)=\phi(y,\fc_r)\int^y_{+\infty}
  \frac{1}{\phi(y',\fc_r)^2}dy'$ for  $y>y_c$,
and 
    $\mathcal J_1(\fc_r)=\frac{1}{b'(y_\fc)}\Pi_1(\fc_r)+\Pi_2(\fc_r)$, $\mathcal J_2(\fc_r)=\pi\frac{b''(y_\fc)}{{b'}(y_\fc)^3}$,
  with
  \begin{align*}
  \Pi_1(\fc_r)= \text{P.V.}\int_{-\infty}^{+\infty} \frac{b'(y_\fc)-b'(y)}{(b(y)-\fc_r)^2} dy,\quad \Pi_2(\fc)= \int_{-\infty}^{+\infty} \frac{1}{(b(y)-\fc_r)^2}\left(\frac{1}{\phi_1(y,\fc_r)^2}-1\right) dy.
\end{align*}

  Define
  \begin{align*}
    &{\Pi}^{\pm}_1(y,\fc_r)\eqdef\int^{y}_{\pm\infty}\frac{b'(y_\fc)-b'(y')}{(b(y')-\fc_r)^2}dy',\quad {\Pi}^{\pm}_2(y,\fc_r)\eqdef\int^y_{\pm\infty}\frac{1}{(b(y')-\fc_r)^2}\left(\frac{1}{\phi_1^2(y',\fc_r)}-1\right)dy'.
  \end{align*}
For $y<y_{\fc}$, we write
  \begin{equation}\label{decom-varphi}
    \begin{aligned}
        \varphi^-(y,\fc_r)&=\phi(y,\fc_r)\int^y_{-\infty}\frac{1}{\phi^2(y',\fc_r)}dy'\\
        &=\phi(y,\fc_r)\int^y_{-\infty}\frac{1}{(b(y')-\fc_r)^2}dy'+\phi(y,\fc_r)\int^y_{-\infty}\frac{1}{(b(y')-\fc_r)^2}\left(\frac{1}{\phi_1^2(y',\fc_r)}-1\right)dy'\\
        &=\frac{\phi(y,\fc_r)}{b'(y_\fc)}\int^{y}_{-\infty}\frac{b'(y_\fc)-b'(y')}{(b(y')-\fc_r)^2}dy'+\frac{\phi(y,\fc_r)}{b'(y_\fc)}\int^{y}_{-\infty}\frac{b'(y')}{(b(y')-\fc_r)^2}dy'\\
        &\qquad +\phi(y,\fc_r)\int^y_{-\infty}\frac{1}{(b(y')-\fc_r)^2}\left(\frac{1}{\phi_1^2(y',\fc_r)}-1\right)dy'\\
        &=\frac{\phi(y,\fc_r)}{b'(y_\fc)}\int^{y}_{-\infty}\frac{b'(y_\fc)-b'(y')}{(b(y')-\fc_r)^2}dy'-\frac{\phi_1(y,\fc_r)}{b'(y_\fc)}\\
        &\qquad +\phi(y,\fc_r)\int^y_{-\infty}\frac{1}{(b(y')-\fc_r)^2}\left(\frac{1}{\phi_1^2(y',\fc_r)}-1\right)dy'\\
        &=\frac{\phi(y,\fc_r)}{b'(y_\fc)}{\Pi}^-_1(y,\fc_r)-\frac{\phi_1(y,\fc_r)}{b'(y_\fc)}+\phi(y,\fc_r){\Pi}^-_2(y,\fc_r).
    \end{aligned}
  \end{equation}
  By the L'Hospital's rule,
  \begin{align}\label{lim-pi1}
      \lim_{y\to y_{\fc}-}\frac{\phi_1(y,\fc_r)\left(b(y)-\fc_r\right)}{b'(y_\fc)}{\Pi}^-_1(y,\fc_r)&
      =\lim_{y\to y_{\fc}-}\frac{\int_{-\infty}^{y}\frac{b'(y_{\fc})-b'(y')}{\left(b(y')-\fc_r\right)^2}dy'}{\frac{1}{\left(b(y)-\fc_r\right)}}=0.
  \end{align}
  For ${\Pi}^-_2(y,\fc_r)$, we have
  \begin{align*}
      0\le -{\Pi}^-_2(y,\fc_r)&=-\int_{-\infty}^{y}\frac{1}{\left(b(y')-\fc_r\right)^2}\left(\frac{1}{\phi_1(y',\fc_r)^2}-1\right)dy' \le -{\Pi}_2.
  \end{align*}
  In Proposition \ref{prop-j1} we proved the bound \(|{\Pi}_2(\fc_r)|\le C\).
  Thus ${\Pi}^-_2(y,\fc_r)$ is bounded, which implies that
  \begin{align}\label{lim-pi2}
      \lim_{y\to y_{\fc}-}\phi(y,\fc_r){\Pi}^-_2(\fc_r)=\lim_{y\to y_{\fc}-}\phi_1(y,\fc_r)\left(b(y)-\fc_r\right){\Pi}^-_2(\fc_r)=0.
  \end{align}
  Combining \eqref{decom-varphi}, \eqref{lim-pi1} and \eqref{lim-pi2}, we have
  \begin{align*}
      \lim_{y\to y_{\fc}-}\varphi^-(y,\fc_r)=-\frac{1}{b'(y_{\fc})}.
  \end{align*}
  By similar arguments for $y>y_{\fc}$, we have
  \begin{align}\label{varphi}
      \lim_{y\to y_{\fc}+}\varphi^+(y,\fc_r)=-\frac{1}{b'(y_{\fc})},
  \end{align}
which gives $\lim_{y\to y_{\fc}+}\varphi^+(y,\fc_r)=\lim_{y\to y_{\fc}-}\varphi^-(y,\fc_r)$. Note that this is true for any $\fc_r\in \mathbb R$.

  Now we prove that $\lim_{y\to y_{\fc}+}\partial_{\fc_r}\varphi^+(y,\fc_r)=\lim_{y\to y_{\fc}-}\partial_{\fc_r}\varphi^-(y,\fc_r)$ is equivalent to
  $\mathcal{J}_1(\fc_r)=\mathcal{J}_2(\fc_r)=0$.

  Recall that  $\partial_G=\partial_{\fc_r}+\frac{\partial_y }{b'(y_{\fc})}$.
  By \eqref{decom-varphi} we obtain
  \begin{equation}\label{eq-parvar}
    \begin{aligned}
        \partial_{\fc_r}\varphi^-(y,\fc_r)
        &=\partial_{\fc_r}\left(\frac{\phi(y,\fc_r)}{b'(y_{\fc})}\right)\cdot\Pi_1^-(y,\fc_r)
        +\frac{\phi(y,\fc_r)}{b'(y_{\fc})}\cdot\partial_{\fc_r}\Pi_1^-(y,\fc_r)-\partial_{\fc_r}\frac{\phi_1(y,\fc_r)}{b'(y_{\fc})}\\
        &\quad+\partial_{\fc_r}\phi(y,\fc_r)\cdot\Pi_2^-(y,\fc_r)
         +\phi(y,\fc_r)\partial_{\fc_r}\Pi_2^-(y,\fc_r)
    \end{aligned}
  \end{equation}
  with
  \begin{equation}
    \begin{aligned}
      &\partial_{\fc_r}\Pi_1^-(y,\fc_r)=\int_{-\infty}^{y}\partial_{G'}\frac{b'(y_{\fc})-b'(y')}{\left(b(y')-\fc_r\right)^2}dy'
      -\frac{b'(y_{\fc})-b'(y)}{b'(y_{\fc})\left(b(y)-\fc_r\right)^2},\\
      &\partial_{\fc_r}\Pi_2^-(y,\fc_r)=\int_{-\infty}^{y}\partial_{G'}\frac{1}{\left(b(y')-\fc_r\right)^2}\left(\frac{1}{\phi_1(y',\fc_r)^2}-1\right)dy'
      -\frac{1}{b'(y_{\fc})\left(b(y)-\fc_r\right)^2}\left(\frac{1}{\phi_1(y,\fc_r)^2}-1\right).
    \end{aligned}
  \end{equation}
  Using  $\phi_1(y_{\fc},\fc_r)=1$, $\partial_y\phi_1(y_{\fc},\fc_r)=0$ and
  $\left|\partial_G\phi_1(y,\fc_r)\right| \lesssim \left|y-y_{\fc}\right|^3\phi_1(y,\fc_r)$ in Lemma \ref{lem-paryphi}, a direct calculation gives
  \begin{equation}\label{lim-coin}
    \begin{aligned}
        &\lim_{y\to y_{\fc}}\partial_{\fc_r}\phi(y,\fc_r)=-\lim_{y\to y_{\fc}}\phi_1(y,\fc_r)=-1, \quad
        \lim_{y\to y_{\fc}}\partial_{\fc_r}\left(\frac{\phi(y,\fc_r)}{b'(y_{\fc})}\right)
        =\lim_{y\to y_{\fc}}\frac{\partial_{\fc_r}\phi(y,\fc_r)}{b'(y_{\fc})}=-\frac{1}{b'(y_{\fc})},\\
        &\lim_{y\to y_{\fc}}\left|\partial_{\fc_r}\phi_1(y,\fc_r)\right|\le\lim_{y\to y_{\fc}}\left|\partial_{G}\phi_1(y,\fc_r)\right|+\frac{1}{b'(y_\fc)}\left|\partial_{y}\phi_1(y,\fc_r)\right|=0,\\
        &\lim_{y\to y_{\fc}}\left(-\frac{\phi(y,\fc_r)}{b'(y_{\fc})^2}\cdot\frac{b'(y_{\fc})-b'(y)}{\left(b(y)-\fc_r\right)^2}-\partial_{\fc_r}\frac{\phi_1(y,\fc_r)}{b'(y_{\fc})}
        -\frac{\phi(y,\fc_r)}{b'(y_{\fc})}\frac{1}{\left(b(y)-\fc_r\right)^2}\left(\frac{1}{\phi_1(y,\fc_r)^2}-1\right)\right)\\
        &=\frac{b''(y_{\fc})}{b'(y_{\fc})^3}+\frac{b''(y_{\fc})}{b'(y_{\fc})^3}+0=\frac{2b''(y_{\fc})}{b'(y_{\fc})^3}=\frac{2}{\pi}\mathcal{J}_2(\fc_r).
    \end{aligned}
  \end{equation}
We also have
  \begin{equation}\label{lim-p1varpi-1}
    \begin{aligned}
        &\left|\frac{\phi(y,\fc_r)}{b'(y_{\fc})}\cdot\int_{-\infty}^{y}\partial_{G'}\frac{b'(y_{\fc})-b'(y')}{\left(b(y')-\fc_r\right)^2}dy'\right|\\
        &\qquad =\left|\frac{\phi(y,\fc_r)}{b'(y_{\fc})}\cdot\int_{-\infty}^{y}
        \frac{1}{b'(y_{\fc})}\left(\frac{b''(y_{\fc})-b''(y')}{\left(b(y')-\fc_r\right)^2}-2\frac{b(y')-b(y_{\fc})}{\left(b(y')-\fc_r\right)^2}\right)dy'\right|\\
        &\qquad \lesssim \left|y-y_{\fc}\right|\int_{-\infty}^{y}\frac{1}{\left|y'-y_{\fc}\right|}dy'\to 0, \quad \text{as }y\to y_{\fc}-.
    \end{aligned}
  \end{equation}
Following the proof \(|\partial_{\fc_r}{\Pi}_2(\fc_r)|\le C\) in Proposition \ref{prop-j1}, we have
  \begin{equation}\label{lim-p1varpi-2}
    \begin{aligned}
        &\left|\phi(y,\fc_r)\int_{-\infty}^{y}\partial_{G'}\frac{1}{\left(b(y')-\fc_r\right)^2}\left(\frac{1}{\phi_1(y',\fc_r)^2}-1\right)dy'\right|\\
        &\quad \le \phi_1(y,\fc_r)\left|y-y_{\fc}\right|\int_{-\infty}^{+\infty}\left|\partial_{G'}\frac{1}{\left(b(y')-\fc_r\right)^2}\left(\frac{1}{\phi_1(y',\fc_r)^2}-1\right)\right|dy'
        \to 0, \quad \text{as }y\to y_{\fc}-.
    \end{aligned}
  \end{equation}
  Then, combining \eqref{lim-coin}, \eqref{lim-p1varpi-1} and \eqref{lim-p1varpi-2} with \eqref{eq-parvar},
  \begin{align}\label{eq-parcr+}
    &\lim_{y\to y_{\fc}-} \partial_{\fc_r}\varphi^-(y,\fc_r)=\lim_{y\to y_{\fc}-}-\frac{1}{b'(y_{\fc})}{\Pi}^-_1(y,\fc_r)-{\Pi}^-_2(y,\fc_r)+\frac{2}{\pi}\mathcal{J}_2(\fc_r).
  \end{align}
  Similarly,  we have for $y>y_{\fc}$,
  \begin{align}\label{eq-parcr-}
    &\lim_{y\to y_{\fc}+} \partial_{\fc_r}\varphi^+(y,\fc_r)=\lim_{y\to y_{\fc}+}-\frac{1}{b'(y_{\fc})}{\Pi}^+_1(y,\fc_r)-{\Pi}^+_2(y,\fc_r)+\frac{2}{\pi}\mathcal{J}_2(\fc_r).
  \end{align}
  Due to the definition of $\mathcal{J}_1(\fc_r)$, we have
  \begin{align}\label{eq-j1-pi}
    \mathcal{J}_1(\fc_r)=\lim_{\delta\to 0+}\frac{1}{b'(y_{\fc})}\left({\Pi}^-_1(y_\fc-\delta,\fc_r)-{\Pi}^+_1(y_\fc+\delta,\fc_r)\right)+\left({\Pi}^-_2(y_\fc-\delta,\fc_r)-{\Pi}^+_2(y_\fc+\delta,\fc_r)\right).
  \end{align}
  Here, \({\Pi}^\pm_1(y,\fc_r)\) may tend to infinity whereas \({\Pi}^\pm_2(y,\fc_r)\) remain bounded as \(y\to y_{\fc}\pm\).
Nevertheless, in the principal-value sense, the difference
${\Pi}^-_1(y_\fc-\delta,\fc_r)-{\Pi}^+_1(y_\fc+\delta,\fc_r) $
stays finite when \(\delta\to 0^+\). Thus
\begin{align}\label{equiv-j1eq0}
    \lim_{y\to y_{\fc}-}\partial_{\fc_r}\varphi^-(y,\fc_r)=\lim_{y\to y_{\fc}+}\partial_{\fc_r}\varphi^+(y,\fc_r)\quad \text{implies}\quad  \mathcal{J}_1(\fc_r)=0.
  \end{align}

  If $\mathcal{J}_2(\fc_r)=0$, that is, $b''(y_\fc)=0$, then $\left|b'(y_\fc)-b'(y)\right|\lesssim b'''(y_\fc)\left|y-y_\fc\right|^2$ for $\left|y-y_\fc\right|\le 1$, we have
  \begin{equation}\label{eq-j2eq0}
    \begin{aligned}
      &\lim_{y\to y_{\fc}-}\left|{\Pi}^-_1(y,\fc_r)\right| +\lim_{y\to y_{\fc}+}\left|{\Pi}^+_1(y,\fc_r)\right|\\
      &\qquad \lesssim \int_{|y-y_\fc|\le 1}\frac{\left|y-y_\fc\right|^2}{\left|y-y_\fc\right|^2}dy+\int_{|y-y_\fc|\ge 1}\frac{1}{\left|b(y)-\fc_r\right|^2}dy<+\infty, \quad \text{if}\quad \mathcal{J}_2(\fc_r)=0.
    \end{aligned}
  \end{equation}
  If $\mathcal{J}_2(\fc_r)\neq 0$, that is $b''(y_{\fc})\neq 0$, then there exists $\delta>0$ such that for $y'\in(y_{\fc}-\delta,y_{\fc})$,
  \begin{align*}
    \frac{\left|b'(y_{\fc})-b'(y')\right|}{\left(b(y')-\fc_r\right)^2}\gtrsim \frac{1}{y_\fc-y'}>0.
  \end{align*}
   Then for $y\in(y_{\fc}-\frac{1}{2}\delta,y_{\fc})$,
      \begin{align*}
          \left|\int_{y-\frac{1}{2}\delta}^{y}\frac{b'(y_{\fc})-b'(y')}{\left(b(y')-\fc_r\right)^2}dy'\right|
          \gtrsim \int_{y-\frac{1}{2}\delta}^{y} \frac{1}{y_\fc-y'}dy'=\ln \frac{y-y_{\fc}}{y_{\fc}-y+\frac{1}{2}\delta}\to +\infty, \quad \text{as}\ y\to y_{\fc}-.
      \end{align*}
  Then
  \begin{align}\label{eq-j2neq0}
      \left|\lim_{y\to y_{\fc}-}\varphi^-(y,\fc_r)\right|=+\infty, \quad \text{if}\quad \mathcal{J}_2(\fc_r)\neq 0.
  \end{align}

  Then by \eqref{eq-j2eq0} and \eqref{eq-j1-pi}, we have
  \begin{align}\label{equiv-j1j2eq0}
    \mathcal{J}_1(\fc_r)=\mathcal{J}_2(\fc_r)=0 \;\;\text{implies}\;\; \lim_{y\to y_{\fc}-}\partial_{\fc_r}\varphi^-(y,\fc_r)=\lim_{y\to y_{\fc}+}\partial_{\fc_r}\varphi^+(y,\fc_r).
  \end{align}

  On the other hand, by \eqref{eq-j2neq0}, we have
  \begin{align}\label{equiv-j2eq0}
    \lim_{y\to y_{\fc}-}\partial_{\fc_r}\varphi^-(y,\fc_r)=\lim_{y\to y_{\fc}+}\partial_{\fc_r}\varphi^+(y,\fc_r) \;\;\text{implies}\;\; \mathcal{J}_2(\fc_r)=0.
  \end{align}
  Combining \eqref{equiv-j1eq0}, \eqref{equiv-j1j2eq0} and \eqref{equiv-j2eq0}, we obtain \eqref{equiv-1}. 

  By using the same technique, we could also show that
  \begin{align*}
    \lim_{y\to y_{\fc}+}\partial_{y}\varphi^+(y,\fc_r)=\lim_{y\to y_{\fc}-}\partial_{y}\varphi^-(y,\fc_r)
   \;\;\text{if\;and\;only\;if}\;\;\mathcal{J}_1(\fc_r)=\mathcal{J}_2(\fc_r)=0.
  \end{align*}

  In the rest of the proof, we assume that $\mathcal{J}_1(\fc_r)=\mathcal{J}_2(\fc_r)=0$ and  prove  \eqref{equiv-2}.

  For $y<y_{\fc}$,  by \eqref{decom-varphi}, we have
  \begin{equation}\label{eq-paryparcvar}
    \begin{aligned}
      &\partial_y\partial_{\fc_r}\varphi^-(y,\fc_r)=\partial_y\partial_{\fc_r}\left(\frac{\phi(y,\fc_r)}{b'(y_{\fc})}\int_{-\infty}^{y}\frac{b'(y_{\fc})-b'(y')}{\left(b(y')-\fc_r\right)^2}dy'\right)-\partial_y\partial_{\fc_r}\frac{\phi_1(y,\fc_r)}{b'(y_{\fc})}\\
        &\qquad +\partial_y\partial_{\fc_r}\left(\phi(y,\fc_r)\int_{-\infty}^{y}\frac{1}{\left(b(y')-\fc_r\right)^2}\left(\frac{1}{\phi_1(y',\fc_r)^2}-1\right)dy'\right)\\
        &\quad = \partial_y\partial_{\fc_r}\frac{\phi(y,\fc_r)}{b'(y_{\fc})}\cdot\Pi_1^-(y,\fc_r)
        +\partial_y\frac{\phi(y,\fc_r)}{b'(y_{\fc})}\cdot\partial_{\fc_r}\Pi_1^-(y,\fc_r)+\partial_{\fc_r}\left(\frac{\phi(y,\fc_r)}{b'(y_{\fc})}\cdot\frac{b'(y_{\fc})-b'(y)}{\left(b(y)-\fc_r\right)^2}\right)\\
        &\qquad -\partial_y\partial_{\fc_r}\frac{\phi_1(y,\fc_r)}{b'(y_{\fc})} +\partial_y\partial_{\fc_r}\phi(y,\fc_r)\cdot\Pi_2^-(y,\fc_r)
        +\partial_y\phi(y,\fc_r)\cdot\partial_{\fc_r}\Pi_2^-(y,\fc_r)\\
        &\qquad +\partial_{\fc_r}\left(\phi(y,\fc_r)\cdot\frac{1}{\left(b(y)-\fc_r\right)^2}\left(\frac{1}{\phi_1(y,\fc_r)^2}-1\right)\right).
    \end{aligned}
  \end{equation}

 One can easily check that $\partial_y^2\phi_1(y_{\fc},\fc_r)=\frac{1}{3}$ (see Proposition 5.3 in \cite{LiMasmoudiZhao2022critical}). Using   $\phi_1(y_{\fc},\fc_r)=1$, $\partial_y\phi_1(y_{\fc},\fc_r)=0$, together with the assumptions $\mathcal{J}_2(\fc_r)=b''(y_\fc)=0$, and the estimate
  $\left|\partial_y\frac{\partial_G\phi_1(y,\fc_r)}{\phi_1(y,\fc_r)}\right| \lesssim \left|y-y_{\fc}\right|^2$ from Lemma \ref{lem-paryphi}, a direct calculation yields
\begin{align*}
      &\lim_{y\to y_{\fc}}\partial_y\partial_{\fc_r}\phi_1(y,\fc_r)=\lim_{y\to y_{\fc}}\left(\partial_y\partial_{G}-\frac{\partial_y^2}{b'(y_{\fc})}\right)\phi_1(y,\fc_r)\\
      &\qquad\qquad\qquad\qquad=\lim_{y\to y_{\fc}}\partial_y\frac{\partial_G\phi_1(y,\fc_r)}{\phi_1(y,\fc_r)}-\frac{1}{3b'(y_{\fc})}=-\frac{1}{3b'(y_{\fc})},\\
      &\lim_{y\to y_\fc}\partial_y\partial_{\fc_r}\frac{\phi(y,\fc_r)}{b'(y_{\fc})}=\lim_{y\to y_\fc}\partial_y\partial_{\fc_r}\phi(y,\fc_r)=0,\quad
      \lim_{y\to y_\fc}\partial_y\frac{\phi(y,\fc_r)}{b'(y_{\fc})}=\frac{1}{b'(y_\fc)}\lim_{y\to y_\fc}\partial_y\phi(y,\fc_r)=1,\\
  &\lim_{y\to y_\fc}\partial_{\fc_r}\left(\frac{\phi(y,\fc_r)}{b'(y_{\fc})}\cdot\frac{b'(y_{\fc})-b'(y)}{\left(b(y)-\fc_r\right)^2}\right)
        =\lim_{y\to y_\fc}\frac{\phi_1(y,\fc_r)}{b'(y_{\fc})}\cdot\partial_{\fc_r}\frac{b'(y_{\fc})-b'(y)}{b(y)-\fc_r}\\
        &\quad =\lim_{y\to y_\fc}\frac{1}{b'(y_{\fc})}\frac{b'(y_{\fc})-b'(y)}{\left(b(y)-\fc_r\right)^2}=-\frac{b'''(y_\fc)}{2b'(y_\fc)^3},\\
        &\lim_{y\to y_\fc}\partial_y\partial_{\fc_r}\frac{\phi_1(y,\fc_r)}{b'(y_{\fc})}=\frac{\partial_y\partial_{\fc_r}\phi_1(y,\fc_r)}{b'(y_\fc)}=-\frac{1}{3b'(y_\fc)^2},\\
        &\lim_{y\to y_\fc}\partial_{\fc_r}\left(\phi(y,\fc_r)\cdot\frac{1}{\left(b(y)-\fc_r\right)^2}\left(\frac{1}{\phi_1(y,\fc_r)^2}-1\right)\right)
        =\lim_{y\to y_\fc}\partial_{\fc_r}\frac{1-\phi_1(y,\fc_r)^2}{\phi_1(y,\fc_r)\left(b(y)-\fc_r\right)}\\
        &\quad =\lim_{y\to y_\fc}\frac{-2\partial_{\fc_r}\phi_1(y,\fc_r)}{b(y)-\fc_r}
        +\frac{2\left(1-\phi_1(y,\fc_r)\right)}{\left(b(y)-\fc_r\right)^2}=\frac{1}{3b'(y_\fc)^2}.
    \end{align*}  
  Then we obtain for $y<y_\fc$
  \begin{align}\label{eq-paryparcr-}
    \partial_y\partial_{\fc_r}\varphi^-(y,\fc_r)=\partial_{\fc_r}\Pi_1^-(y,\fc_r)+b'(y_\fc)\partial_{\fc_r}\Pi_2^-(y,\fc_r)+\frac{2}{3b'(y_\fc)^2}-\frac{b'''(y_\fc)}{2b'(y_\fc)^3}.
  \end{align}
  Similarly, we have for $y>y_\fc$
  \begin{align}\label{eq-paryparcr+}
    \partial_y\partial_{\fc_r}\varphi^+(y,\fc_r)=\partial_{\fc_r}\Pi_1^+(y,\fc_r)+b'(y_\fc)\partial_{\fc_r}\Pi_2^+(y,\fc_r)+\frac{2}{3b'(y_\fc)^2}-\frac{b'''(y_\fc)}{2b'(y_\fc)^3}.
  \end{align}

By the same argument as in \eqref{equiv-j1eq0}, we obtain
  \begin{align}\label{equiv-parj1eq0}
    \lim_{y\to y_{\fc}-}\partial_y\partial_{\fc_r}\varphi^-(y,\fc_r)=\lim_{y\to y_{\fc}+}\partial_y\partial_{\fc_r}\varphi^+(y,\fc_r) \;\;\text{implies}\;\; \partial_{\fc_r}\mathcal{J}_1(\fc_r)=0.
  \end{align}

  Since $b''(y_\fc)=0$, then using Taylor's expansion,
  we obtain
  \begin{align*}
    \left|\partial_{\fc_r}\frac{b'(y_\fc)-b'(y')}{\left(b(y')-\fc_r\right)^2}\right|&=
    \left|\frac{\frac{b''(y_\fc)}{b'(y_\fc)}\left(b(y')-\fc_r\right)+2\left(b'(y_\fc)-b'(y')\right)}{\left(b(y')-\fc_r\right)^3}\right|\\
    &=\frac{-2b'''(y_\fc)\left|y-y_\fc\right|^2+O\left(\left|y-y_\fc\right|^3\right)}{\left(b(y')-\fc_r\right)^3},\quad \text{as}\quad  y\to y_\fc.
  \end{align*}
  Then it follows
  \begin{equation}\label{eq-parcr-j2eq0}
    \begin{aligned}
      &\lim_{y\to y_\fc-}\left|\partial_{\fc_r}\Pi_1^-(y,\fc_r)\right|+\lim_{y\to y_\fc+}\left|\partial_{\fc_r}\Pi_1^+(y,\fc_r)\right|\\
      &\qquad \lesssim \int_{|y-y_\fc|\le 1}\frac{\left|y-y_\fc\right|^3}{\left|y-y_\fc\right|^3}dy+\int_{|y-y_\fc|\ge 1}\frac{1}{\left|b(y)-\fc_r\right|^3}dy<+\infty, \quad \text{if}\quad\partial_{\fc_r}\mathcal{J}_2(\fc_r)=0,
    \end{aligned}
  \end{equation}
  and by the same reasoning as in \eqref{eq-j2neq0},
  \begin{equation}\label{eq-parcr-j2neq0}
    \begin{aligned}
      &\left|\lim_{y\to y_\fc-}\partial_{\fc_r}\Pi_1^-(y,\fc_r)\right|=+\infty\quad \text{if}\quad \partial_{\fc_r}\mathcal{J}_2(\fc_r)\neq 0.
    \end{aligned}
  \end{equation}
Then by \eqref{eq-paryparcr-}, \eqref{eq-paryparcr+} and \eqref{eq-parcr-j2eq0} , we obtain
  \begin{align}\label{equiv-parj1parj2eq0}
    \partial_{\fc_r}\mathcal{J}_1(\fc_r)=\partial_{\fc_r}\mathcal{J}_2(\fc_r)=0 \;\;\text{implies}\;\;  \lim_{y\to y_{\fc}-}\partial_y\partial_{\fc_r}\varphi^-(y,\fc_r)=\lim_{y\to y_{\fc}+}\partial_y\partial_{\fc_r}\varphi^+(y,\fc_r).
  \end{align}
On the other hand, by \eqref{eq-parcr-j2neq0}, we have
  \begin{align}\label{equiv-parj2eq0}
    \lim_{y\to y_{\fc}-}\partial_y\partial_{\fc_r}\varphi^-(y,\fc_r)=\lim_{y\to y_{\fc}+}\partial_y\partial_{\fc_r}\varphi^+(y,\fc_r) \;\;\text{implies}\;\; \partial_{\fc_r}\mathcal{J}_2(\fc_r)=0.
  \end{align}
Combining \eqref{equiv-parj1eq0}, \eqref{equiv-parj1parj2eq0} and \eqref{equiv-parj2eq0}, we obtain \eqref{equiv-2}, which completes the proof of the proposition.
\end{proof}

Now we give a proof of Theorem \ref{thm2} utilizing the above proposition. An example of shear flow with multiple eigenvalue can be found in Appendix \ref{appendix-B}.

\begin{proof}[Proof of Theorem \ref{thm2}]
  By the assumptions of the theorem, we set
  \begin{align*}
    \mathcal{J}_1(0)=\mathcal{J}_2(0)=\partial_{\fc_r}\mathcal{J}_1(0)=\partial_{\fc_r}\mathcal{J}_2(0)=0.
  \end{align*}
  From \eqref{varphi} together with the proof of \eqref{equiv-j1j2eq0}, the conditions $\mathcal{J}_1(0)=\mathcal{J}_2(0)=0$ imply the following limits are finite and
  \begin{align}\label{parvarphi}
      \lim_{y\to 0+}\varphi^+(y,0)=\lim_{y\to 0-}\varphi^-(y,0),\quad
      \lim_{y\to 0+}\partial_y\varphi^+(y,0)=\lim_{y\to 0-}\partial_y\varphi^-(y,0).
  \end{align}
  By definition, \(\varphi^-(y,0)\) and \(\partial_y\varphi^-(y,0)\) decay exponentially as \(y\to -\infty\); the same holds for \(\varphi^+(y,0)\) and \(\partial_y\varphi^+(y,0)\) as \(y\to +\infty\).

  Moreover, because \(\mathcal{J}_2(0)=b''(y_\fc)=0\) and \(\partial_y^2\varphi^\pm(y,0)=\varphi^\pm(y,0)+\frac{b''(y)}{b(y)}\varphi^\pm(y,0)\), the second derivatives \(\partial_y^2\varphi^\pm(y,0)\) are non-singular. Then \(\varphi^-(y,0)\in H_y^2(\mathbb{R}^-)\) and \(\varphi^+(y,0)\in H_y^2(\mathbb{R}^+)\).

  Define
  \begin{align*}
  \Gamma(y,0)=\left\{\begin{aligned}
  &\varphi^-(y,0)=\phi(y,0)\int_{-\infty}^y\frac{1}{\phi^2(y',0)}dy',\quad y<0,\\
  &\varphi^+(y,0)=\phi(y,0)\int_{+\infty}^y\frac{1}{\phi^2(y',0)}dy',\quad y>0.
  \end{aligned}\right.
  \end{align*}
  Hence $\Gamma(y,0)\in H_y^2$, $\partial_y^2\Gamma(y,0)\in L^\infty_y$ and satisfies the following equation for $\fc_r=0$:
  \begin{align}\label{eq-gamma}
    \left(b(y)-\fc_r\right)\left(\frac{d^2}{dy^2}-1\right)\psi(y)-b''(y)\psi(y)=0.
  \end{align}
By Proposition \ref{prop-linear-growth}, we have
  \begin{align*}
    &\lim_{y\to 0+}\partial_{\fc_r}\varphi^+(y,\fc_r)=\lim_{y\to 0-}\partial_{\fc_r}\varphi^-(y,\fc_r),\quad
    \lim_{y\to 0+}\partial_y\partial_{\fc_r}\varphi^+(y,\fc_r)=\lim_{y\to 0-}\partial_y\partial_{\fc_r}\varphi^-(y,\fc_r),
  \end{align*}
  Then $\partial_{\fc_r}\Gamma(y,0)$ and $\partial_y\partial_{\fc_r}\Gamma(y,0)$ are finite and continuous in $y$.

  Furthermore, Lemma \ref{lem-paryphi} gives $\left|\frac{\partial_G\phi_1(y,\fc_r)}{\phi_1(y,\fc_r)^2}\right|\lesssim e^{-C_4\left|y-y_\fc\right|}$ and $\left|\partial_y\frac{\partial_G\phi_1(y,\fc_r)}{\phi_1(y,\fc_r)^2}\right|\lesssim \left|y-y_\fc\right|^2$,
   while \eqref{eq-phi1-est} yields $\left|\partial_y\phi_1(y,\fc_r)\right|\le \phi_1(y,\fc_r)$.
  Hence for $|y-y_\fc|\ge1$, we have
  \begin{align*}
    &\left|\frac{\partial_{\fc_r}\phi(y,\fc_r)}{\phi(y,\fc_r)^2}\right|=\frac{1}{\phi(y,\fc_r)^2}\left|\left(\partial_{G}-\frac{\partial_y}{b'(y_\fc)}\right)\phi_1(y,\fc_r)\cdot\left(b(y)-\fc_r\right)+\phi_1(y,\fc_r)\right| \lesssim e^{-C_4\left|y-y_\fc\right|},\\
    &\left|\frac{\partial_y\partial_{\fc_r}\phi_1(y,\fc_r)}{\phi_1(y,\fc_r)^2}\right|=\frac{1}{\phi_1(y,\fc_r)}\left|\partial_y\frac{\left(\partial_{G}-\frac{\partial_y}{b'(y_\fc)}\right)\phi_1(y,\fc_r)}{\phi_1(y,\fc_r)}+\frac{\left(\partial_{G}-\frac{\partial_y}{b'(y_\fc)}\right)\phi_1(y,\fc_r)\partial_y\phi_1(y,\fc_r)}{\phi_1(y,\fc_r)^2}\right|\\
    &\qquad\qquad\qquad\qquad \lesssim e^{-C_4\left|y-y_\fc\right|},\\
    &\left|\frac{\partial_y\partial_{\fc_r}\phi(y,\fc_r)}{\phi(y,\fc_r)^2}\right|
    =\frac{1}{\phi(y,\fc_r)^2}\left|\partial_y\partial_{\fc_r}\phi_1(y,\fc_r)\left(b(y)-\fc_r\right)+\partial_{\fc_r}\phi_1(y,\fc_r)b'(y)-\partial_y\phi_1(y,\fc_r)\right|\\
    &\qquad\qquad\qquad\qquad \lesssim e^{-C_4\left|y-y_\fc\right|}.
  \end{align*}
  Consequently, by the same argument used for $\Gamma(y,0)$, we deduce that $\partial_{\fc_r}\Gamma(y,0)\in H_y^2$ and $\partial_y^2\partial_{\fc_r}\Gamma(y,0)\in L_y^\infty$. Since \(\Gamma(y,0)\) solves \eqref{eq-gamma} for \(\fc_r=0\), the derivative \(\partial_{\fc_r}\Gamma(y,0)\) satisfies the following equation at \(\fc_r=0\):
  \begin{align}\label{eq-parcr-gamma}
    \left(b(y)-\fc_r\right)\left(\frac{d^2}{dy^2}-1\right)\psi(y)-b''(y)\psi(y)=\left(\frac{d^2}{dy^2}-1\right)\Gamma(y,0).
  \end{align} Define
  \begin{align*}
    \mathfrak w(t,y)\eqdef t\left(\frac{d^2}{dy^2}-1\right)\Gamma(y,0)+i\left(\frac{d^2}{dy^2}-1\right)\partial_{\fc_r}\Gamma(y,0).
  \end{align*}
  Then by \eqref{eq-gamma} and \eqref{eq-parcr-gamma}, we obtain
  \begin{align*}
    &\partial_t \mathfrak w(t,y)=\left(\frac{d^2}{dy^2}-1\right)\Gamma(y,0),
  \end{align*}
  and
  \begin{align*}
    &b(y)\mathfrak w(t,y)-b''(y)\left(\partial_y^2-1\right)^{-1}\mathfrak w(t,y)=t\left(\left(b(y)-\fc_r\right)\left(\frac{d^2}{dy^2}-1\right)\Gamma(y,0)-b''(y)\Gamma(y,0)\right)\\
    &\qquad +i\left(\left(b(y)-\fc_r\right)\left(\frac{d^2}{dy^2}-1\right)\partial_{\fc_r}\Gamma(y,0)-b''(y)\partial_{\fc_r}\Gamma(y,0)\right)=i\left(\frac{d^2}{dy^2}-1\right)\Gamma(y,0).
  \end{align*}
  That is, $\mathfrak w(t,y)$ satisfies mode $k=1$ linear Euler system
  \begin{align*}
    \partial_t \omega(t,y)+i\mathcal{R}\omega(t,y)=0,\quad \text{with}\quad \mathcal{R}=\mathcal{R}_{b,1}=b(y)\text{Id}-b''(y)\left(\partial_y^2-1\right)^{-1}.
  \end{align*}
  Setting $\omega_{\text{in}}(y)=i\left(\frac{d^2}{dy^2}-1\right)\partial_{\fc_r}\Gamma(y,0)$, then $\omega_{\text{in}}(y)\in L^2\cap L^\infty$. Thus, we have constructed the required solution.

\end{proof}

\subsection{Instability for the viscous problem}

After proving the inviscid case, we provide a brief proof for the viscous case.

\begin{proof}[Proof of Corollary \ref{cor}]
  Recall that
  \begin{align*}
  &\mathcal R_{b,k}=b(y) \mathrm{Id}-b''(y)(\pa_y^2-k^2)^{-1},\quad \mathcal R=\mathcal R_{b,1},\quad \mathcal O_{b,k,\nu}=ik \mathcal R_{b,k}-\nu \left(\pa_y^2-k^2\right).
\end{align*}
  For the viscous case, given $\omega_{\text{in}}(y)$, we define $\tilde{\omega}(t,y)$  by
  \begin{align*}
    \tilde{\omega}(t,y)=e^{-\mathcal{O}_{b,1,\nu}t}\omega_{\text{in}}(y)-e^{-i\mathcal{R}t}\omega_{\text{in}}(y),
  \end{align*}
  which represents the difference between the viscous solution and the inviscid one.

  We first prove that $ e^{-\mathcal{O}_{b,1,\nu}t}\omega_{\text{in}}(y)\in H^2$.
  This function is the solution to the $k=1$ mode of \eqref{eq-linNS-k}, namely:
  \begin{equation}\label{eq-linns-1}
    \begin{aligned}
      \partial_t \omega(t,y)
      +i b(y) \omega(t,y)-ib''(y)(\pa_y^2-1)^{-1}\omega(t,y)-\nu \left(\pa_y^2-1\right)\omega(t,y)=0.
    \end{aligned}
  \end{equation}
  Taking the inner product of \eqref{eq-linns-1} with $\bar{\omega}(t,y)$, integrating by parts, and adding the conjugate equation, we obtain
  \begin{align*}
    &\int_{\mathbb{R}}\partial_t \omega(t,y)\cdot\bar{\omega}(t,y)
    +\partial_t \bar{\omega}(t,y)\cdot \omega(t,y)
    +\nu\left|\partial_y\omega(t,y)\right|^2
    +\nu\left|\omega(t,y)\right|^2 dy\\
    &\quad =\int_{\mathbb{R}}ib''(y)(\pa_y^2-1)^{-1}\omega(t,y)\cdot \bar{\omega}(t,y)
    -ib''(y)(\pa_y^2-1)^{-1} \bar{\omega}(t,y)\cdot \omega(t,y)dy,
  \end{align*}
  where both sides are real-valued.  By Plancherel’s theorem, we have  $\left\|(\pa_y^2-1)^{-1}\omega\right\|_{L^2}\le \left\|\omega\right\|_{L^2}$. Consequently,
  \begin{equation}\label{est-exp1}
    \begin{aligned}
      \frac{d}{dt}\left\|\omega(t,y)\right\|_{L^2}^2
      \le 2\left\|b''(y)\right\|_{L^\infty}\left\|\omega(t,y)\right\|_{L^2}^2.
    \end{aligned}
  \end{equation}
  Furthermore, noting that $\left|b'(y)\right|$ is uniformly bounded, we apply $\partial_y$ and $\partial_y^2$ to \eqref{eq-linns-1} and repeat the above argument to obtain
  \begin{equation}\label{est-exp2}
    \begin{aligned}
      \frac{d}{dt}\left\|\partial_y\omega(t,y)\right\|_{L^2}^2
      \lesssim \left\|\omega(t,y)\right\|_{H^2}^2,\quad
      \frac{d}{dt}\left\|\partial_y^2\omega(t,y)\right\|_{L^2}^2
      \lesssim \left\|\omega(t,y)\right\|_{H^2}^2.
    \end{aligned}
  \end{equation}
  Here, we need $b''(y)\in H^3$ to ensure the boundedness of $b^{(4)}(y)$.
  Combining \eqref{est-exp1} and \eqref{est-exp2}, and using Gronwall's inequality, we conclude that for $0\le t\le T$,
  \begin{equation}\label{est-exp}
    \begin{aligned}
      \left\| e^{-\mathcal{O}_{b,1,\nu}t}\omega_{\text{in}}(y)\right\|_{H^2}\le Ce^{Ct}\left\|\omega_{\text{in}}\right\|_{H^2},
    \end{aligned}
  \end{equation}
where $C$ is independent of $\nu$. In particular, the same estimate also holds for $\mathcal R$.

The function $\tilde{\omega}(t,y)$ satisfies
  \begin{align*}
      \partial_t \tilde\omega(t,y)
      +i b(y) \tilde\omega(t,y)&-ib''(y)(\pa_y^2-1)^{-1}\tilde\omega(t,y)\\
      &\qquad -\nu \left(\pa_y^2-1\right)\tilde\omega(t,y)=\nu \left(\pa_y^2-1\right)e^{-i\mathcal{R}t}\omega_{\text{in}}(y),
  \end{align*}
   with initial condition $\tilde\omega(0,y)=0$.
By Duhamel's principle, we have
\begin{align*}
    \tilde{\omega}(t,y)=\nu\int_0^t e^{-\mathcal{O}_{b,1,\nu}(t-s)}\left(\pa_y^2-1\right)e^{-i\mathcal{R}s}\omega_{\text{in}}(y)ds.
\end{align*}
Applying estimate \eqref{est-exp} shows that for  $0\le t\le T$, the $L^2$ norm of the above integral is uniformly bounded and converges to $0$ as $\nu\to0$. Therefore, for sufficiently small $\nu$, we have
\begin{align*}
   \left\|e^{-\mathcal{O}_{b,1,\nu}t}\omega_{\text{in}}(y)\right\|_{L^2}\ge \frac{1}{2}\left\|e^{-i\mathcal{R}t}\omega_{\text{in}}(y)\right\|_{L^2}.
\end{align*}
Finally, combining this inequality with Theorem \ref{thm} completes the proof of Corollary \ref{cor}.

 \end{proof}

\begin{appendix}
\section{Existence of shear flow with a multiple embedded eigenvalue}\label{appendix-B}
In this part, we give an example of a shear flow that satisfies Assumption \ref{assum} and possesses  a multiple embedded eigenvalue.
Let
\begin{align*}
  b_{N}(y)=y+N \left(\int_{0}^y \gamma_0e^{-\frac{z^2}{\gamma^2_0}} dz -\int_{0}^y \gamma_0\gamma^2_1e^{-\frac{z^2}{\gamma^2_0\gamma^2_1}} dz\right).
\end{align*}
Here $\gamma_0,\gamma_1>0$ are small constants, and $N\approx 1$ will be determined later.

Define
\begin{align*}
  L_N=-\frac{d^2}{dy^2}+\frac{b''_{N}(y)}{b_{N}(y)},\quad
  \lambda_N =\min_{\substack{\left\|\phi\right\|_{L^2}=1\\\phi\in H^2}}<L_N\phi,\phi>
  =\min_{\substack{\left\|\phi\right\|_{L^2}=1\\\phi\in H^2}}\left(\left\|\phi'\right\|_{L^2}^2
  +\int_{\mathbb{R}}\frac{b''_{N}(y)}{b_{N}(y)}\phi^2(y)dy\right).
\end{align*}
By our setting,  $\frac{b''_{N}(y)}{b_{N}(y)}$ is a compact perturbation of $-\frac{d^2}{dy^2}$,
so it has the same essential spectrum as $-\frac{d^2}{dy^2}$ and  $\lambda_N$ is its first negative eigenvalue.

Following the proof of Lemma 2.6 in \cite{LiZhao2025}, we choose $N\approx 1$ such that $\lambda_N=1$.
Consequently, the operator \(L_N\) possesses an eigenvalue \(1\), and hence the Rayleigh operator
\[
\mathcal{R}_{b_N,1}=b_{N}(y)\,\text{Id}-b''_{N}(y)\bigl(\partial_y^2-1\bigr)^{-1}
\]
has an embedded eigenvalue $0$.
Thus $b_{N}(y)$ satisfies Assumption  \ref{assum}.

By Lemma \ref{lem-iff-emb} we have \(\mathcal{J}_1(0)=\mathcal{J}_2(0)=0\).
Moreover, the oddness of \(b_{N}\) gives \(\partial_{\fc_r}\mathcal{J}_1(0)=0\), and the condition \(b_{N,\gamma_0,\gamma_1}'''(0)=0\) implies \(\partial_{\fc_r}\mathcal{J}_2(0)=0\).
Proposition \ref{prop-linear-growth} then shows that \(0\) is a multiple embedded eigenvalue of \(\mathcal{R}_{b_N,1}\).

\end{appendix}
\section*{Acknowledgements}
H. Li is partially supported by NSF of China under Grant 12501287. S. Ren is partially supported by NSF of China under Grant 12101551 and Natural Science Foundation of Zhejiang Province under Grant LMS26A010014. Y. Wang is partially supported by NSF of China under Grant 12471200.

\bibliographystyle{siam.bst}
\bibliography{references.bib}

@book{krantz1999handbook,
  title={Handbook of complex variables},
  author={Krantz, Steven George and Kress, Steve and Kress, R},
  year={1999},
  publisher={Springer}
}

@article {LiZhao2025,
    AUTHOR = {Li, Hui and Zhao, Weiren},
     TITLE = {Viscosity driven instability of shear flows without
              boundaries},
   JOURNAL = {J. Math. Pures Appl. (9)},
  FJOURNAL = {Journal de Math\'{e}matiques Pures et Appliqu\'{e}es.
              Neuvi\`eme S\'{e}rie},
    VOLUME = {200},
      YEAR = {2025},
     PAGES = {Paper No. 103724, 28},
}

@article{reynolds1883,
  title={XXIX. An experimental investigation of the circumstances which determine whether the motion of water shall be direct or sinuous, and of the law of resistance in parallel channels},
  author={Reynolds, Osborne},
  journal={Philosophical Transactions of the Royal society of London},
  number={174},
  pages={935--982},
  year={1883},
  publisher={The Royal Society London}
}

@article{trefethen1993,
  title={Hydrodynamic stability without eigenvalues},
  author={Trefethen, Lloyd N and Trefethen, Anne E and Reddy, Satish C and Driscoll, Tobin A},
  journal={Science},
  volume={261},
  number={5121},
  pages={578--584},
  year={1993},
  publisher={American Association for the Advancement of Science}
}

@article{dikii1964,
  title={On the stability of plane-parallel Couette flow},
  author={Dikii, LA},
  journal={Journal of Applied Mathematics and Mechanics},
  volume={28},
  number={2},
  pages={479--483},
  year={1964},
  publisher={Elsevier}
}

@article{fjortoft1950application,
  title={Application of integral theorems in deriving criteria of stability for laminar flows and for the baroclinic circular vortex},
  author={Fj{\o}rtoft, Ragnar},
  journal={Geofys. Publ.},
  volume={17},
  number={6},
  pages={1--52},
  year={1950}
}

@article{sinambela2023transition,
    AUTHOR = {Sinambela, Daniel and Zhao, Weiren},
     TITLE = {The transition to instability for stable shear flows in
              inviscid fluids},
   JOURNAL = {J. Funct. Anal.},
  FJOURNAL = {Journal of Functional Analysis},
    VOLUME = {289},
      YEAR = {2025},
    NUMBER = {2},
     PAGES = {Paper No. 110905, 63},
}

@article{li2011resolution,
  title={A resolution of the Sommerfeld paradox},
  author={Li, Y Charles and Lin, Zhiwu},
  journal={SIAM journal on mathematical analysis},
  volume={43},
  number={4},
  pages={1923--1954},
  year={2011},
  publisher={SIAM}
}

@article{lin2003instability,
  title={Instability of some ideal plane flows},
  author={Lin, Zhiwu},
  journal={SIAM journal on mathematical analysis},
  volume={35},
  number={2},
  pages={318--356},
  year={2003},
  publisher={SIAM}
}

@article {LiZhao2024,
    AUTHOR = {Li, Hui and Zhao, Weiren},
     TITLE = {Asymptotic stability in the critical space of 2{D} monotone
              shear flow in the viscous fluid},
   JOURNAL = {Comm. Math. Phys.},
  FJOURNAL = {Communications in Mathematical Physics},
    VOLUME = {405},
      YEAR = {2024},
    NUMBER = {11},
     PAGES = {Paper No. 267, 50},
}

@article {Jiahao2022,
    AUTHOR = {Jia, Hao},
     TITLE = {Uniform {L}inear {I}nviscid {D}amping and {E}nhanced
              {D}issipation {N}ear {M}onotonic {S}hear {F}lows in {H}igh
              {R}eynolds {N}umber {R}egime ({I}): {T}he {W}hole {S}pace
              {C}ase},
   JOURNAL = {J. Math. Fluid Mech.},
  FJOURNAL = {Journal of Mathematical Fluid Mechanics},
    VOLUME = {25},
      YEAR = {2023},
    NUMBER = {3},
     PAGES = {42},
      ISSN = {1422-6928},
   MRCLASS = {76E05 (76D03 76D05)},
  MRNUMBER = {4589745}
}

@article{WeiZhangZhu2020cmp,
	author = {Wei, Dongyi and Zhang, Zhifei and Zhu, Hao},
	journal = {Communications in Mathematical Physics},
	number = {1},
	pages = {127--174},
	publisher = {Springer},
	title = {Linear Inviscid Damping for the $\beta$-Plane Equation},
	volume = {375},
	year = {2020}}

@article {RenZhang2025,
    AUTHOR = {Ren, Siqi and Zhang, Zhifei},
     TITLE = {Linear inviscid damping in the presence of an embedding
              eigenvalue},
   JOURNAL = {Comm. Math. Phys.},
  FJOURNAL = {Communications in Mathematical Physics},
    VOLUME = {406},
      YEAR = {2025},
    NUMBER = {2},
     PAGES = {Paper No. 39, 70},
      ISSN = {0010-3616,1432-0916},
}

@article{LiMasmoudiZhao2022critical,
    AUTHOR = {Li, Hui and Masmoudi, Nader and Zhao, Weiren},
     TITLE = {A dynamical approach to the study of instability near
              {C}ouette flow},
   JOURNAL = {Comm. Pure Appl. Math.},
  FJOURNAL = {Communications on Pure and Applied Mathematics},
    VOLUME = {77},
      YEAR = {2024},
    NUMBER = {6},
     PAGES = {2863--2946},
       }

@article{Sommerfeld1908,
	author = {Sommerfeld, A},
	journal = {Atti del IV Congresso internazionale dei matematici},
	pages = {116--124},
	title = {Ein Beitrag zur hydrodynamischen Erkl{\"a}rung der turbulenten Fl{\"u}ssigkeitsbewegung},
	year = {1908}}

@article {CZ2019,
    AUTHOR = {Coti Zelati, Michele and Zillinger, Christian},
     TITLE = {On degenerate circular and shear flows: the point vortex and
              power law circular flows},
   JOURNAL = {Comm. Partial Differential Equations},
  FJOURNAL = {Communications in Partial Differential Equations},
    VOLUME = {44},
      YEAR = {2019},
    NUMBER = {2},
     PAGES = {110--155},
      ISSN = {0360-5302,1532-4133},
}

@article{BCV2017,
	author = {Bedrossian, Jacob and Coti Zelati, Michele and Vicol, Vlad},
	fjournal = {Annals of PDE. Journal Dedicated to the Analysis of Problems from Physical Sciences},
	issn = {2524-5317},
	journal = {Ann. PDE},
	number = {1},
	pages = {Paper No. 4, 192},
	title = {Vortex axisymmetrization, inviscid damping, and vorticity depletion in the linearized 2{D} {E}uler equations},
	volume = {5},
	year = {2019}}

@article{BM2015,
	author = {Bedrossian, Jacob and Masmoudi, Nader},
	fjournal = {Publications Math\'{e}matiques. Institut de Hautes \'{E}tudes Scientifiques},
	issn = {0073-8301},
	journal = {Publ. Math. Inst. Hautes \'{E}tudes Sci.},
	pages = {195--300},
	title = {Inviscid damping and the asymptotic stability of planar shear flows in the 2{D} {E}uler equations},
	volume = {122},
	year = {2015}}

@article{Case1960,
	author = {Case, K. M.},
	fjournal = {The Physics of Fluids},
	issn = {0031-9171},
	journal = {Phys. Fluids},
	pages = {143--148},
	title = {Stability of inviscid plane {C}ouette flow},
	volume = {3},
	year = {1960}}

@article{deng2023long,
  title={Long-Time Instability of the Couette Flow in Low Gevrey Spaces},
  author={Deng, Yu and Masmoudi, Nader},
  journal={Communications on Pure and Applied Mathematics},
  volume={76},
  number={10},
  pages={2804--2887},
  year={2023},
  publisher={Wiley Online Library}
}

@article{DZ2021,
	author = {Deng, Yu and Zillinger, Christian},
	fjournal = {Archive for Rational Mechanics and Analysis},
	issn = {0003-9527},
	journal = {Arch. Ration. Mech. Anal.},
	number = {1},
	pages = {643--700},
	title = {Echo chains as a linear mechanism: norm inflation, modified exponents and asymptotics},
	volume = {242},
	year = {2021}}

@article {Chapman2002,
    AUTHOR = {Chapman, S. Jonathan},
     TITLE = {Subcritical transition in channel flows},
   JOURNAL = {J. Fluid Mech.},
  FJOURNAL = {Journal of Fluid Mechanics},
    VOLUME = {451},
      YEAR = {2002},
     PAGES = {35--97},
      ISSN = {0022-1120,1469-7645},
}

@article{GNRS2020,
	author = {Grenier, Emmanuel and Nguyen, Toan T. and Rousset, Fr\'{e}d\'{e}ric and Soffer, Avy},
	fjournal = {Journal of Functional Analysis},
	issn = {0022-1236},
	journal = {J. Funct. Anal.},
	number = {3},
	pages = {108339, 27},
	title = {Linear inviscid damping and enhanced viscous dissipation of shear flows by using the conjugate operator method},
	volume = {278},
	year = {2020}}

@article{Howard1961,
	author = {Howard, Louis N.},
	fjournal = {Journal of Fluid Mechanics},
	issn = {0022-1120},
	journal = {J. Fluid Mech.},
	pages = {509--512},
	title = {Note on a paper of {J}ohn {W}. {M}iles},
	volume = {10},
	year = {1961}}

@article{IJ2020,
    AUTHOR = {Ionescu, Alexandru D. and Jia, Hao},
     TITLE = {Non-linear inviscid damping near monotonic shear flows},
   JOURNAL = {Acta Math.},
  FJOURNAL = {Acta Mathematica},
    VOLUME = {230},
      YEAR = {2023},
    NUMBER = {2},
     PAGES = {321--399},
       }

@article{Jia2020siam,
	author = {Jia, Hao},
	fjournal = {SIAM Journal on Mathematical Analysis},
	issn = {0036-1410},
	journal = {SIAM J. Math. Anal.},
	number = {1},
	pages = {623--652},
	title = {Linear inviscid damping near monotone shear flows},
	volume = {52},
	year = {2020}}

@article{IonescuJia2020cmp,
	author = {Ionescu, Alexandru D. and Jia, Hao},
	fjournal = {Communications in Mathematical Physics},
	issn = {0010-3616},
	journal = {Comm. Math. Phys.},
	number = {3},
	pages = {2015--2096},
	title = {Inviscid damping near the {C}ouette flow in a channel},
	volume = {374},
	year = {2020}}

@article{Jia2020arma,
	author = {Jia, Hao},
	fjournal = {Archive for Rational Mechanics and Analysis},
	issn = {0003-9527},
	journal = {Arch. Ration. Mech. Anal.},
	number = {2},
	pages = {1327--1355},
	title = {Linear inviscid damping in {G}evrey spaces},
	volume = {235},
	year = {2020}}

@article{Kelvin1887,
	author = {Kelvin, Lord},
	journal = {Phil. Mag},
	number = {5},
	pages = {188--196},
	title = {Stability of fluid motion: rectilinear motion of viscous fluid between two parallel plates},
	volume = {24},
	year = {1887}}

@article{LinZeng2011,
	author = {Lin, Zhiwu and Zeng, Chongchun},
	fjournal = {Archive for Rational Mechanics and Analysis},
	issn = {0003-9527},
	journal = {Arch. Ration. Mech. Anal.},
	number = {3},
	pages = {1075--1097},
	title = {Inviscid dynamical structures near {C}ouette flow},
	volume = {200},
	year = {2011}}

@article{MasmoudiZhao2020,
    AUTHOR = {Masmoudi, Nader and Zhao, Weiren},
     TITLE = {Nonlinear inviscid damping for a class of monotone shear flows
              in a finite channel},
   JOURNAL = {Ann. of Math. (2)},
  FJOURNAL = {Annals of Mathematics. Second Series},
    VOLUME = {199},
      YEAR = {2024},
    NUMBER = {3},
     PAGES = {1093--1175},
       }

@article{Orr1907,
	author = {Orr, William M'F},
	journal = {Proc. Ir. Acad. Sect. A, Math Astron. Phys. Sci},
	number = {9},
	pages = {9--68},
	title = {The stability or instability of the steady motions of a perfect liquid and of a viscous liquid},
	volume = {27},
	year = {1907}}

@article{Rayleigh1880,
	author = {Rayleigh, Lord},
	fjournal = {Proceedings of the London Mathematical Society},
	issn = {0024-6115},
	journal = {Proc. Lond. Math. Soc.},
	pages = {57--70},
	title = {On the {S}tability, or {I}nstability, of certain {F}luid {M}otions},
	volume = {11},
	year = {1879/80}}

@article{RosSat1966,
	author = {Rosencrans, SI and Sattinger, DH},
	journal = {Journal of Mathematics and Physics},
	number = {1-4},
	pages = {289--300},
	publisher = {Wiley Online Library},
	title = {On the spectrum of an operator occurring in the theory of hydrodynamic stability},
	volume = {45},
	year = {1966}}

@article{WeiZhangZhao2018,
	author = {Wei, Dongyi and Zhang, Zhifei and Zhao, Weiren},
	fjournal = {Communications on Pure and Applied Mathematics},
	issn = {0010-3640},
	journal = {Comm. Pure Appl. Math.},
	number = {4},
	pages = {617--687},
	title = {Linear inviscid damping for a class of monotone shear flow in {S}obolev spaces},
	volume = {71},
	year = {2018}}

@article{WeiZhangZhao2019,
	author = {Wei, Dongyi and Zhang, Zhifei and Zhao, Weiren},
	fjournal = {Annals of PDE. Journal Dedicated to the Analysis of Problems from Physical Sciences},
	issn = {2524-5317},
	journal = {Ann. PDE},
	number = {1},
	pages = {Paper No. 3, 101},
	title = {Linear inviscid damping and vorticity depletion for shear flows},
	volume = {5},
	year = {2019}}

@article{WeiZhangZhao2020,
	author = {Wei, Dongyi and Zhang, Zhifei and Zhao, Weiren},
	fjournal = {Advances in Mathematics},
	issn = {0001-8708},
	journal = {Adv. Math.},
	pages = {106963, 103},
	title = {Linear inviscid damping and enhanced dissipation for the {K}olmogorov flow},
	volume = {362},
	year = {2020}}

@article{Zillinger2017,
	author = {Zillinger, Christian},
	fjournal = {Transactions of the American Mathematical Society},
	issn = {0002-9947},
	journal = {Trans. Amer. Math. Soc.},
	mrclass = {76E05 (35B35 35P25 35Q31 35Q35 76B03)},
	number = {12},
	pages = {8799--8855},
	title = {Linear inviscid damping for monotone shear flows},
	volume = {369},
	year = {2017}}

@article{Zillinger2017jde,
	author = {Zillinger, Christian},
	fjournal = {Journal of Differential Equations},
	issn = {0022-0396},
	journal = {J. Differential Equations},
	number = {11},
	pages = {7856--7899},
	title = {On circular flows: linear stability and damping},
	volume = {263},
	year = {2017}}

@article {RWWZ2023,
    AUTHOR = {Ren, Siqi and Wang, Luqi and Wei, Dongyi and Zhang, Zhifei},
     TITLE = {Linear inviscid damping and vortex axisymmetrization via the
              vector field method},
   JOURNAL = {J. Funct. Anal.},
  FJOURNAL = {Journal of Functional Analysis},
    VOLUME = {285},
      YEAR = {2023},
    NUMBER = {1},
     PAGES = {Paper No. 109919, 41},
      ISSN = {0022-1236,1096-0783},
}

@article {Ren2025,
    AUTHOR = {Ren, Siqi},
     TITLE = {Linear inviscid damping for monotonic shear flow in unbounded
              domain},
   JOURNAL = {J. Differential Equations},
  FJOURNAL = {Journal of Differential Equations},
    VOLUME = {434},
      YEAR = {2025},
     PAGES = {Paper No. 113287, 34},
}

@article {IonJia2022,
    AUTHOR = {Ionescu, Alexandru D. and Jia, Hao},
     TITLE = {Linear vortex symmetrization: the spectral density function},
   JOURNAL = {Arch. Ration. Mech. Anal.},
  FJOURNAL = {Archive for Rational Mechanics and Analysis},
    VOLUME = {246},
      YEAR = {2022},
    NUMBER = {1},
     PAGES = {61--137},
      ISSN = {0003-9527,1432-0673},
}

@article {BCJ2026,
    AUTHOR = {Beekie, Rajendra and Chen, Shan and Jia, Hao},
     TITLE = {Uniform {V}orticity {D}epletion and {I}nviscid {D}amping for
              {P}eriodic {S}hear {F}lows in the {H}igh {R}eynolds {N}umber
              {R}egime},
   JOURNAL = {Arch. Ration. Mech. Anal.},
  FJOURNAL = {Archive for Rational Mechanics and Analysis},
    VOLUME = {250},
      YEAR = {2026},
    NUMBER = {1},
     PAGES = {Paper No. 7},
}

@article{BG2024,
  title={Asymptotic behavior of solutions of the linearized Euler equations near a shear layer},
  author={Bian, Dongfen and Grenier, Emmanuel},
  journal={arXiv preprint arXiv:2403.13549},
  year={2024}
}

@article{ionescu2024stability,
  title={On the stability of shear flows in bounded channels, II: non-monotonic shear flows},
  author={Ionescu, Alexandru D and Iyer, Sameer and Jia, Hao},
  journal={Vietnam Journal of Mathematics},
  volume={52},
  number={4},
  pages={851--882},
  year={2024},
  publisher={Springer}
}

@article {Zillinger2016,
    AUTHOR = {Zillinger, Christian},
     TITLE = {Linear inviscid damping for monotone shear flows in a finite
              periodic channel, boundary effects, blow-up and critical
              {S}obolev regularity},
   JOURNAL = {Arch. Ration. Mech. Anal.},
  FJOURNAL = {Archive for Rational Mechanics and Analysis},
    VOLUME = {221},
      YEAR = {2016},
    NUMBER = {3},
     PAGES = {1449--1509},
}

@article{stepin1996rayleigh,
  title={The Rayleigh hydrodynamical problem: a theorem on eigenfunction expansion and the stability of plane-parallel flows},
  author={Stepin, Stanislav Anatol'evich},
  journal={Izvestiya: Mathematics},
  volume={60},
  number={6},
  pages={1293},
  year={1996},
  publisher={IOP Publishing}
}

\end{document}